\documentclass[EJP]{ejpecp} 

\usepackage[T1]{fontenc}

\usepackage{enumerate}
\usepackage{mathrsfs}

\SHORTTITLE{Flows of $\Lambda$-Fleming-Viot and lookdown representation} 

\TITLE{From flows of $\Lambda$-Fleming-Viot processes to lookdown processes via flows of partitions} 

\AUTHORS{%
  Cyril~Labb\'e\footnote{Universit\'e Pierre et Marie Curie (Paris 6), France.
    \EMAIL{cyr.labbe@gmail.com}}}

\KEYWORDS{Stochastic flow ; Fleming-Viot process ; Lookdown process ; Partition ; Coalescent ; Exchangeable bridge} 

\AMSSUBJ{60J25} 
\AMSSUBJSECONDARY{60G09 ; 92D25} 

\SUBMITTED{December 5, 2013} 
\ACCEPTED{June 22, 2014} 

\ARXIVID{1107.3419v3} 

\VOLUME{19}
\YEAR{2014}
\PAPERNUM{55}
\DOI{v19-3192}

\ABSTRACT{The goal of this paper is to unify the lookdown representation and the stochastic flow of bridges, which are two approaches to construct the $\Lambda$-Fleming-Viot process along with its genealogy. First we introduce the stochastic flow of partitions and show that it provides a new formulation of the lookdown representation. Second we study the asymptotic behaviour of the $\Lambda$-Fleming-Viot process and we provide sufficient conditions for the existence of an infinite sequence of Eves that generalise the primitive Eve of Bertoin and Le Gall. Finally under the condition that this infinite sequence of Eves does exist, we construct the lookdown representation pathwise from a flow of bridges.}

\newcommand{\cqfd}{\hfill \framebox{}}

\newcommand{\CDIPE}{\emph{Eves - extinction case}}
\newcommand{\BS}{\emph{Eves - persistent case}}

\newcommand{\Disc}{\textsf{\small{Discrete}}}
\newcommand{\Dust}{\textsf{\small{Infinite \footnotesize{with}\small\ dust}}}
\newcommand{\Inf}{\textsf{\small{Infinite \footnotesize{no}\small\ dust}}}
\newcommand{\Fin}{\textsf{\small{CDI}}}

\DeclareMathOperator{\Coag}{Coag}

\newcommand{\rmm}{\mathrm{m}}
\newcommand{\rme}{\mathrm{e}}

\newcommand{\tun}{\mathtt{1}}
\newcommand{\tB}{\mathtt{B}}
\newcommand{\tE}{\mathtt{E}}
\newcommand{\tO}{\mathtt{O}}
\newcommand{\tP}{\mathtt{P}}
\newcommand{\tp}{\mathtt{p}}
\newcommand{\tq}{\mathrm{q}}

\newcommand{\rF}{\mathrm{F}}

\newcommand{\rX}{\mathrm{X}}

\newcommand{\bP}{\mathbf{P}}
\newcommand{\bQ}{\mathbf{Q}}

\newcommand{\bbD}{\mathbb{D}}
\newcommand{\bbE}{\mathbb{E}}

\newcommand{\bbN}{\mathbb{N}}

\newcommand{\bbP}{\mathbb{P}}
\newcommand{\bbQ}{\mathbb{Q}}
\newcommand{\bbR}{\mathbb{R}}

\newcommand{\cA}{\mathcal{A}}
\newcommand{\cB}{\mathcal{B}}

\newcommand{\cH}{\mathcal{H}}

\newcommand{\cO}{\mathcal{O}}
\newcommand{\cP}{\mathcal{P}}

\newcommand{\cS}{\mathcal{S}}

\newcommand{\ccB}{\mathscr{B}}

\newcommand{\ccE}{\mathscr{E}}
\newcommand{\ccF}{\mathscr{F}}
\newcommand{\ccG}{\mathscr{G}}
\newcommand{\ccH}{\mathscr{H}}

\newcommand{\ccM}{\mathscr{M}}

\newcommand{\ccP}{\mathscr{P}}

\begin{document}



\section{Introduction}
The $\Lambda$-coalescent has been introduced by Pitman~\cite{Pitman99} and Sagitov~\cite{Sagitov99}. This is a Markov process $(\Pi_t,t\geq 0)$ with values in the set $\ccP_\infty$ of partitions of $\bbN:=\{1,2\ldots\}$ whose distribution is characterised by a finite measure $\Lambda$ on $[0,1]$. For any $n\geq 1$ the restriction $(\Pi_t^{[n]},t\geq 0)$ of the process to the set $\ccP_n$ of partitions of $[n]:=\{1,\ldots,n\}$ is Markov and can be described as follows. If $\Pi_t^{[n]}$ has $m$ blocks at a given time $t\geq 0$, then any $k$ of them merge at rate
\begin{equation}\label{EqLambdank}
\lambda_{m,k}:=\int_{[0,1]}x^{k}(1-x)^{m-k}x^{-2}\Lambda(dx)
\end{equation}
for every integer $k \in \{2,\ldots,m\}$. If we assume that $\Pi_0$ is the trivial partition $\tO_{[\infty]}\!:=\!\{\{1\},\{2\},\ldots\}$ then the process $(\Pi_t,t\geq 0)$ can be interpreted as the genealogy of an infinite population: each individual is represented by an integer so that the coalescence of a collection of blocks corresponds to groups of individuals finding their most recent common ancestor backward in time.\\
The $\Lambda$-coalescent is in duality (see Lemma 5 of Bertoin and Le Gall~\cite{BertoinLeGall-1}) with the so-called $\Lambda$-Fleming-Viot process which, on the contrary, describes the evolution forward in time of an infinite population. This Markov process has been introduced by Bertoin and Le Gall (see Theorem 3 in~\cite{BertoinLeGall-1}), and implicitly by Donnelly and Kurtz~\cite{DK99}. Let us provide a rigorous definition of the $\Lambda$-Fleming-Viot process. Let $n$ be an integer, let $f$ be a continuous function on $[0,1]^n$ and set
\begin{align*}
G_f(\mu) &= \int_{[0,1]^n} \mu(dx_1)\ldots\mu(dx_n)\, f(x_1,\ldots,x_n)\;,\;\;\;\;\forall \mu\in\ccM_1,\\
LG_f(\mu) &= \!\!\!\sum_{\substack{K\subset [n]\\ k:=\# K \geq 2}}\!\!\!\lambda_{n,k}\int_{[0,1]^n} \mu(dx_1)\ldots\mu(dx_n)\Big(f\big(R_K(x_1,\ldots,x_{n-k+1})\big) - f(x_1,\ldots,x_n)\Big),
\end{align*}
with $R_K(x_1,\ldots,x_{n-k+1})=(y_1,\ldots,y_n)$ where $y_i=x_{\min K}$ if $i\in K$, and the $y_i, i\notin K$, listed by increasing indices, are the values $x_1,\ldots,x_{\min K -1},x_{\min K +1},\ldots,x_{n-k+1}$.
\begin{definition}(Bertoin and Le Gall~\cite{BertoinLeGall-1})
The $\Lambda$-Fleming-Viot process is the unique $\ccM_1$-valued process $(\rho_t,t\geq 0)$ that starts from the Lebesgue measure and such that for every integer $n\geq 1$ and every continuous function $f$ on $[0,1]^n$, the process $G_f(\rho_t)-\int_0^t LG_f(\rho_s)ds$ is a martingale.
\end{definition}
\noindent In this population model, each individual possesses a genetic type, taken to be an element of $[0,1]$. The process $\rho$ describes how the frequencies of these types evolve in time. The "jumps" of this process occur at rate $\nu(dx):=x^{-2}\Lambda(dx)$ and can be interpreted as \textit{reproduction events} where a parent, uniformly chosen among the population, reproduces: a fraction $x$ of individuals dies out and is replaced by individuals with the same type as the parent. The duality with the $\Lambda$-coalescent can be thought of as follows: each reproduction event induces a coalescence event backward in time.

\smallskip

Pitman~\cite{Pitman99} proposed a construction of the $\Lambda$-coalescent from a Poisson point process, we also refer to Schweinsberg~\cite{Schweinsberg00b} for the construction of the $\Xi$-coalescent. Roughly speaking, this Poisson point process encodes the coalescence events that occur over time. Using this Poisson point process forward in time, it is possible to construct the $\Lambda$-Fleming-Viot process thanks to the so-called \textit{lookdown representation} of Donnelly and Kurtz~\cite{DK99}, we also refer to Birkner et al.~\cite{Birkner09} for a lookdown representation of the $\Xi$-Fleming-Viot process and to a recent work of Gufler~\cite{Gufler14} on a lookdown representation for tree-valued Fleming-Viot processes. We will call the underlying Poisson point process the \textit{lookdown graph} for reasons that will be made clear later on. Another - and seemingly different - approach to construct the $\Lambda$-Fleming-Viot process comes from the stochastic flow of bridges introduced by Bertoin and Le Gall~\cite{BertoinLeGall-1}. The main objective of the present paper is to unify these two constructions, this is achieved in three main steps. First we propose a new formulation of the lookdown representation that relies on the introduction of an object called the \textit{stochastic flow of partitions}. Second we investigate the asymptotic behaviour of the $\Lambda$-Fleming-Viot process, and introduce the notion of \textit{Eves} that generalise the primitive Eve of Bertoin and Le Gall~\cite{BertoinLeGall-1}. We stress that this study is new and interesting in its own right. Third, we use the Eves and the stochastic flow of partitions to define the lookdown representation pathwise from the stochastic flow of bridges.

\subsection{The lookdown representation via the flow of partitions}
In the lookdown representation of Donnelly and Kurtz~\cite{DK99}, the $\Lambda$-Fleming-Viot process is obtained as the process of empirical measures of a countable particle system $(\xi_t(i),t\geq 0),i\geq 1$. At any time $t\geq 0$, $\xi_t(i),i\geq 1$ should be seen as the \textit{types} of a countable collection of individuals alive at time $t$ within the population. This particle system is called the \textit{lookdown process}; it satisfies a consistency property that allows to construct the $n$ first particles and then to add the $(n+1)$-th particle without modifying what has already been defined. We give a brief description of the construction of the $n$ first particles. The \textit{initial types} $\xi_0(i),i\in [n]$ are i.i.d.~uniform$[0,1]$ r.v. To every subset $K\subset\{1,\ldots,n\}$ whose cardinality $k$ belongs to $\{2,\ldots,n\}$ is associated an exponential clock with parameter $\lambda_{n,k}$. When this clock rings, the types of the $n$ first particles are updated as follows. At each level in $K\backslash\{\min K\}$, an offspring of the particle currently at $\min K$ is inserted, while at the same time the levels of all the other particles that were previsouly above $\min K$ are pushed upwards (without changing their order) so that after the update each level is again occupied by exactly one particle (and in the $n$ particle system the particles beyond level $n$ are removed). In such a \textit{reproduction event}, types are inherited and the particle at $\min K$ is called the parent. The total transition rate of the $n$ first particles is equal to
\begin{equation}\label{EqLambdan}
\lambda_n := \sum_{k=2}^n \binom{n}{k} \lambda_{n,k}.
\end{equation}
The consistency of the particle system is a consequence of the fact that the $n$-th particle $(\xi_t(n),t\geq 0)$ does never give its value to any of the $(n-1)$ first particles $(\xi_t(i),t\geq 0),i\in[n-1]$ so that the infinite collection of particles can be defined consistently. This particle system admits a process of empirical measures which is a $\Lambda$-Fleming-Viot process. We refer to Subsection \ref{SubsectionLD} for precise definitions.
\begin{figure}[h!]
\begin{center}
\includegraphics[width=11cm]{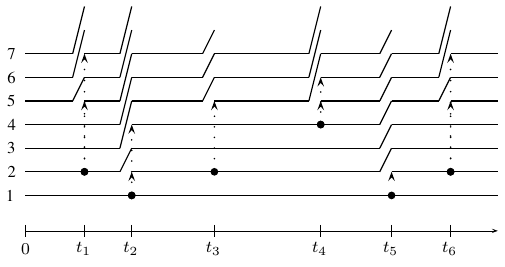}
\caption{A graphical representation of the lookdown graph. Each dot corresponds to a parent that reproduces: the levels carrying an ending arrow together with the level of the parent are now occupied by children of the parent. For example, at time $t_1$, level $2$ reproduces on levels $2$, $5$ and $7$ while former levels $5$, $6$ and $7$ are pushed up to the next available levels. The ancestor of the individual alive at time $t_3$ at level $7$ is given by $A_{t_3}(7)=4$. The corresponding flow of partitions, restricted to $[7]$, would be entirely defined by the partitions $\hat{\Pi}_{t_1-,t_1}=\{\{1\},\{2,5,7\},\{3\},\{4\},\{6\}\}$, $\hat{\Pi}_{t_2-,t_2}=\{\{1,2,4\},\{3\},\{5\},\{6\},\{7\}\}$ and so on.
}\label{Lookdown}
\end{center}
\end{figure}

There exists a graphical representation - called the \textit{lookdown graph} - of the point process of reproduction events; a formal construction (along the lines of~\cite{DK99}) of this object was provided by Pfaffelhuber and Wakolbinger~\cite{PW06} in the so-called Kingman case ($\Lambda=\delta_0(\cdot)$). There, each vertex in the graph corresponds to an individual in the population alive at a given time. We refer to Figure \ref{Lookdown} for an illustration, or to Figure 1 in Birkner et al.~\cite{Birkner09} (which includes the case of simultaneous mergers). With a slight abuse of notation, we will call lookdown graph the point process of reproduction events itself.\\
One contribution of the present paper is to show that the lookdown graph can be encoded by a collection of partitions of $\bbN:=\{1,2,\ldots\}$ that we call a stochastic flow of partitions and that gives an explicit description of the genealogy of the lookdown process. To state the definition of this object, we introduce some notation. The restriction of an element $\pi\in\ccP_\infty$ to the set $\ccP_n$ is denoted $\pi^{[n]}$. We endow $\ccP_\infty$ with the distance $d_{\mathscr{P}}$ defined as follows
\begin{equation}\label{EqDistance}
d_{\mathscr{P}}(\pi,\pi')=2^{-i} \text{ where }i=\sup\{n\geq 1: \pi^{[n]}={\pi'}^{[n]}\}.
\end{equation}
We order the blocks of a partition $\pi$ in the increasing order of their least element and we denote by $\pi(i)$ the $i$-th block according to this ordering, with the convention that $\pi(i)=\emptyset$ if $\pi$ has less than $i$ blocks. Then we define the asymptotic frequency of the $i$-th block
\[ |\pi(i)| := \lim\limits_{n\rightarrow\infty}\frac{1}{n}\#\{\pi(i)\cap[n]\} \]
when this limit exists. For any two partitions $\pi,\pi'$, we introduce the partition $\Coag(\pi,\pi')$ as follows. For every $i\geq 1$, the $i$-th block of $\Coag(\pi,\pi')$ is the union of the blocks $\pi(j)$ for $j\in\pi'(i)$.
\begin{definition}\label{DefStoFoP}
A stochastic flow of partitions is a collection $\hat{\Pi}=(\hat{\Pi}_{s,t},-\infty < s \leq t < \infty)$ of partitions of $\bbN:=\{1,2,\ldots\}$ such that:
\begin{itemize}
\item For every $r < s < t$, $\hat{\Pi}_{r,t} = \Coag(\hat{\Pi}_{s,t},\hat{\Pi}_{r,s})\ a.s.$ (cocycle property),
\item $\hat{\Pi}_{s,t}$ is an exchangeable random partition whose law only depends on $t-s$. Furthermore, for any $s_1 < s_2 < \ldots < s_n$ the partitions $\hat{\Pi}_{s_1,s_2},\hat{\Pi}_{s_2,s_3},\ldots,\hat{\Pi}_{s_{n-1},s_n}$ are independent,
\item $\hat{\Pi}_{0,0} = \tO_{[\infty]} :=\big\{\{1\},\{2\},\ldots\big\}$ and $\hat{\Pi}_{0,t} \rightarrow \tO_{[\infty]}$ in probability as $t \downarrow 0$ for the distance $d_{\mathscr{P}}$.
\end{itemize}
Furthermore if the process $(\hat{\Pi}_{-t,0},t\geq 0)$ is a $\Lambda$-coalescent then we say that $\hat{\Pi}$ is a $\Lambda$ flow of partitions.
\end{definition}
\noindent In Section \ref{SectionFoP} we show a correspondence between lookdown graphs and flows of partitions, see Proposition \ref{PropOneToOne}. This correspondence allows us to propose a new formulation of the lookdown representation that we now present. Let $\xi_0=(\xi_0(i),i\geq 1)$ be a sequence of i.i.d.~uniform$[0,1]$ r.v., referred to as the \textit{initial types}, and let $\hat{\Pi}$ be an independent $\Lambda$ flow of partitions. We define a collection of particles $(\xi_t(i),t\geq 0),i\geq 1$ by setting for every $t > 0$ and every $i,j\geq 1$
\[ \xi_t(i):=\xi_0(j) \;\;\;\text{ if }i\in\hat{\Pi}_{0,t}(j).\]
In other words, $j$ is the index of the block of $\hat{\Pi}_{0,t}$ that contains $i$, and more generally, the partition $\hat{\Pi}_{0,t}$ provides the genealogical relationships between the particles at times $0$ and $t$. Our correspondence between lookdown graphs and flows of partitions then ensures that $(\xi_t(i),t\geq 0),i\geq 1$ is a \textit{lookdown process} so that its process of empirical measures is a $\Lambda$-Fleming-Viot process. Actually we prove that its process of empirical measures coincides almost surely with the $\ccM_1$-valued process $(\ccE_{0,t}(\hat{\Pi},\xi_0),t\geq 0)$ defined by
\begin{equation}\label{EqEmpMeas}
\ccE_{0,t}(\hat{\Pi},\xi_0):= \sum_{i\geq 1}|\hat{\Pi}_{0,t}(i)|\delta_{\xi_0(i)} + \big(1-\sum_{i\geq 1}|\hat{\Pi}_{0,t}(i)|\big)\,\ell
\end{equation}
where $\ell$ stands for the Lebesgue measure on $[0,1]$, see Proposition \ref{Corollary}. Consequently $(\ccE_{0,t}(\hat{\Pi},\xi_0),t\geq 0)$ is a $\Lambda$-Fleming-Viot process. Notice that we require here the a.s.~simultaneous existence of the asymptotic frequencies of all the blocks at all times $t\geq 0$. This a.s.~simultaneous existence holds for a stochastic flow of partitions constructed from a Poisson point process, see Subsection \ref{SubsubsectionStoFoPPoisson}, but it does not necessarily hold for any stochastic flow of partitions: however we prove in Subsection \ref{SubsubsectionStoFoPRegul} that any stochastic flow of partitions admits a modification for which this a.s.~simultaneous existence holds.\\
Observe that the lookdown representation provides much information on the $\Lambda$-Fleming-Viot process. In particular, it shows that at any time $t\geq 0$ the probability measure $\ccE_{0,t}(\hat{\Pi},\xi_0)$ has an atomic component and an absolutely continuous component (with respect to the Lebesgue measure). When this absolutely continuous component is not zero, we say that the measure has \textit{dust}. This terminology comes from the theory of exchangeable random partitions of $\bbN$: if the union of the singleton blocks of an exchangeable random partition $\pi$ admits a positive asymptotic frequency, we say that $\pi$ has dust (see Section 2.1 in Bertoin~\cite{BertoinRandomFragmentation}).

\subsection{The Eves}
We shall assume that $\Lambda(\{1\})=0$ to avoid trivial reproduction events where the whole population is replaced at once by descendants of a parent. For convenience, we introduce the following notation
\begin{equation}\label{EqPsiLambda}
\Psi(u) = \Lambda(\{0\})u^2 + \int_{(0,1)}(e^{-xu}-1+xu)\nu(dx),\;\;\;\forall u\geq 0\;\;\;\;\text{ where }\nu(dx)=x^{-2}\Lambda(dx).
\end{equation}
Let $\rho=(\rho_t,t\geq 0)$ be a $\Lambda$-Fleming-Viot process assumed to be c\`adl\`ag for the weak convergence in $\ccM_1$ (this is a harmless assumption since the semigroup has the Feller property~\cite{BertoinLeGall-1}). For the moment we do not assume that $\rho$ has been obtained from the lookdown representation. The behaviour of the measure $\nu(dx)=x^{-2}\Lambda(dx)$ near $0$ provides much information on the behaviour of the $\Lambda$-Fleming-Viot process as our next result will show. To state this result, let us introduce some notation. When $\nu([0,1)) < \infty$, we say that we are in regime \Disc. When $\nu([0,1)) = \infty$ and $\int_{[0,1)}x\nu(dx) < \infty$, we say that we are in regime \Dust. The case where $\int_{[0,1)}x\nu(dx) = \infty$ and $\int^\infty \frac{du}{\Psi(u)}=\infty$ is called regime \Inf. Finally, when $\int^\infty \frac{du}{\Psi(u)}<\infty$, we say that we are in regime \Fin. This acronym comes from the theory of coalescent processes: a coalescent process starting with an infinity of blocks is said to \textit{Come Down from Infinity} if its number of blocks is finite at any positive time. These four regimes correspond to four distinct behaviours of the $\Lambda$-Fleming-Viot process:
\begin{proposition}\label{PropClassif} If $\Lambda$ belongs to regime
\begin{enumerate}
\item[\textup{1.}] \Disc, then for all $t>0$, almost surely $\rho_t$ has dust and finitely many atoms;
\item[\textup{2.}] \Dust, then for all $t>0$, almost surely $\rho_t$ has dust and infinitely many atoms;
\item[\textup{3.}] \Inf, then for all $t>0$, almost surely $\rho_t$ has no dust and infinitely many atoms;
\item[\textup{4.}] \Fin, then for all $t>0$, almost surely $\rho_t$ has no dust and finitely many atoms.
\end{enumerate}
\end{proposition}
\noindent We rely on the above classification to present our results on the asymptotic behaviour of the $\Lambda$-Fleming-Viot process. Let us recall that Bertoin and Le Gall (see Section 5.3 in~\cite{BertoinLeGall-1}) showed the existence of a r.v.~$\rme^1$ such that almost surely
\[ \rho_t(\{\rme^1\}) \underset{t\rightarrow\infty}{\longrightarrow} 1. \]
The r.v.~$\rme^1$ has a uniform distribution on $[0,1]$ and is called the primitive Eve of the population. It can be interpreted as an ancestor whose progeny is predominant among the population. In the present paper, we investigate the existence of an infinite sequence of Eves $\rme^i,i\geq 1$ that generalise the primitive Eve in the following sense: for every $i\geq 1$, the progeny of $\rme^i$ is overwhelming among the population that does not descend from $\rme^1,\ldots,\rme^{i-1}$. The precise definition is given below.
\begin{definition}\label{DefSequenceEves}In regimes \Disc, \Dust\ and \Inf, we say that $\rho$ admits an infinite sequence of Eves if there exists a collection $(\rme^{i})_{i\geq 1}$ of points in $[0,1]$ such that for every $i\geq 1$
\begin{equation*}
\frac{\rho_t(\{\rme^i\})}{\rho_t([0,1]\backslash\{\rme^1,\ldots,\rme^{i-1}\})} \underset{t\rightarrow \infty}{\longrightarrow} 1.
\end{equation*}
In regime \Fin, we say that $\rho$ admits an infinite sequence of Eves if one can order the atoms by strictly decreasing extinction times; the sequence is then denoted by $(\rme^{i})_{i\geq 1}$.
\end{definition}
\noindent In Proposition \ref{PropMeasurabilityEves} we show that the event where $\rho$ admits an infinite sequence of Eves is measurable in the sigma-field generated by $\rho$ and that, on this event, the Eves are measurable.\\
The following result establishes a connection between the Eves and the lookdown representation.
\begin{proposition}\label{PropEvesLD}
Let $\rho_t=\ccE_{0,t}(\hat{\Pi},\xi_0),t\geq 0$ where $\hat{\Pi}$ is a $\Lambda$ flow of partitions and $\xi_0$ is an independent sequence of i.i.d. uniform$[0,1]$ r.v. Then almost surely $\xi_0(1)$ coincides with the primitive Eve of Bertoin and Le Gall. Furthermore if the $\Lambda$-Fleming-Viot process admits almost surely an infinite sequence of Eves $(\rme^i,i\geq 1)$ then we have almost surely for every $i\geq 1$, $\rme^i=\xi_0(i)$.
\end{proposition}
\noindent Consequently the lookdown representation appears as the appropriate construction to investigate the existence of the Eves. We now present our main results on that question.
\begin{theorem}\label{ThCDI}
If $\Psi$ is a regularly varying function at $+\infty$ with index in $(1,2]$ then almost surely the $\Lambda$-Fleming-Viot process admits an infinite sequence of Eves.
\end{theorem}
\noindent This assumption on $\Psi$ implies that we are in regime \Fin. This is actually equivalent with the assumption that $y\mapsto\Lambda([0,y))$ is regularly varying at $0+$ with index in $[0,1)$ - we refer to Section 0.7 in the book of Bertoin~\cite{BertoinBookLevy96}. Let us observe that this assumption covers a large class of $\Lambda$-Fleming-Viot processes in this regime. One can cite in particular the measures $\Lambda$ associated with the Beta$(2-\alpha,\alpha)$-coalescent with $\alpha\in(1,2)$ (see for instance~\cite{BBS07,Article7}) and the measure $\Lambda(dx)=\delta_0(dx)$ which is associated with the celebrated Kingman coalescent~\cite{Kingman82}. Let us describe our strategy of proof. According to Definition \ref{DefSequenceEves} and Proposition \ref{PropEvesLD}, the existence of the infinite sequence of Eves in regime \Fin\ is equivalent with the non-simultaneous extinction of the initial types $\xi_0(2),\xi_0(3),\ldots$ in the lookdown representation. If we call $\tE$ the event where at least two initial types become extinct simultaneously then we have the following result:
\begin{proposition}\label{PropProbaE}
In regime \Fin, the event $\tE$ has probability $0$ or $1$.
\end{proposition}
\noindent Once this result is established, we obtain Theorem \ref{ThCDI} by showing that $\tE$ has probability zero when $\Psi$ is regularly varying. Concerning the $\Lambda$-Fleming-Viot processes in regime \Fin\ which do not fulfil this assumption, the existence of the Eves remains an open question.
\begin{conjecture}\label{ConjectureCDI}
In regime \Fin, without further condition on $\Lambda$, we have $\bbP(\tE)=0$ and thus, the $\Lambda$-Fleming-Viot process always admits an infinite sequence of Eves.
\end{conjecture}

\smallskip

We now turn to regime \Inf.
\begin{proposition}\label{PropBS}
When $\Lambda(dx)\!=\!dx$, the $\Lambda$-Fleming-Viot process admits almost surely an infinite sequence of Eves.
\end{proposition}
\noindent This result strongly relies on the connection between the Bolthausen-Sznitman coalescent~\cite{BertoinLeGall-0,BolthausenSznitman98} and the Neveu continuous state branching process, we refer to Subsection \ref{SubsectionBS}. The existence of the sequence of Eves for other measures $\Lambda$ in regime \Inf\ remains an open question.
\begin{conjecture}\label{ConjectureRegime3}
In regime \Inf, the $\Lambda$-Fleming-Viot process admits almost surely an infinite sequence of Eves without further condition on $\Lambda$.
\end{conjecture}

\smallskip

Finally we investigate the dust regimes.
\begin{theorem}\label{ThEvesDust}
Suppose that $\Lambda$ belongs to:\begin{itemize}
\item Regime \Disc\, or
\item Regime \Dust\ and fulfils the condition $\int_{[0,1)}x\log\frac{1}{x}\,\nu(dx) < \infty$,
\end{itemize}
then almost surely the $\Lambda$-Fleming-Viot process does not admit an infinite sequence of Eves.
\end{theorem}
\noindent Let us comment briefly on this result. We rely on the lookdown representation. In regime \Disc, we prove that eventually all the reproduction events choose a parent with type $\xi_0(1)$ so that only the frequency of the primitive Eve makes positive jumps. It is then easy to deduce that the second initial type $\xi_0(2)$ does not fix in the remaining population that does not descend from $\xi_0(1)$. In regime \Dust\ under the condition $\int_{[0,1)}x\log\frac{1}{x}\nu(dx) < \infty$, we prove the almost sure existence of at least one initial type, say $\xi_0(i)$ for a certain $i\geq 2$, whose frequency stays equal to $0$ forever. Consequently the $i$-th Eve does not exist. Notice that the $x\log\frac{1}{x}$ condition is fulfilled whenever $\Lambda$ is the law Beta$(2-\alpha,\alpha)$ with $\alpha\in(0,1)$ associated with the Beta$(2-\alpha,\alpha)$-coalescent, see~\cite{Article7}. Let us mention that in~\cite{DuquesneLabbe14}, it is shown that the measure-valued branching process with branching mechanism $\Psi$ (we refer to Dawson~\cite{Dawson93}, Etheridge~\cite{EtheridgeSuperprocesses} or Le Gall~\cite{LeGall99} for a definition of this object) has a residual dust component when $t$ tends to infinity if and only if $\int_{(0,1)}x\log\frac{1}{x}\, \nu(dx) < \infty$. However when this $x\log\frac{1}{x}$ condition is not fulfilled the frequencies in the population when $t$ goes to infinity are of comparable order and the measure-valued branching process does not admit an infinite sequence of Eves. As branching processes and $\Lambda$-Fleming-Viot processes present many similarities~\cite{BBS07,BertoinLeGall-3,Article7}, it is natural to expect the following behaviour:
\begin{conjecture}\label{ConjectureDust}
In regime \Dust\ when $\int_{[0,1)}x\log \frac{1}{x}\,\nu(dx)=\infty$, every initial type of the lookdown representation gets a positive frequency at a certain time. However there does not exist an infinite sequence of Eves.
\end{conjecture}

\subsection{The flow of bridges and the unification}
Let us present another approach to construct the $\Lambda$-Fleming-Viot process. An exchangeable bridge as defined by Kallenberg~\cite{Kallenberg73} is a non-decreasing, right-continuous random process $F:[0,1]\rightarrow[0,1]$ such that $F(0)=0$, $F(1)=1$, and $F$ has exchangeable increments. Bertoin and Le Gall observed that the distribution function of the $\Lambda$-Fleming-Viot process taken at a given time $t\geq 0$, say $F_t$, is an exchangeable bridge. Moreover for all $0 < s < t$, $F_t$ has the same distribution as $F_{t-s}'\circ F_{s}$ where $F_{t-s}'$ is an independent copy of $F_{t-s}$ and $\circ$ stands for the composition operator of real-valued functions. This observation motivates the following definition.
\begin{definition}(Bertoin and Le Gall~\cite{BertoinLeGall-1})\label{DefFoB}
A flow of bridges is a collection of bridges $\rF:=(\rF_{s,t},-\infty < s \leq t < \infty)$ such that :\begin{itemize}
\item For every $r < s < t$, $\rF_{r,t} = \rF_{s,t} \circ \rF_{r,s}\ a.s.$ (cocycle property),
\item The law of $\rF_{s,t}$ only depends on $t - s$. Furthermore, if $s_1 < s_2 < \ldots < s_n$\\
the bridges $\rF_{s_1,s_2},\rF_{s_2,s_3},\ldots,\rF_{s_{n-1},s_n}$ are independent,
\item $\rF_{0,0} = Id$ and $\rF_{0,t} \rightarrow Id$ in probability as $t \downarrow 0$ for the Skorohod topology.
\end{itemize}
Furthermore, if one denotes by $\rho_{s,t}$ the probability measure with distribution function $\rF_{s,t}$ and if $(\rho_{0,t},t\geq 0)$ is a $\Lambda$-Fleming-Viot process, then $\rF$ is called a $\Lambda$ flow of bridges.
\end{definition}
\noindent Observe that the stationarity of the increments of the flow of bridges ensures that the distribution of $(\rho_{s,s+t},t\geq 0)$ does not depend on $s\in\bbR$. Consequently a $\Lambda$ flow of bridges does not only construct one $\Lambda$-Fleming-Viot process, it actually couples an infinite collection of $\Lambda$-Fleming-Viot processes indexed by $s\in\bbR$.\\
We now present the main contribution of the present paper. Consider a measure $\Lambda$ such that the $\Lambda$-Fleming-Viot process admits almost surely an infinite sequence of Eves. Notice that we do not restrict ourselves to the particular cases that we have identified in Proposition \ref{PropBS} and Theorem \ref{ThCDI} and we only assume that Definition \ref{DefSequenceEves} is satisfied. Let $\rF$ be a $\Lambda$ flow of bridges. For every $s\in\bbR$, consider a c\`adl\`ag modification of the $\Lambda$-Fleming-Viot process that we still denote by $(\rho_{s,s+t},t\geq 0)$ for simplicity. We let $\rme_s:=(\rme_s^i,i\geq 1)$ be the sequence of Eves of the $\Lambda$-Fleming-Viot process $(\rho_{s,s+t},t\geq 0)$. As the cocycle property of the flow of bridges expresses a consistency of the collection of $\Lambda$-Fleming-Viot processes, it is natural to look at the relationships between the Eves taken at different times. We introduce the random partitions $\hat{\Pi}_{s,t}$ of $\bbN$ by setting for every $i\geq 1$
\[ \hat{\Pi}_{s,t}(i) := \big\{j\geq 1: \rF^{-1}_{s,t}(\rme_t^j) =\rme^i_s\big\}. \]
where $\rF^{-1}_{s,t}$ is the c\`adl\`ag inverse of $\rF_{s,t}$. This formula means that we put in a same block all the Eves at time $t$ that descend from a same Eve at time $s$.
\begin{theorem}\label{ThFoP}
Suppose that the $\Lambda$-Fleming-Viot process admits almost surely an infinite sequence of Eves. The collection of partitions $(\hat{\Pi}_{s,t},-\infty < s \leq t < \infty)$ is a $\Lambda$ flow of partitions.
\end{theorem}
\noindent In Corollary \ref{CorIndep}, we will also show that the Eves $\rme_s$ at time $s$ are independent from $(\hat{\Pi}_{s+r,s+t},0\leq r \leq t)$.

\smallskip

From the $\Lambda$ flow of partitions $(\hat{\Pi}_{s,t},-\infty < s \leq t < \infty)$ we can define a collection of measure-valued processes using our new formulation of the lookdown representation. Before we proceed to this final part of the introduction, we need to clarify some regularity issues. In Subsection \ref{SubsubsectionStoFoPRegul}, we show the existence of a modification of the flow of partitions such that a.s.~all the partitions of the modified flow admit asymptotic frequencies and the cocycle property holds simultaneously for all triplet of times $r<s<t$. From now on, we will implicitly deal with this modified flow of partitions and we will keep the notation $\hat{\Pi}$ for simplicity. Let us mention that this regularisation does not seem possible for a stochastic flow of bridges. Indeed, a key argument in our proof relies on the continuity of the coagulation operator whereas this property does not hold with the composition operator for bridges.\\
For any given $s\in\bbR$, we introduce the $\ccM_1$-valued process $(\ccE_{s,s+t}(\hat{\Pi},\rme_s),t\geq 0)$ by setting
\[ \ccE_{s,s+t}(\hat{\Pi},\rme_s):= \sum_{i\geq 1}|\hat{\Pi}_{s,s+t}(i)|\delta_{\rme_s^i} + (1-\sum_{i\geq 1}|\hat{\Pi}_{s,s+t}(i)|)\,\ell,\;\;\;\;\forall t\geq 0.\]
Recall that $\ell$ denotes the Lebesgue measure on $[0,1]$. This leads us to the statement of our main result.
\begin{theorem}\label{ThLDFoB}
Suppose that the $\Lambda$-Fleming-Viot process admits almost surely an infinite sequence of Eves. The $\Lambda$ flow of bridges is uniquely decomposed into two random objects: the flow of partitions and the Eves process. More precisely, for all $s \in \mathbb{R}$ almost surely
\begin{equation*}
(\mathscr{E}_{s,s+t}(\hat{\Pi},\rme_s),t\geq 0) = (\rho_{s,s+t},t\geq 0).
\end{equation*}
Furthermore, suppose that we are given a $\Lambda$ flow of partitions $\hat{\Pi}'$ and, for each $s \in \mathbb{R}$, a sequence $\chi_s:=(\chi_{s}(i),i\geq 1)$ of r.v.~taking almost surely distinct values in $[0,1]$. If for every $s\in\bbR$ almost surely $(\mathscr{E}_{s,s+t}(\hat{\Pi}',\chi_{s}),t\geq 0) = (\rho_{s,s+t},t \geq 0)$ then
\begin{enumerate}[(i)]
\item For all $s \in \mathbb{R}$ and all $i\geq 1$, almost surely $\chi_{s}(i) = \rme^{i}_s$,
\item Almost surely $\hat{\Pi}' = \hat{\Pi}$.
\end{enumerate}
\end{theorem}
\noindent This theorem shows that the $\Lambda$ flow of partitions $\hat{\Pi}$ and the process of Eves $(\rme_s,s\in\bbR)$ are sufficient to recover the whole collection of $\Lambda$-Fleming-Viot processes encoded by the flow of bridges. Furthermore these two ingredients are unique. Notice that our construction actually defines a \textit{flow} indexed by $s\in\bbR$ of lookdown processes.
\paragraph{Organisation of the paper.}In Section \ref{SectionFoP} we present our new formulation of the lookdown representation based on flows of partitions. In Section \ref{SectionEves} we study general properties of the Eves and we deal with the measurability issues. In Section \ref{SectionExistence} we prove Theorem \ref{ThCDI}, Proposition \ref{PropBS} and Theorem \ref{ThEvesDust}. In Section \ref{SectionUnification} we prove Theorems \ref{ThFoP} and \ref{ThLDFoB}. Some technical proofs are postponed to Section \ref{SectionAppendix}.

\section{Flows of partitions}\label{SectionFoP}
We start this section with the original definition of the lookdown representation. Then we introduce our formalism based on partitions of integers and show a one-to-one correspondence between lookdown graphs and flows of partitions. Let us mention that stochastic flows of partitions have been independently introduced by Foucart~\cite{Foucart12} to define generalised Fleming-Viot processes with immigration.

\subsection{The lookdown representation}\label{SubsectionLD}
In the lookdown representation, the population is composed of a countable collection of individuals: at any time each individual is located at a so-called \textit{level}, taken to be an element of $\bbN$. This induces at any time an ordering of the population. At time $0$ the individual located at level $i\geq 1$ is called the $i$-th ancestor, his type is denoted by $\xi_0(i) \in [0,1]$. As time passes, \textit{reproduction events} occur in which an individual, called the parent, gives birth to a collection of children and dies out. To give a precise definition we need to introduce some notation. We define $\cS_\infty$ as the subset of $\{0,1\}^\bbN$ whose elements have at least two coordinates equal to $1$. For every $n\geq 2$ we also introduce $\cS_n$ as the subset of $\{0,1\}^\bbN$ whose elements have at least two coordinates equal to $1$ among $[n]$. For a subset $\tp$ of $\bbR\times\cS_\infty$ and an integer $n\geq 2$, let $\tp_{|[s,t]\times\cS_n}$ be the intersection of $\tp$ with $[s,t]\times\cS_n$.
\begin{definition}
A deterministic lookdown graph is a countable subset $\tp$ of $\mathbb{R}\times{\cal S}_{\infty}$ such that for all $n \in \mathbb{N}$ and all $s \leq t$ the set $\tp_{|[s,t]\times{\cal S}_n}$ has a finite number of points. To every point $(t,v) \in \tp$ we associate the set $I_{t,v} := \{i \geq 1: v(i) = 1\}$.
\end{definition}
\noindent A deterministic lookdown graph should be seen as a collection of \textit{reproduction events} $(t,v) \in \mathbb{R}\times{\cal S}_{\infty}$, where $t$ denotes the reproduction time, $\min I_{t,v}$ is the level of the parent and $I_{t,v}$ is the set of levels that participate to the reproduction event.
\begin{definition}\label{DefDetLDProcess}
Let $\tp$ be a lookdown graph and $\xi_0=(\xi_0(i),i\geq 1)$ be a collection of values in $[0,1]$. The deterministic lookdown process constructed from $\tp$ and $\xi_0$ is the particle system $(\xi_t(i),t\geq 0),i\geq 1$ defined as follows:
\begin{itemize}
\item The initial types are given by $(\xi_0(i),i\geq 1)$,
\item At any reproduction event $(t,v) \in \tp$ with $t > 0$ we have
\begin{equation}\label{EquationLookdown}
\begin{cases}
\xi_{t}(i) = \xi_{t-}(\min(I_{t,v}))&\mbox{ \ \ for every }i \in I_{t,v},\\
\xi_{t}(i) = \xi_{t-}(i-(\#\{I_{t,v}\cap[i]\}-1)\vee 0)&\mbox{ \ \ for every }i \notin I_{t,v}.
\end{cases}
\end{equation}
\end{itemize}
\end{definition}
\noindent The transitions should be interpreted as follows. The parent of the reproduction event is located at level $\min(I_{t,v})$, this individual dies at time $t$. At any level $i\in I_{t,v}$ a new individual is born with the type $\xi_{t-}(\min(I_{t,v}))$ of the parent. All the individuals alive at time $t-$ except the parent are then redistributed, keeping the same order, on those levels that do not belong to $I_{t,v}$ (see Figure \ref{Lookdown}). When $v$ does not belong to $\cS_n$, then $\xi_t(i)=\xi_{t-}(i)$ for every $i\in[n]$. The finiteness of the set $\tp_{|[s,t]\times{\cal S}_n}$ for all $s<t$ thus ensures that the $n$ first particles $(\xi_t(i),t\geq 0),i\in[n]$ make only finitely many jumps on any compact interval of time so that they are well-defined processes. Observe the consistency of the particle system when $n$ varies: the $(n+1)$-th particle does not affect the $n$ first.\vspace{4pt}\\

We now explain how one randomises the previous objects so that the lookdown process admits almost surely a process of empirical measures that forms a $\Lambda$-Fleming-Viot process. Let $(\Omega,\ccF,\bbP)$ be a probability space and let ${\cal P}$ be a Poisson point process on $\mathbb{R}\times{\cal S}_{\infty}$ with intensity measure $dt\otimes(\mu_{K}+\mu_{\Lambda})$. The measures $\mu_{K}$ and $\mu_{\Lambda}$ are defined by
\begin{equation}\label{EqMeasureMu}
\mu_{\Lambda}(.) := \int_{(0,1)}\sigma_{x}(.)\,\nu(dx) \;\;,\;\;\mu_{K}(.) := \Lambda(0)\sum_{1\leq i < j}\delta_{s_{i,j}}(.)
\end{equation}
where for all $x\in(0,1)$, $\sigma_{x}(.)$ is the distribution on ${\cal S}_{\infty}$ of a sequence of i.i.d. Bernoulli random variables with parameter $x$, and for all $1\! \leq \!i \!<\! j$, $s_{i,j}$ is the element of ${\cal S}_{\infty}$ that has only two coordinates equal to $1$: $i$ and $j$. The measure $\mu_{\Lambda}$ corresponds to reproduction events involving a positive proportion of individuals while $\mu_K$ corresponds to reproduction events involving only two individuals at once.
\begin{definition}\label{DefLambdaLDGraph}
A lookdown graph associated with the measure $\Lambda$ - or $\Lambda$ lookdown graph in short - is a Poisson point process ${\cal P}$ on $\mathbb{R}\times{\cal S}_{\infty}$ with intensity measure $dt\otimes(\mu_{K}+\mu_{\Lambda})$.
\end{definition}
\noindent Using Formula (\ref{EqLambdan}) we easily get for any $n\geq 2$
\[(\mu_{K}+\mu_{\Lambda})\big(\cS_n\big) = \Lambda(0)\binom{n}{2} + \int_{(0,1)}\sum_{k=2}^{n}\binom{n}{k} x^k(1-x)^{n-k}\,\nu(dx)=\lambda_n.\]
Basic properties of Poisson point processes ensure that $\cP_{|[s,t]\times\cS_n}$ is a Poisson point process on $[s,t]\times\cS_n$ whose intensity has a total mass equal to $(t-s)\lambda_n < \infty$. Consequently for $\bbP$-almost all $\omega$, $\cP(\omega)$ is a deterministic lookdown graph. Let $\xi_0=(\xi_0(i),i\geq 1)$ be an independent sequence of i.i.d. uniform$[0,1]$ r.v. 
\begin{definition}\label{DefLDProcess}
Applying Definition \ref{DefDetLDProcess} to $\cP(\omega)$ and $\xi_0(\omega)$ for $\bbP$-almost all $\omega\in\Omega$, we get a particle system $(\xi_t(i),t\geq 0),i\geq 1$ that we call the lookdown process associated to $\cP$ and $\xi_0$.
\end{definition}
\noindent We let $\ccM_1$ be the set of probability measures on $[0,1]$ endowed with the topology of the weak convergence.
\begin{theorem}\label{ThDK}(Donnelly-Kurtz~\cite{DK99}, Birkner et al.~\cite{Birkner09}) $\bbP$-almost surely the particle system $(\xi_t(i),t\geq 0),i\geq 1$ admits a process of empirical measures
\[ t\mapsto \Xi_t:= \lim\limits_{n\rightarrow\infty}\frac{1}{n}\sum_{i=1}^n\delta_{\xi_t(i)}. \]
This $\ccM_1$-valued process is a c\`adl\`ag $\Lambda$-Fleming-Viot process.
\end{theorem}
\noindent For convenience we give an outline of the proof.
\begin{proof}
Let us introduce for every $n\geq 1$, the process of empirical measures of the $n$ first particles:
\[ t\mapsto\Xi_t^{[n]}:= \frac{1}{n}\sum_{i=1}^n\delta_{\xi_t(i)}.\]
This process takes values in the space $\bbD([0,\infty),\ccM_1)$ of c\`adl\`ag $\ccM_1$-valued functions. Let $f_k,k\geq 1$ be a dense sequence of continuous functions on $[0,1]$. Lemma 3.4 and 3.5 of Donnelly and Kurtz~\cite{DK99} show that for every $k\geq 1$, $\bbP$-a.s.~the sequence of processes $(\int_{[0,1]}f_k(x)\Xi_t^{[n]}(dx),t\geq 0),n\geq 1$ is a Cauchy sequence in $\bbD([0,\infty),\bbR)$ \textit{endowed} with the distance:
\[ d_u(X,Y) = \int_{[0,\infty)}\!\!\!e^{-t}\sup_{s\leq t}1\wedge|X_s-Y_s|\,dt .\]
Subsequently, using the distance
\[ d(m,m') := \sum_{k\geq 1}2^{-k}\bigg(\Big|\int_{[0,1]}f_k(x)m(dx)-\int_{[0,1]}f_k(x)m'(dx)\Big|\wedge 1\bigg) \]
that metrises the topology of the weak convergence on $\ccM_1$, it is a simple matter to check that almost surely $(\Xi_t^{[n]},t\geq 0),n\geq 1$ is a Cauchy sequence in $\bbD([0,\infty),\ccM_1)$ endowed with the distance:
\[ \tilde{d}_u(M,M') = \int_{[0,\infty)}\!\!\!e^{-t}\sup_{s\leq t} d(M_s,M'_s)\,dt .\]
The sequence converges $\bbP$-a.s.~since the latter is a complete space. The identification of the distribution of the limiting process can be carried out by comparing (modulo a simple calculation) the expression of the generator obtained by Donnelly and Kurtz in Section 4~\cite{DK99} with the expression of the generator of the $\Lambda$-Fleming-Viot process obtained by Bertoin and Le Gall in Theorem 3~\cite{BertoinLeGall-1}.
\end{proof}

\subsection{Deterministic flows of partitions}
We start with the definition of the deterministic flow of partitions. Then we prove a one-to-one correspondence between the set of deterministic flows of partitions and the set of deterministic lookdown graphs. Finally we show that the deterministic flow of partitions formalises the implicit genealogy encoded by a deterministic lookdown graph. A key r\^ole will be played by the coagulation operator. Recall that for any $\pi,\pi' \in \ccP_\infty$, $\Coag(\pi,\pi')$ is the element of $\ccP_\infty$ whose $i$-th block is equal to the union of the $\pi(j)$'s for $j\in\pi'(i)$. We extend this notation to the case where $\pi,\pi'\in\ccP_n$ by taking implicitly $\Coag(\pi,\pi')$ as an element of $\ccP_n$ and using the same definition for its blocks.
\begin{definition}\label{DefDetFoP}
A deterministic flow of partitions is a collection $\hat{\pi}=(\hat{\pi}_{s,t},-\infty < s \leq t < \infty)$ of partitions of $\bbN$ that satisfies:\begin{itemize}
\item For all $r < s < t$, we have $\hat{\pi}_{r,t} = \Coag(\hat{\pi}_{s,t},\hat{\pi}_{r,s})$ (cocycle property).
\item For all $s\in\bbR$, $\hat{\pi}_{s,s}=\tO_{[\infty]}$ and the limit $\lim\limits_{r\uparrow s}\hat{\pi}_{r,s}=:\hat{\pi}_{s-,s}$ exists and is either a partition with a unique non-singleton block or $\tO_{[\infty]}$.
\item For all $t\in\bbR$, $\lim\limits_{\substack{r<s<t\\r,s \,\uparrow\, t}}\hat{\pi}_{r,s}$ exists and equals $\tO_{[\infty]}$ (left regularity).
\item For all $s\in\bbR$, $\lim\limits_{t\downarrow s}\hat{\pi}_{s,t}$ exists and equals $\tO_{[\infty]}$ (right regularity).
\end{itemize}
\end{definition}
\begin{remark}
The left regularity is not a consequence of the three other properties. Denote by $\tun_{[\infty]}$ the partition with a unique block containing all the integers. Take $\hat{\pi}_{s,t}=\tun_{[\infty]}$ if $(s,t]\cap\{1-2^{-n},n\geq 0\} \ne \emptyset$ and $\hat{\pi}_{s,t}=\tO_{[\infty]}$ otherwise. This collection of partitions satisfies all the properties except the left regularity: indeed, $\hat{\pi}_{r,s}$ ``oscillates'' between $\tO_{[\infty]}$ and $\tun_{[\infty]}$ as $r,s \uparrow 1$.
\end{remark}
\noindent It is natural to call $\{(s,\hat{\pi}_{s-,s}): \hat{\pi}_{s-,s}\ne \tO_{[\infty]}\}$ the collection of the jumps of $\hat{\pi}$. For convenience, we write this collection of jumps in a slightly different manner. To every partition $\hat{\pi}_{s-,s}$ that differs from $\tO_{[\infty]}$ we associate the element $v_s$ of $\cS_{\infty}$ whose $i$-th coordinate is $1$ if and only if $i$ belongs to the non-singleton block of $\hat{\pi}_{s-,s}$. We then set $J(\hat{\pi}) := \{(s,v_s): \hat{\pi}_{s-,s}\ne \tO_{[\infty]}\}$ for the collection of jumps of $\hat{\pi}$. Similarly for every $n\geq 2$, we set $J_n(\hat{\pi}) := \{(s,v_s): \hat{\pi}_{s-,s}^{[n]}\ne \tO_{[n]}\}$. Notice that $J(\hat{\pi})$ and $J_n(\hat{\pi})$ are subsets of $\bbR\times\cS_\infty$ and $\bbR\times\cS_n$ respectively. 
\begin{lemma}\label{LemmaJpi}
For every $n\geq 2$, the intersection of $J_n(\hat{\pi})$ with any set of the form $[s,t]\times\cS_n$ has finitely many points.
\end{lemma}
\begin{proof}
Suppose that the intersection of $J_n(\hat{\pi})$ with the set $[s,t]\times\cS_n$ has infinitely many points. Necessarily these points accumulate near a certain point in $[s,t]$. For simplicity we assume that they accumulate on the right of $r \in [s,t)$ and we let $(r_1,v_1),(r_2,v_2),\ldots$ be a sequence of these points such that $r_i\downarrow r$ as $i\rightarrow\infty$. By Definition \ref{DefDetFoP}, as $u\uparrow r_i$ the partition $\hat{\pi}_{u,r_i}$ converges to $\hat{\pi}_{r_i-,r_i}$. Since $v_{r_i}\in\cS_n$, $\hat{\pi}_{r_i-,r_i}\ne \tO_{[n]}$ so that there exists $u\in(r_{i+1},r_i)$ such that $\hat{\pi}_{u,r_i}^{[n]}\ne \tO_{[n]}$. Therefore the cocycle property yields $\hat{\pi}_{r_{i+1},r_i}^{[n]}=\Coag(\hat{\pi}_{u,r_i}^{[n]},\hat{\pi}_{r_{i+1},u}^{[n]})$ so that $\hat{\pi}_{r_{i+1},r_i}^{[n]}\ne \tO_{[n]}$. Using the cocycle property once again we get that $\hat{\pi}_{r,r_i}^{[n]}\ne \tO_{[n]}$. Taking the limit as $i\rightarrow\infty$ we deduce that the right regularity at $r$ is not satisfied. If we had assumed that the points accumulate on the left of a given point $r$ then the left regularity at $r$ would have failed.
\end{proof}
We now make the connection between deterministic flows of partitions and deterministic lookdown graphs.
\begin{proposition}\label{PropOneToOne}
The map $J$ that associates to a deterministic flow of partitions $\hat{\pi}$ the collection of its jumps $J(\hat{\pi})$ is a bijection between the set of deterministic flows of partitions and the set of deterministic lookdown graphs.
\end{proposition}
\begin{proof}
Let $\hat{\pi}$ be a deterministic flow of partitions. One can write
\[ J(\hat{\pi}) = \underset{n\geq 2}{\bigcup}J_n(\hat{\pi}).\]
Hence the set on the left is a countable union of countable sets thanks to Lemma \ref{LemmaJpi}, so that it is itself countable. Moreover Lemma \ref{LemmaJpi} shows that the intersection of $J(\hat{\pi})$ with any set of the form $[s,t]\times\cS_n$ has only finitely many points. Consequently $J(\hat{\pi})$ is a deterministic lookdown graph. Let us show that $J$ is injective. Suppose that $J_n(\hat{\pi})\cap\big((s,t]\times\cS_n\big)$ has $k$ points, say $(r_1,v_1),\ldots,(r_k,v_k)$ in the increasing order of the time coordinates. It is easy to check that $\hat{\pi}^{[n]}_{r_{i-1},r_i}$ is the element of $\ccP_n$ with a unique non-singleton block given by $\{j\in[n]: v_i(j)=1\}$ while $\hat{\pi}^{[n]}_{r_{i-1},r_i-}=\tO_{[n]}$. The cocyle property then ensures that the partitions $\hat{\pi}^{[n]}$ are completely determined by the knowledge of $J_n(\pi)$. The injectivity follows.

Let us now turn to surjectivity. Let $\tp$ be a deterministic lookdown graph. For every $n\geq 1$, we construct the restriction of $\hat{\pi}$ to $\ccP_n$ as follows. Fix $s \leq t$. If $s=t$, set $\hat{\pi}^{[n]}_{s,t} := \tO_{[n]}$. If $s < t$, let $(t_1,v_1),\ldots,(t_k,v_k)$ be the finite collection of points in $\tp_{|(s,t]\times\cS_n}$ ranked by increasing time coordinates. We define the map $\phi_n:\cS_n\rightarrow\ccP_n$ as follows. For every $v\in\cS_n$, $\phi_n(v)$ is the element of $\ccP_n$ with a unique non-singleton block equal to $\{i\in[n]: v(i)=1\}$. We then set
\[ \hat{\pi}_{s,t}^{[n]} := \Coag\Big(\phi_n(v_k),\Coag\big(\phi_n(v_{k-1}),\ldots,\Coag(\phi_n(v_2),\phi_n(v_1))\ldots\big)\Big).\]
Our construction is consistent when $n$ varies so that there exists a unique partition $\hat{\pi}_{s,t}$ whose restriction to $[n]$ is given by $\hat{\pi}_{s,t}^{[n]}$, for all $s \leq t$. Let us now check that $\hat{\pi}$ is indeed a deterministic flow of partitions. The cocycle property is a direct consequence of our definition with the coagulation operator. The finiteness of $\tp_{|(s,t]\times\cS_n}$ for every $s\leq t$ and every $n\geq 1$ ensures the existence of $\hat{\pi}_{s-,s}$ for all $s\in\bbR$ along with the left and right regularity. The fact that $\hat{\pi}_{s-,s}$ is either $\tO_{[\infty]}$ or a partition with a unique non-singleton block is a consequence of the fact that we have only coagulated partitions with a unique non-singleton block. From our construction it is immediate that $J(\hat{\pi})=\tp$. The map $J$ is therefore surjective.
\end{proof}
Let us now explain why the concept of deterministic flow of partitions is relevant in our context. Let $\tp$ be a deterministic lookdown graph. For any level $j\geq 1$ taken at a given time $t\geq 0$, one can define its \textit{ancestral line} by tracing its genealogy backward in time until its ancestor at time $0$. This ancestral line starts at $j$ and stays constant between two reproduction events. At a reproduction event, say $(s,v)$ with $s\in(0,t]$, either the ancestral line belongs to $I_{s,v}$ and then it jumps to $\min I_{s,v}$, or the ancestral line does not belong to $I_{s,v}$. In the latter case, the ancestral line jumps from its current position, say $i$, to $i-(\#\{I_{s,v}\cap[i]\}-1)\vee 0$. The value of the ancestral line at $0$ is then the ancestor of level $j$ at time $t$, we denote it by $A_t(j)$. We refer to Figure \ref{Lookdown} for an illustration of these ancestral lines.\\
Denote by $J^{-1}$ the inverse map of $J$ and set $\hat{\pi}:=J^{-1}(\tp)$. The following lemma shows that $\hat{\pi}$ is the genealogical structure implicitly encoded by $\tp$.
\begin{lemma}\label{LemmaGenealogyPi}
For all $i\geq 1$ and all $t\geq 0$, $\hat{\pi}_{0,t}(i) = \{j\geq 1: A_t(j)=i\}$.
\end{lemma}
\noindent In other words, $\hat{\pi}_{0,t}(i)$ is the progeny at time $t$ of the $i$-th ancestor. This allows to state an alternative definition of the lookdown process. Consider a sequence of initial types $\xi_0(i),i\geq 1$. Then the lookdown process of Definition \ref{DefDetLDProcess} satisfies and is characterised by
\[ \forall i,j\geq 1,\;\; j\in\hat{\pi}_{0,t}(i) \Rightarrow \xi_t(j) = \xi_0(i).\]
\begin{proof}
Fix $n\geq 1$. Let $(s_1,v_1),(s_2,v_2),\ldots$ be the elements of $\tp_{|(0,\infty)\times\cS_n}$ ranked by increasing time coordinates. From the definition of the map $J$, it is immediate to check that $t\mapsto\hat{\pi}_{0,t}^{[n]}$ only evolves at times $s_k,k\geq 1$. Set $s_0=0$. We prove by induction on $k$ that for every $i\geq 1$, $\hat{\pi}_{0,s_k}^{[n]}(i) = \{j\in [n]: A_{s_k}(j)=i\}$. At rank $0$ this is trivial. Suppose that it holds at a certain rank $k-1\geq 0$. At rank $k$, we set $I_{s_k,v_k}^{[n]}:=\{j\in[n]: v_k(j) = 1\}$ and for every $j\in[n]$
\[ b(j) := \min I_{s_k,v_k}^{[n]} \;\;\text{ if } j\in I_{s_k,v_k}^{[n]}\;\;\;\;\;,\;\;\;\;\;b(j):=j-(\#\{I_{s_k,v_k}^{[n]}\cap[j]\}-1)\vee 0 \;\;\text{ otherwise}.\]
Then the definition of the ancestral line immediately yields that $A_{s_k}(j) = A_{s_{k-1}}(b(j))$. Let $\pi$ be the partition of $[n]$ whose blocks are given by $\{j\in[n]: A_{s_k}(j)=i\}, i\in[n]$. From the formula above we check that $j$ and $j'$ are in a same block of $\pi$ if and only if $A_{s_{k-1}}(b(j))=A_{s_{k-1}}(b(j'))$ which is equivalent to saying (thanks to the induction hypothesis) that $b(j)$ and $b(j')$ are in a same block of $\hat{\pi}_{0,s_{k-1}}^{[n]}$. Since $\hat{\pi}_{s_{k-1},s_k}^{[n]}=\hat{\pi}_{s_k-,s_k}^{[n]}$, it is elementary to check that $b(j)$ is the index of the block of $\hat{\pi}_{s_{k-1},s_k}^{[n]}$ containing $j$. Since $\hat{\pi}_{0,s_k}^{[n]}=\Coag(\hat{\pi}_{s_{k-1},s_k}^{[n]},\hat{\pi}_{0,s_{k-1}}^{[n]})$ we deduce that $\hat{\pi}_{0,s_k}^{[n]} = \pi$ and the induction is complete. We have proved the asserted equality for the restrictions to $[n]$ for any given $n\geq 1$. The lemma follows.
\end{proof}

\subsection{Stochastic flows of partitions}\label{SubsectionStoFoP}
We now consider the stochastic flow of partitions of Definition \ref{DefStoFoP}. We first propose a Poissonian construction of this object and show that almost surely its trajectories are deterministic flows of partitions. Consequently almost all the trajectories encode a deterministic lookdown graph and can be used to apply the lookdown construction. Second we consider a stochastic flow of partitions not necessarily constructed from a Poisson point process. Nothing ensures that almost all its trajectories are deterministic flows of partitions. However we show the existence of a modification of this stochastic flow of partitions such that almost all its trajectories are a deterministic flow of partitions. The lookdown representation can therefore be applied to the modification. We will need this result for the unification.

\subsubsection{Poissonian construction}\label{SubsubsectionStoFoPPoisson}
Let $(\Omega,\ccF,\bbP)$ be a probability space. Fix a finite measure $\Lambda$ on $[0,1)$ and consider a $\Lambda$ lookdown graph ${\cal P}$ as introduced in Definition \ref{DefLambdaLDGraph}. For $\mathbb{P}$-a.a.~$\omega \in \Omega$, ${\cal P}(\omega)$ is a deterministic lookdown graph so that we can define $\hat{\Pi}(\omega):=J^{-1}(\cP(\omega))$.
\begin{proposition}
The collection of partitions $\hat{\Pi}$ is a $\Lambda$ flow of partitions.
\end{proposition}
\begin{proof}
The cocycle and continuity properties are satisfied for $\bbP$-a.a.~trajectories by construction. The independence of the increments comes from the independence properties of Poisson point processes. Finally the Poissonian construction of coalescent processes (see Section 4.2.3 of the book of Bertoin~\cite{BertoinRandomFragmentation}) ensures that $(\hat{\Pi}_{-t,0},t\geq 0)$ is a $\Lambda$-coalescent.
\end{proof}
\noindent Let $\xi_0=(\xi_0(i))_{i\geq 1}$ be a sequence of i.i.d.~uniform$[0,1]$ r.v. As remarked after Lemma \ref{LemmaGenealogyPi}, one can construct the lookdown process associated with $\cP$ and $\xi_0$ as follows.
\begin{definition}[Second definition of the lookdown process.]\label{DefLDProcess2} The lookdown process associated with $\hat{\Pi}$ and $\xi_0$ is the unique collection of processes $(\xi_t(i),t\geq 0),i\geq 1$ such that $\bbP$-a.s.~for every integer $i,j \geq 1$ and all $t\geq 0$
\[ j\in\hat{\Pi}_{0,t}(i) \Rightarrow \xi_t(j)=\xi_0(i).\]
This process coincides $\bbP$-a.s.~with the lookdown process associated with $\cP$ and $\xi_0$ of Definition \ref{DefLDProcess}.
\end{definition}
\noindent The exchangeability of $\hat{\Pi}_{0,t}$ and $\xi_0$ ensures that $(\xi_t(i))_{i\geq 1}$ is itself exchangeable (see for instance the proof of Theorem 2.1 in~\cite{BertoinRandomFragmentation}). To provide a definition of the process of empirical measures of this lookdown process in terms of the flow of partitions, we need the following result on the regularity of $\hat{\Pi}$ (the proof is postponed to Subsection \ref{AppendixRegularityFoP}).
\begin{proposition}\label{PropRegularityFoP}
There exists an event of $\bbP$-probability one on which the following holds true:\begin{enumerate}[i)]
\item The trajectories of $\hat{\Pi}$ are deterministic flows of partitions.
\item For every $s \in\bbR, t\geq 0$ the partitions $\hat{\Pi}_{s,s+t},\hat{\Pi}_{s,s+t-},\hat{\Pi}_{s-,s+t}$ possess asymptotic frequencies and whenever $t > 0$ we have the following convergences for every integer $i\geq 1$
\begin{eqnarray*}
\lim\limits_{\epsilon\downarrow 0}|\hat{\Pi}_{s,s+t+\epsilon}(i)|=\lim\limits_{\epsilon\downarrow 0}|\hat{\Pi}_{s+\epsilon,s+t}(i)|&=&|\hat{\Pi}_{s,s+t}(i)|\\
\lim\limits_{\epsilon\downarrow 0}|\hat{\Pi}_{s,s+t-\epsilon}(i)|&=&|\hat{\Pi}_{s,s+t-}(i)|\\
\lim\limits_{\epsilon\downarrow 0}|\hat{\Pi}_{s-\epsilon,s+t}(i)|&=&|\hat{\Pi}_{s-,s+t}(i)|
\end{eqnarray*}
\item For every $s\in\bbR$, $t\mapsto\sum_{i\geq 1}|\hat{\Pi}_{s,s+t}(i)|$ is c\`adl\`ag on $(0,\infty)$.
\end{enumerate}
\end{proposition}
\noindent Now we set for all $t\geq 0$
\[ \ccE_{0,t}(\hat{\Pi},\xi_0) := \sum_{i\geq 1}|\hat{\Pi}_{0,t}(i)|\delta_{\xi_0(i)} + \big(1-\sum_{i\geq 1}|\hat{\Pi}_{0,t}(i)|\big)\,\ell\;,\]
where $\ell$ is the Lebesgue measure on $[0,1]$. This definition makes sense on the event of $\bbP$-probability one of Proposition \ref{PropRegularityFoP}. On the complementary event, we set any arbitrary values to $(\ccE_{0,t},t\geq 0)$.
\begin{proposition}\label{Corollary}
The process $(\ccE_{0,t}(\hat{\Pi},\xi_0),t\geq 0)$ is a c\`adl\`ag $\Lambda$-Fleming-Viot process. It coincides $\bbP$-a.s. with the process of empirical measures of the lookdown process $(\xi_t(i),t\geq 0),i\geq 1$ introduced in Theorem \ref{ThDK}.
\end{proposition}
\begin{proof}
From Proposition \ref{PropRegularityFoP}, we know that $\bbP$-a.s. $(\ccE_{0,t}(\hat{\Pi},\xi_0),t> 0)$ is a c\`adl\`ag $\ccM_1$-valued process. Moreover for all $t\geq 0$, $\bbP$-a.s. $\ccE_{0,t}(\hat{\Pi},\xi_0)$ coincides with the empirical measure of $(\xi_t(i),i\geq 1)$ (see for instance Lemma 2 of Foucart~\cite{Foucart12}). From Theorem \ref{ThDK}, we know that the process of empirical measures of the lookdown process is a c\`adl\`ag $\Lambda$-Fleming-Viot process. Since two c\`adl\`ag processes that coincide $\bbP$-a.s.~for all rational values are $\bbP$-a.s.~equal, the almost sure equality of the statement follows.
\end{proof}
\begin{remark}\label{RemarkLookdown}
If $\xi_0$ is not independent of the whole flow of partitions but only of $(\hat{\Pi}_{s,t},0 \leq s \leq t)$ then Proposition \ref{Corollary} obviously still holds.
\end{remark}
\begin{remark}
The process $(\ccE_{0,t}(\hat{\Pi},\xi_0),t> 0)$ is even c\`adl\`ag for the total variation distance on $\ccM_1$. This is a consequence of the fact that the atomic support does not evolve in time, together with the continuity of the frequencies of all the blocks and of the dust component.
\end{remark}

\subsubsection{Regularisation}\label{SubsubsectionStoFoPRegul}
We now consider a $\Lambda$ flow of partitions $\hat{\Pi}$ in the sense of Definition \ref{DefStoFoP}. Observe that its trajectories are not necessarily deterministic flows of partitions. Indeed, the cocycle property does not necessarily hold simultaneously for all triplets $r < s < t$ on a same event of probability one. Hence the bijection $J$ cannot be directly applied to obtain lookdown graphs. Below we prove the existence of a modification $\tilde{\hat{\Pi}}$ whose trajectories are genuine deterministic flows of partitions. The reason that motivates this technical discussion is that we will identify in Section \ref{SectionUnification} a stochastic flow of partitions embedded into a flow of bridges from which we will construct a lookdown process.
\begin{proposition}\label{PropRegularisation}
Let $\hat{\Pi}$ be a $\Lambda$ flow of partitions. There exists a collection of random partitions $\tilde{\hat{\Pi}}$ and an event $\Omega_{\hat{\Pi}}$ of probability one such that for all $s \leq t$, almost surely $\hat{\Pi}_{s,t}=\tilde{\hat{\Pi}}_{s,t}$ and such that for all $\omega\in\Omega_{\hat{\Pi}}$, the trajectory $\tilde{\hat{\Pi}}(\omega)$ is a deterministic flows of partitions.
\end{proposition}
\noindent Our strategy of proof is to restrict to the rational marginals of the flow of partitions $\hat{\Pi}$, and to prove that they satisfy the properties of a deterministic flow of partitions up to an event of probability zero. Then, we take right and left limits on this object and prove that we recover a modification of the initial flow. The proof is postponed to Subsection \ref{SubsectionProofRegularisation}.

As a consequence of Proposition \ref{PropRegularisation}, we can set $\cP(\omega):=J(\tilde{\hat{\Pi}}(\omega))$ for all $\omega\in\Omega_{\hat{\Pi}}$. On the complementary event, $\cP(\omega)$ can be set to any arbitrary value. The point process $\cP$ is a $\Lambda$ lookdown graph so that Proposition \ref{PropRegularityFoP} can be applied to $\tilde{\hat{\Pi}}$. This ensures, in particular, that almost surely all the partitions defined by $\tilde{\hat{\Pi}}$ admit asymptotic frequencies.

\section{The Eves}\label{SectionEves}
We start with the classification into four regimes of the $\Lambda$-Fleming-Viot process. In view of the proof of Proposition \ref{PropClassif}, we recall some results of the literature on the behaviour of the $\Lambda$-coalescent, we refer to Pitman~\cite{Pitman99}, Schweinsberg~\cite{Schweinsberg00}, Bertoin and Le Gall~\cite{BertoinLeGall-3}, Gnedin et al.~\cite{GnedIksaMaryn11} and Freeman~\cite{FreemanSingletons} for the proofs. Let $(\Pi_t,t\geq 0)$ be a $\Lambda$-coalescent. We use the regimes introduced in the statement of Proposition \ref{PropClassif} which are characterised in terms of the measure $\Lambda$. In regime \Disc, almost surely for all $t>0$ the partition $\Pi_t$ has dust and finitely many non-singleton blocks. In regime \Dust, almost surely for all $t>0$ the partition $\Pi_t$ has dust and infinitely many non-singleton blocks. In regime \Inf, almost surely for all $t>0$ the partition $\Pi_t$ has no dust and infinitely many non-singleton blocks. Finally in regime \Fin, almost surely for all $t>0$ the partition $\Pi_t$ has no dust and finitely many non-singleton blocks. In the latter regime, we say that the $\Lambda$-coalescent Comes Down from Infinity (CDI).\vspace{6pt}\\
\textit{Proof of Proposition \ref{PropClassif}.} Let $(\rho_t,t\geq 0)$ be a $\Lambda$-Fleming-Viot process. Fix $t>0$. We know that $\rho_t$ has the same distribution on $\ccM_1$ as the r.v.~$\ccE_{0,t}(\hat{\Pi},\xi_0)$ defined in Subsection \ref{SubsubsectionStoFoPPoisson}. By definition of the measure $\ccE_{0,t}(\hat{\Pi},\xi_0)$, the dust and the number of atoms of this random probability measure have the same law as the dust and the number of non-singleton blocks of a $\Lambda$-coalescent taken a time $t$. Using the results on the $\Lambda$-coalescent recalled above, we obtain the asserted classification.\cqfd\\

To study the existence of the Eves we need to deal with a regular version of the $\Lambda$-Fleming-Viot process. Let $\bP$ be the distribution of the $\Lambda$-Fleming-Viot process on the space $\bbD:=\bbD([0,\infty),\ccM_1)$ of c\`adl\`ag $\ccM_1$-valued functions endowed with the usual Skorohod's topology. Recall that $\ccM_1$ is equipped with the topology of the weak convergence of probability measures. We denote by $\ccB(\bbD)$ the Borel sigma-field associated with $\bbD$ augmented with the $\bP$-null sets.
\begin{proposition}\label{PropMeasurabilityEves}
The set
\[ \cO := \{\rho \in \bbD: \rho\text{ admits an infinite sequence of Eves} \}\]
belongs to $\ccB(\bbD)$. Moreover the map
\begin{eqnarray*}
\bbD\cap \cO &\rightarrow& [0,1]^{\bbN}\\
\rho &\mapsto& (\rme^i)_{i\geq 1}
\end{eqnarray*}
is measurable when $[0,1]^\bbN$ is endowed with the product sigma-field.
\end{proposition}
\noindent A consequence of this result is that the existence of the Eves does not depend on the construction of the $\Lambda$-Fleming-Viot process we are using.
\begin{proof}
We start with regimes \Disc, \Dust\ and \Inf. We set for any $k,l,n\geq 1$ and any $\epsilon \in (0,1)$
\[ \cO(k,\epsilon,l,n)\!\!:=\!\!\underset{i_1,\ldots,i_k \in [2^n]}{\bigcup}\,\underset{j\in[k]}{\bigcap}\underset{\substack{t\geq l\\t\in\bbQ}}{\bigcap}\bigg\{\rho_t\Big(\Big[\frac{i_j-1}{2^{n}},\frac{i_j}{2^{n}}\Big)\Big) > (1-\epsilon)\bigg(1-\sum_{m=1}^{j-1}\rho_t\Big(\Big[\frac{i_m-1}{2^{n}},\frac{i_m}{2^{n}}\Big)\Big)\bigg)\bigg\}\]
which is clearly an element of $\ccB(\bbD)$. Subsequently we set
\[ \cO(k) = \underset{\epsilon \in (0,1)\cap\bbQ}{\bigcap}\, \varliminf\limits_{l\rightarrow\infty}\varliminf\limits_{n\rightarrow\infty} \cO(k,\epsilon,l,n).\]
The event $\cO(k)$ is the event where the $k$ first Eves exist. Then $\cO=\lim_{k\rightarrow\infty}\downarrow \cO(k)$ so that $\cO$ belongs to $\ccB(\bbD)$. The event $\{\rme^1 < x \}$ is then obtained by modifying the definition of $\cO(1,\epsilon,l,n)$ by restricting to the $i_1$'s which are smaller than $\inf\{j\in[2^n]:x\leq j2^{-n}\}$ and taking the limit as above. This can be generalised easily to prove the measurability of the whole sequence.\vspace{4pt}\\
We now turn to regime \Fin\ which is more involved. We rely on the following four claims.\vspace{-14pt}
\paragraph{Claim 1.} The set 
\begin{equation*}
C:=\left\{\rho \in \bbD:\begin{array}{l} \forall t\in\bbQ_+^*, \rho_t\text{ is a weighted sum of finitely many atoms}\\
\forall t,s \in\bbQ_+^* \text{ the atoms at time }t+s\mbox{ are a subset of the atoms at time }t.\end{array}\right\}
\end{equation*}
belongs to $\ccB(\bbD)$.\vspace{-8pt}
\paragraph{Claim 2.} On the set $C$, for all $t>0$ the measure $\rho_t$ is a sum of finitely many atoms. Furthermore, for every $s \in \bbQ_+^*$ and every $t > s$ the atoms of $\rho_t$ are necessarily atoms of $\rho_s$.\vspace{6pt}\\
Consequently on $C$ we can define $\#\rho_t$ as the number of atoms of the measure $\rho_t$ for any given time $t>0$. Of course $\rho\mapsto\#\rho_t$ is measurable from $C$ to $\bbN$.\vspace{-5pt}
\paragraph{Claim 3.} For every $i\geq 2$ and every $\rho\in C$ we define $\tau_i :=\inf\{t> 0: \#\rho_t < i\}$. Then $\tau_i$ is an $\ccH_t$-stopping time where $\ccH_t:=\sigma(\rho_s,s\in[0,t])$ augmented with the $\bP$-null sets.\vspace{-8pt}
\paragraph{Claim 4.} The set
\begin{equation*}
C'\!:= \left\{\rho \in \bbD:\begin{array}{l}\forall t\in\bbR_+^*, \rho_t\text{ is a weighted sum of finitely many atoms}\\
\forall t,s \in\bbR_+^* \text{ the atoms at time }t+s\mbox{ are a subset of the atoms at time }t\\
t \mapsto \#\rho_t\text{ is c\`adl\`ag on }(0,\infty)
\end{array}\right\}
\end{equation*}
belongs to $\ccB(\bbD)$.\vspace{8pt}\\
The set $\cO$ is the subset of $C'$ where $t \mapsto \#\rho_t$ makes only jumps of size $-1$ so that $\cO$ belongs to $\ccB(\bbD)$. To end the proof of Proposition \ref{PropMeasurabilityEves} in regime \Fin\ we observe that on the event $\cO$, we have for every $x\in[0,1]$
\[ \{\rme^1\leq x\} = \cO\,\bigcap\bigg(\underset{t\in\bbQ_+}{\bigcup}\{\rho: \#\rho_t=1, \rho_t([0,x])=1\}\bigg)\]
so that the measurability of $\rme^1$ is immediate. This can be generalised easily to prove the measurability of the whole sequence.\\
\textit{Proof of Claim 1.} For every $t,s > 0$ and every $k\geq 1, m\geq 0, n\geq 1$ let
\begin{eqnarray*}
C(k,t,m,s,n):=\!\!\!\!\!\underset{\substack{i_1,\ldots,i_{k+m} \in [2^n]\\i_1,\ldots,i_{k+m}\text{are distinct}}}{\bigcup}\!\!\!\!\!\bigg(\!\!\!\!\!\!\!&&\underset{j\in[k]}{\bigcap}\bigg\{\rho_t\Big(\Big[\frac{i_j-1}{2^{n}},\frac{i_j}{2^{n}}\Big)\Big) > 0;\rho_{t+s}\Big(\Big[\frac{i_j-1}{2^{n}},\frac{i_j}{2^{n}}\Big)\Big)>0\bigg\}\\
\!\!\!\!\!\!\!&&\!\!\!\!\!\underset{k<j\leq k+m}{\bigcap}\!\bigg\{\rho_t\Big(\Big[\frac{i_j-1}{2^{n}},\frac{i_j}{2^{n}}\Big)\Big) > 0;\rho_{t+s}\Big(\Big[\frac{i_j-1}{2^{n}},\frac{i_j}{2^{n}}\Big)\Big)=0\bigg\}\\
\!\!\!\!\!\!\!&&\!\!\!\!\!\!\!\!\!\!\!\!\!\underset{i\in[2^n],i\ne i_1,\ldots,i_{k+m}}{\bigcap}\!\!\!\!\!\Big\{\rho_t\Big(\Big[\frac{i-1}{2^{n}},\frac{i}{2^{n}}\Big)\Big) =\rho_{t+s}\Big(\Big[\frac{i-1}{2^{n}},\frac{i}{2^{n}}\Big)\Big)= 0\Big\}\bigg)
\end{eqnarray*}
which is obviously a measurable set. Then we have
\[ C = \underset{t,s \in \bbQ_+^*}{\bigcap}\underset{k\geq 1}{\bigcup}\underset{m\geq 0}{\bigcup}\varliminf\limits_{n\rightarrow\infty}C(k,t,m,s,n).\]
so that it is a Borel set of $\bbD$. Claim 1 is proved.\\
\textit{Proof of Claim 2.} Fix $t > 0$ and let $s \in (0,t)\cap\bbQ$. We restrict ourselves to $\rho$'s which belong to $C$. We know that for every $\epsilon > 0$ such that $t+\epsilon \in \bbQ$ the atoms of $\rho_{t+\epsilon}$ are atoms of $\rho_s$. Let $S$ be the union of the atoms of $\rho_{t+\epsilon}$ when $\epsilon$ varies. Necessarily $S$ is a finite set so that it is a closed set of $[0,1]$. By the right continuity of $\rho$ we deduce that
\[ \rho_t([0,1]\backslash S) \leq \varliminf\limits_{\epsilon \downarrow 0} \rho_{t+\epsilon}([0,1]\backslash S)=0\]
where implicitly $\epsilon$ is taken such that $t+\epsilon$ belongs to $\bbQ$. This ensures that $\rho_t$ is a sum of finitely many atoms which are also atoms of $\rho_s$. Claim 2 is proved.\\
\textit{Proof of Claim 3.} If $\tau_i \leq t$ then there are two cases. Either at time $t$ the number of atoms is smaller than $i$, this is a measurable event. Or at any time $s\in\{t\}\cup\big((0,t)\cap\bbQ\big)$ the number of atoms is larger than or equal to $i$ and there exists at least one time $r\in(0,t)\backslash\bbQ$ such that $\#\rho_r < i$. Let us show the measurability in the latter case. Fix $m\geq 0$ and consider the largest interval of time $[s_-,s_+)$ on which for all rational $s\in[s_-,s_+)\cap\bbQ$ the measure $\rho_s$ has $m+i$ atoms. We call $a_1 < \ldots <a_{m+i}$ these atoms. By the right continuity of $\rho$ we know that $s\mapsto \rho_s(\{a_1\}),\ldots,\rho_s(\{a_{m+i}\})$ are c\`adl\`ag processes on $[s_-,s_+)$. For each collection $j_1 < \ldots < j_{m+1}$ of $m+1$ indices among $[m+i]$, we introduce the hitting time of $0 \in \bbR^{m+1}$ by $(s_-,s_+)\ni s\mapsto \big(\rho_s(\{a_{j_1}\}),\ldots,\rho_s(\{a_{j_{m+1}}\})\big)$ which is a stopping time in the augmented filtration $\ccH_t$. Then the latter case above coincides $\bP$-a.s.~with the union on $m\geq 0$ and on the $j_1<\ldots<j_{m+1}$'s of the event where this hitting time belongs to $[s_-,s_+)\backslash\{0\}$. The claim follows.\\
\textit{Proof of Claim 4.} The set $C'$ is the subset of $C$ where the times $\tau_i$ are the infimum of the rational times at which $\rho$ has less than $i$ atoms. Therefore
\[ C' := C \,\underset{i\geq 2}{\bigcap}\Big\{\tau_i=\inf\{t\in\bbQ_+^*: \#\rho_t < i\}\Big\}\]
and the claim follows. This ends the proof of the proposition.
\end{proof}

From now on, we rely on the lookdown representation to study the Eves. We consider a probability space $(\Omega,\ccF,\bbP)$ on which is defined a $\Lambda$ flow of partitions $\hat{\Pi}$ arising from the Poissonian construction of Subsection \ref{SubsubsectionStoFoPPoisson} and an independent sequence $\xi_0$ of i.i.d. uniform$[0,1]$ r.v. We then set $\rho_t:=\ccE_{0,t}(\hat{\Pi},\xi_0)$ for all $t\geq 0$. Proposition \ref{Corollary} ensures that $(\rho_t,t\geq 0)$ is a $\Lambda$-Fleming-Viot $\bbP$-a.s. c\`adl\`ag. The study of the existence of the Eves in regime \Fin\ requires a specific technical work. We define $\#\pi$ as the number of non-empty blocks of a partition $\pi$.
\begin{lemma}\label{LemmaMeasurabilityCDI}
In regime \Fin\ we define for every $i\geq 1$ the r.v. $d(i) := \inf\{t> 0: |\hat{\Pi}_{0,t}(i)|=0\}$. Then $\bbP$-a.s. for every $i\geq 1$ we have $d(i)=\inf\{t> 0: \hat{\Pi}_{0,t}(i)=\emptyset \}$. Consequently $\bbP$-a.s.\begin{itemize}
\item for every $i\geq 1$, $d(i) \geq d(i+1)$,
\item for every $i\geq 1$, the block $\hat{\Pi}_{0,t}(i)$ has a strictly positive frequency when $t\in(0,d(i))$ and is empty when $t\in[d(i),\infty)$,
\item $d(i)\downarrow 0$ as $i\rightarrow \infty$.
\end{itemize}
\end{lemma}
\begin{remark}
If we have a strict inequality $d(i) > d(i+1)$ for every $i\geq 1$, then we have existence of an infinite sequence of Eves.
\end{remark}

\begin{proof}
First observe the following deterministic fact. For any two partitions $\pi,\pi' \in \ccP_\infty$, if $\pi'$ has $n$ blocks then $\Coag(\pi,\pi')$ has at most $n$ blocks. Recall that in regime \Fin\ at any time $t > 0$ the exchangeable partition $\hat{\Pi}_{0,t}$ has $\bbP$-a.s.~no dust and finitely many blocks that all have stricly positive asymptotic frequency. Using these two observations and the fact that $\bbP$-a.s.
\[ \rho_t = \sum_{i\geq 1}|\hat{\Pi}_{0,t}(i)|\delta_{\xi_0(i)} + \big(1-\sum_{i\geq 1}|\hat{\Pi}_{0,t}(i)|\big)\ell\]
we easily get that $\bbP(C)=1$ where $C$ is the event introduced in Claim 1 of the proof of Proposition \ref{PropMeasurabilityEves}.\\
We now prove that $\bbP(C')=1$ where $C'$ is the event introduced in Claim 4. For every $i\geq 2$ we define $q_i:=\inf\{t\in\bbQ_+^*: \#\rho_t<i\}$. As $\bbP(C)=1$ we know that $\bbP$-a.s.~the sequence $q_i,i\geq 2$ is non-increasing. Since $\hat{\Pi}_{0,0}=\tO_{[\infty]}$ and $\bbP$-a.s.~for every $t\in\bbQ_+^*$ the partition $\hat{\Pi}_{0,t}$ has a finite number of blocks which have positive asymptotic frequencies we deduce that $\bbP$-a.s.~the sequence $q_i,i\geq 2$ has no positive lower bound and thus converges to $0$. Moreover $\bbP$-a.s.~$q_2<\infty$ since $\bbP$-a.s.~there exists $t\in\bbQ_+^*$ such that $\#\hat{\Pi}_{0,t}=1$. For every $i\geq 2$ and every $\epsilon\in\bbQ_+^*$ we define $\tau_i(\epsilon):=\inf\{t\geq \epsilon: \#\rho_t<i\}$. The right continuity of $\rho$ entails that $\bbP$-a.s.~the measure $\rho_{\tau_i(\epsilon)}$ has less than $i$ atoms. Furthermore the arguments in the proof of Claim 3 above ensure that this is an $\ccH_t$-stopping time. Then Proposition 3.1 in Donnelly and Kurtz~\cite{DK99} yields that the sequence $(\xi_{\tau_i(\epsilon)}(j))_{j\geq 1}$ is exchangeable. Necessarily its empirical measure is $\bbP$-a.s. equal to $\rho_{\tau_i(\epsilon)}$. By de Finetti Theorem (see for instance p.103 in Bertoin~\cite{BertoinRandomFragmentation}) conditionally given $\rho_{\tau_i(\epsilon)}$ the sequence $(\xi_{\tau_i(\epsilon)}(j))_{j\geq 1}$ is i.i.d. with distribution $\rho_{\tau_i(\epsilon)}$. Therefore $\bbP$-a.s.~this sequence takes its values in the set of atoms of $\rho_{\tau_i(\epsilon)}$. This together with the exchangeability of $\hat{\Pi}_{0,t}$ at any given time $t$ ensures the existence of an event $\Omega^*$ of $\bbP$-probability $1$ on which $\hat{\Pi}$ fulfils the regularity properties of Proposition \ref{PropRegularityFoP}, on which $\rho\in C$ and on which for all $t\in\bbQ_+^* \cup\{\tau_i(\epsilon),i\geq 2,\epsilon\in\bbQ_+^*\}$ we have\begin{enumerate}[(a)]
\item $\#\hat{\Pi}_{0,t} = \#\rho_{t}$,
\item The quantity $|\hat{\Pi}_{0,t}(j)|=\rho_{t}(\{\xi_{0}(j)\})$ is strictly positive i.f.f. $j\leq \#\hat{\Pi}_{0,t}$.
\end{enumerate}
We now work deterministically on $\Omega^*$. Fix $i\geq 2$. There exists $\epsilon \in \bbQ_+^*$ such that $\epsilon < q_i$. Necessarily $\tau_i(\epsilon)\leq q_i$. Suppose that $\tau_i(\epsilon)<q_i$. Then by (a) and (b) above and the deterministic fact at the beginning of the proof, we have $\#\rho_t<i$ for all $t\in[\tau_i(\epsilon),\infty)$ which is in contradiction with the definition of $q_i$. Therefore $\tau_i(\epsilon)=q_i$ and for all $\epsilon' \in\bbQ_+^*$ such that $\epsilon'<\epsilon$ we have $\tau_i(\epsilon')=q_i$. Since $\tau_i = \lim_{\epsilon\downarrow 0}\tau_i(\epsilon)$ we deduce that $\tau_i=q_i$. Consequently $\bbP(C')=1$.\\
Set $\tau_1=\infty$. Notice that the right continuity of the asymptotic frequencies of the blocks ensures that $|\hat{\Pi}_{0,d(i)}(i)|=0$. For every $i\geq 1$ the atom $\xi_0(i)$ disappears at time $d(i)$ and thus necessarily $d(i)=\tau_k$ for a certain $k\geq 1$. Let us show that $d(i)=\tau_i$ for every $i\geq 1$. If $d(1) < \infty$ then $k\geq 2$ so that properties (a) and (b) above hold true for $t=d(1)$ and we deduce that $\#\hat{\Pi}_{0,d(1)}$ has no blocks which is not possible for a partition. Consequently $d(1)=\infty$. We now prove by induction that $d(i)=\tau_i \leq d(i-1)$ for every $i\geq 2$. At rank $i=2$, we know that for all $t\in(0,d(2))$ we have $\rho_t(\{\xi_0(j)\})>0$ when $j=1,2$ so that $\#\rho_t \geq 2$ and $\tau_2 \geq d(2)$. Since $\rho_{d(2)}(\{\xi_{0}(2)\})=0$ we deduce using property (b) above that $\tau_2=d(2)$. As $d(1) = \infty$, the inequality $d(2) \leq d(1)$ is trivially satisfied. Suppose now that $d(j)=\tau_{j}$ for all $j\in[i-1]$ with $i\geq 2$ being given. Necessarily $d(i)\leq d(i-1)$. Indeed if $d(i) > d(i-1)$ then $\rho_{d(i-1)}(\{\xi_{0}(i)\}) > 0$ and property (b) would yield that 
$\rho_{d(i-1)}(\{\xi_{0}(i-1)\}) > 0$ which is contradictory. Now observe that for all $t \in (0,d(i))$ we have $\rho_{t}(\{\xi_{0}(j)\})>0$ for $j\in[i]$ so that $\#\rho_t \geq i$. Consequently $d(i)$ is the first time at which $\rho$ has less than $i$ atoms. Hence $d(i)=\tau_i$. The induction is complete. We have shown that for every $i\geq 1$ on $(0,d(i))$ the asymptotic frequency of $\hat{\Pi}_{0,t}(i)$ is strictly positive and that $\hat{\Pi}_{0,t}(i)$ is empty on $[d(i),\infty)$. Furthermore since $d(i)=q_i$ we know that the sequence $d(i),i\geq 2$ is non-increasing and converges to $0$.
\end{proof}
We now proceed to the proof of Proposition \ref{PropEvesLD}. In view of Definition \ref{DefSequenceEves} we will say that we are in the \CDIPE\ when the $\Lambda$-Fleming-Viot process is in regime \Fin\ and admits an infinite sequence of Eves, and we will say that we are in the \BS\ when the $\Lambda$-Fleming-Viot process is in any other regime and admits an infinite sequence of Eves.
\begin{remark}\label{RemarkDistinct}
If we do not assume that $\xi_0$ is a sequence of i.i.d.~uniform$[0,1]$ r.v.~but only a sequence of almost surely distinct values in $[0,1]$ and if $\rho$ is a $\Lambda$-Fleming-Viot process, then Proposition \ref{PropEvesLD} still holds and the proof works verbatim.
\end{remark}
\noindent\textit{Proof of Proposition \ref{PropEvesLD}}. Assume that $\rho$ almost surely admits an infinite sequence of Eves. Consider first the \CDIPE. By construction of $\rho$, we know that $\bbP$-a.s. for all $t\geq 0$
\[ \rho_t = \sum_{i=1}^{\#\hat{\Pi}_{0,t}}|\hat{\Pi}_{0,t}(i)|\delta_{\xi_0(i)}\]
where $\#\pi$ denotes the number of blocks of a partition $\pi$. Therefore, the atoms of $\rho$ are the $\xi_0(i)$'s - which are $\bbP$-a.s.~distinct - and their masses are given by the sizes of the blocks $(|\hat{\Pi}_{0,t}(i)|,t\geq 0)$'s. By Lemma \ref{LemmaMeasurabilityCDI}, the hitting times of $0$ of these blocks are given by the r.v.~$d(i)$'s. The almost sure existence of an infinite sequence of Eves states that almost surely any two atoms of $\rho$ die at distinct times. As we have seen that these atoms are the $\xi_0(i)$'s, the almost sure existence of an infinite sequence of Eves yields that $\bbP$-a.s.~for all $i\ne j$, $\xi_0(i)$ and $\xi_0(j)$ die out at distinct times. This ensures that $\bbP$-a.s.~for all $i\geq 1$ we have $d(i) > d(i+1)$, and $\bbP$-a.s. for all $i\geq 1$ we have $\xi_0(i)=\rme^i$.\\
Consider now the \BS. For all $t\geq 0$, we denote by $a_t(1) \geq a_t(2) \geq \ldots \geq 0$ the masses of $\rho_t$ ranked in the decreasing order. We prove by induction on $i\geq 1$ that
\[ \forall j\in[i],\;\;\bbP(|\hat{\Pi}_{0,t}(j)|=a_t(j))\underset{t\rightarrow\infty}{\longrightarrow}1 \text{ and }\bbP\text{-a.s. }\frac{|\hat{\Pi}_{0,t}(j)|}{1-\sum_{k=1}^{j-1}|\hat{\Pi}_{0,t}(k)|} \underset{t\rightarrow\infty}{\longrightarrow} 1.\]
We start at rank $i=1$. From the definition of $\rho$ and the fact that the $\xi_0(i)$'s are $\bbP$-a.s.~distinct there exists an event of $\bbP$-probability one on which there is a one-to-one correspondence between the blocks of $\hat{\Pi}_{0,t}$ with positive frequency and the masses of $\rho_t$. The definition of $\rme^1$ entails that $\bbP$-a.s. $|\hat{\Pi}_{0,t}(1)|$ converges to either $0$ or $1$ as $t\rightarrow\infty$ and that $a_t(1)$ converges to $1$ as $t\rightarrow\infty$. The exchangeability of the partition $\hat{\Pi}_{0,t}$ implies that $\bbP(|\hat{\Pi}_{0,t}(1)|=a_t(1) \,|\, (a_t(k))_{k\geq 1}) = a_t(1)$. Consequently $\bbP(|\hat{\Pi}_{0,t}(1)|=a_t(1)) \rightarrow 1$ as $t\rightarrow\infty$ by the dominated convergence theorem. Since we know that $\bbP$-a.s.~$|\hat{\Pi}_{0,t}(1)|$ converges towards $0$ or $1$, we deduce that it $\bbP$-a.s.~converges to $1$. The claimed property is proved at rank $i=1$. Assume that the property holds at a given rank $i-1 \geq 1$. The existence of an infinite sequence of Eves yields that $\bbP$-a.s.
\[ \frac{a_t(i)}{1-\sum_{j=1}^{i-1}a_t(j)} \underset{t\rightarrow\infty}{\longrightarrow} 1.\]
The one-to-one correspondence between the atoms of $\rho_t$ and the blocks of $\hat{\Pi}_{0,t}$ with positive frequency along with the induction hypothesis and the existence of an infinite sequence of Eves imply that $\bbP$-a.s.
\[ \frac{|\hat{\Pi}_{0,t}(i)|}{1-\sum_{j=1}^{i-1}|\hat{\Pi}_{0,t}(j)|}\]
converges to either $0$ or $1$ as $t\rightarrow\infty$. The exchangeability of $\hat{\Pi}_{0,t}$ entails that
\[ \bbP\Big(|\hat{\Pi}_{0,t}(i)|=a_t(i) \,\Big|\,\{\forall j\in[i-1],|\hat{\Pi}_{0,t}(j)|=a_t(j)\}; (a_t(k))_{k\geq i}\Big) = \frac{a_t(i)}{1-\sum_{j=1}^{i-1}a_t(j)}.\]
Using the dominated convergence theorem and the arguments above we deduce that $\bbP(\hat{\Pi}_{0,t}(i)=a_t(i))$ converges to $1$. Henceforth $\bbP$-a.s.~$\frac{|\hat{\Pi}_{0,t}(i)|}{1-\sum_{j=1}^{i-1}|\hat{\Pi}_{0,t}(j)|}$ goes to $1$ as $t$ goes to infinity and the induction is complete. Since for all $i\geq 1$, $\rho_t(\{\xi_0(i)\})=|\hat{\Pi}_{0,t}(i)|$ we deduce from the convergence obtained by induction and the uniqueness of the Eves that $\bbP$-a.s. $\xi_0(i)=\rme^i$.\\
In the case where the $\Lambda$-Fleming-Viot process does not admit an infinite sequence of Eves the arguments above still hold for the primitive Eve.\cqfd
\begin{corollary}\label{CorDistribEves}
Consider a c\`adl\`ag $\Lambda$-Fleming-Viot process that admits almost surely a sequence of Eves $(\rme^i,i\geq 1)$. Then the Eves are i.i.d. uniform$[0,1]$ and are independent of the sequence $(\rho_t(\{\rme^i\}),t\geq 0),i\geq 1$.
\end{corollary}
\begin{proof}
First assume that the $\Lambda$-Fleming-Viot process is obtained via the lookdown representation. Recall that the Eves coincide with the initial types thanks to Proposition \ref{PropEvesLD}. Since $(\xi_0(i),i\geq 1)$ is a sequence of i.i.d.~uniform$[0,1]$ r.v.~ which is independent of $\hat{\Pi}$ and since for every $i\geq 1$, $(\rho_t(\{\rme^i\}),t\geq 0)=(\hat{\Pi}_{0,t}(i),t\geq 0)$, the asserted result follows. Now consider any c\`adl\`ag $\Lambda$-Fleming-Viot process $(\rho_t,t\geq 0)$. Proposition \ref{PropMeasurabilityEves} ensures that the Eves are $\sigma(\rho)$-measurable r.v.~and therefore the distribution of $(\rho,(\rme^i)_{i\geq 1})$ coincides with that obtained via the lookdown construction.
\end{proof}

\section{Results on the existence of the Eves}\label{SectionExistence}
The goal of this section is to prove Theorem \ref{ThCDI}, Proposition \ref{PropBS} and Theorem \ref{ThEvesDust}. We consider each of the four regimes separately. Except for regime \Inf, we rely on the lookdown representation of the $\Lambda$-Fleming-Viot process. Let $(\Omega,\ccF,\bbP)$ be a probability space on which is defined a $\Lambda$ flow of partitions arising from the Poissonian construction and an independent sequence $\xi_0=(\xi_0(i),i\geq 1)$ of i.i.d.~uniform$[0,1]$ r.v. We set $\rho_t:=\ccE_{0,t}(\hat{\Pi},\xi_0)$ for all $t\geq 0$. We also rely on the lookdown process $(\xi_t(i),t\geq 0),i\geq 1$ defined from $\hat{\Pi}$ and $\xi_0$ according to Definition \ref{DefLDProcess2}. Let us introduce some technical tools for later use. We define the filtration
\begin{equation}\label{EqFiltration}
\ccF_t:=\sigma(\hat{\Pi}_{0,s},0 \leq s \leq t),\;\;t\geq 0
\end{equation}
associated to the flow of partitions. We augment this filtration with the $\bbP$-null sets.
\begin{lemma}\label{LemmaMarkovFoP}
The process $(\hat{\Pi}_{0,t},t\geq 0)$ is a $\ccP_\infty$-valued Markov process with a Feller semigroup. For any $\ccF_t$-stopping time $\tau$, conditionally given $\{\tau<\infty\}$ the process $(\hat{\Pi}_{\tau,\tau+t},t\geq 0)$ is independent of $\ccF_\tau$ and has the same distribution as $(\hat{\Pi}_{0,t},t\geq 0)$.
\end{lemma}
\begin{proof}
The very definition of stochastic flows of partitions ensures that $(\hat{\Pi}_{0,t},t\geq 0)$ is Markov with a semigroup $Q_t$ defined as follows. For every $\pi \in \ccP_{\infty}$ and every bounded measurable map $f$ on $\ccP_\infty$\
\[ Q_tf(\pi)=\bbE\Big[f(\Coag(\hat{\Pi}_{0,t},\pi))\Big].\]
Consider a bounded continuous map $f$. Since $\Coag$ is a bi-continuous operator (see Section 4.2 of the book of Bertoin~\cite{BertoinRandomFragmentation}), the dominated convergence theorem ensures that $\pi\mapsto Q_tf(\pi)$ is a bounded continuous map. Since $\ccP_\infty$ is a compact metric space, the map $f$ is uniformly continuous. For every $\epsilon > 0$, there exists $n\geq 1$ such that $d_{\ccP}(\pi,\pi')\leq 2^{-n} \Rightarrow |f(\pi)-f(\pi')|<\epsilon$ where $d_{\ccP}$ is the distance introduced in Formula (\ref{EqDistance}). Thus we get
\begin{equation*}
\sup\limits_{\pi\in\ccP_\infty}|Q_tf(\pi)-f(\pi)|\leq \epsilon\,\bbP\big(\hat{\Pi}_{0,t}^{[n]}=\tO_{[n]}\big) + 2\sup_{\pi\in\ccP_\infty}|f(\pi)|\Big(1-\bbP\big(\hat{\Pi}_{0,t}^{[n]}=\tO_{[n]}\big)\Big)
\end{equation*}
Since $\hat{\Pi}_{0,t} \rightarrow \tO_{[\infty]}$ in probability as $t\downarrow 0$, the right member goes to $\epsilon$ as $t\downarrow 0$. This implies the Feller property of $Q$. Let $\tau$ be an $\ccF_t$-stopping time. Let $A\in\ccF_{\tau}$. Let $f_1,\ldots,f_p$ be $p\geq 1$ bounded continuous maps on $\ccP_\infty$. We want to prove that for every $0 \leq t_1 \leq \ldots \leq t_p$
\[ \bbE[\tun_A\tun_{\{\tau<\infty\}} f_1(\hat{\Pi}_{\tau,\tau+t_1})\ldots f_p(\hat{\Pi}_{\tau,\tau+t_p})]=\bbE[\tun_A\tun_{\{\tau<\infty\}}]\bbE[f_1(\hat{\Pi}_{0,t_1})\ldots f_p(\hat{\Pi}_{0,t_p})].\]
We check the result for $p=1$, since the general case can be treated similarly. For every integer $n\geq 1$, let $\tau_n$ be the smallest real number of the form $k/n$ which is strictly larger than $\tau$. The right continuity of the trajectories at the first line, the independence and stationarity of the increments of a flow of partitions at the third and fourth lines ensure that
\begin{eqnarray*}
\bbE[\tun_A\tun_{\{\tau<\infty\}} f_1(\hat{\Pi}_{\tau,\tau+t_1})]&=&\lim\limits_{n\rightarrow\infty}\bbE[\tun_A\tun_{\{\tau<\infty\}} f_1(\hat{\Pi}_{\tau_n,\tau_n+t_1})]\\
&=&\lim\limits_{n\rightarrow\infty}\sum_{k\geq 0}\bbE[\tun_A\tun_{\{(k-1)/n \leq \tau < k/n\}} f_1(\hat{\Pi}_{k/n,k/n+t_1})]\\
&=&\lim\limits_{n\rightarrow\infty}\sum_{k\geq 0}\bbE[\tun_A\tun_{\{(k-1)/n \leq \tau < k/n\}}]\bbE[f_1(\hat{\Pi}_{k/n,k/n+t_1})]\\
&=&\bbE[\tun_A\tun_{\{\tau<\infty\}}]\bbE[f_1(\hat{\Pi}_{0,t_1})].
\end{eqnarray*}
\end{proof}
We also introduce the lowest level associated to an ancestor.
\begin{definition}\label{DefLt}
For all $t\geq 0$ and every $i\geq 1$, we set $L_t(i) := \inf\{j\geq 1: j\in\hat{\Pi}_{0,t}(i)\}$ where $\inf\emptyset =\infty$ by convention. The r.v. $L_t(i)$ is the lowest level at time $t$ that belongs to the progeny of the $i$-th ancestor.
\end{definition}
\noindent In regime \Fin\ and from Lemma \ref{LemmaMeasurabilityCDI} we know that $\bbP$-a.s.~for every $i\geq 1$ we have $L_t(i) < \infty$ when $t\in[0,d(i))$ while $L_t(i)=\infty$ when $t\in[d(i),\infty)$.
\begin{lemma}\label{LemmaAsymptFreq}
Let $\pi$ be a deterministic partition with a unique non-singleton block whose index is $k\geq 1$. Assume that $\pi(k)$ admits an asymptotic frequency, say $u \in (0,1]$. Then for any exchangeable random partition $\pi'$ almost surely the partition $\Coag(\pi,\pi')$ admits asymptotic frequencies and we have
\[ \forall i\geq 1,\;\;|\Coag(\pi,\pi')(i)| = |\pi'(i)|(1-u) + \tun_{\{k\in\pi'(i)\}} u.\]
\end{lemma}
\begin{proof}
Since Lemma 4.6 in the book of Bertoin~\cite{BertoinRandomFragmentation} ensures that $\Coag(\pi,\pi')$ admits asymptotic frequencies almost surely, we only need to prove the asserted formula on the these frequencies. Let $\gamma(n):=\min \pi(n)$ be the smallest integer in the $n$-th block of $\pi$. In particular $\gamma(k)=k$. Since $\pi$ has an infinite number of blocks, $\gamma(n)$ goes to infinity as $n$ tends to infinity. Observe that
\[ \forall n\geq 1,\;\; \gamma(n) -\big(\#\big(\pi(k)\cap[\gamma(n)]\big)-1\big)\vee 0 = n.\]
Since $\pi(k)$ admits an asymptotic frequency equal to $u$, we deduce that $n/\gamma(n)$ goes to $1-u$ as $n$ goes to infinity. From the definition of the coagulation operator for every $i\geq 1$ if $k\in\pi'(i)$ then the block $\Coag(\pi,\pi')(i)$ contains the elements in $\pi(k)$ together with the images through $\gamma$ of the elements in $\pi'(i)$ while if $k\notin\pi'(i)$ then the block $\Coag(\pi,\pi')(i)$ contains only the images through $\gamma$ of the elements in $\pi'(i)$. Therefore we get for every $n\geq k$
\begin{equation*}
\# \Big(\Coag(\pi,\pi')(i)\bigcap[\gamma(n)]\Big) = \#\big(\pi'(i)\cap[n]\big) + \tun_{\{k\in\pi'(i)\}}\Big(\#\big(\pi(k)\cap[\gamma(n)]\big)-1\Big).
\end{equation*}
Notice that the $-1$ in the r.h.s. prevents from counting twice the element $k$. Since the blocks of $\pi'$ have asymptotic frequencies almost surely and since $n/\gamma(n)\rightarrow 1-u$ we get the asserted equality of the lemma by dividing both members of the above formula by $\gamma(n)$ and taking the limit as $n\rightarrow\infty$.
\end{proof}

\subsection{Regime \Disc}
Recall from Proposition \ref{PropEvesLD} that the primitive Eve $\rme^1$ of Bertoin and Le Gall - which does exist without any condition on $\Lambda$ - is almost surely equal to $\xi_0(1)$. Consider the process $t\mapsto\rho_t([0,1]\backslash\{\rme^1\})$. Theorem 4 of Bertoin and Le Gall~\cite{BertoinLeGall-1} shows that this process is Markov with a Feller semigroup. Since we have constructed $\rho$ with the lookdown representation, $t\mapsto\rho_t$ evolves at the reproduction events of the lookdown process. In regime \Disc, these reproduction events are finitely many on any compact interval of time. Therefore we can enumerate $\{(s,\hat{\Pi}_{s-,s}): s > 0, \hat{\Pi}_{s-,s}\ne\tO_{[\infty]}\}$ by increasing time coordinates, say $(t_i,\pi_i),i\geq 1$. For each partition $\pi_i$, we let $u_i$ be the asymptotic frequency of its unique non-singleton block (these asymptotic frequencies are well-defined on a same event of probability one, see the proof of Proposition \ref{PropRegularityFoP}). From Formula (\ref{EqMeasureMu}) we deduce that the $u_i$'s are i.i.d.~with distribution $\nu(du)/\nu([0,1))$. We work with the collection $(t_i,u_i),i\geq 1$. We then set $\rX_0:=1$ and for every $i\geq 1$, $\rX_i:=\rho_{t_i}([0,1]\backslash\{\xi_0(1)\})$. From the arguments above, we deduce that $(\rX_i,i\geq 0)$ is a Markov chain. Let us denote by $\bQ_x$ the distribution of this Markov chain when it starts from $x\in(0,1]$. The following lemma provides the transition probabilities of this chain. 
\begin{lemma}\label{LemmaMarkovChain}
Fix $x\in(0,1]$ and let $u$ be a r.v.~with law $\nu(\cdot)/\nu([0,1))$. Under $\bQ_x$ the r.v.~$\rX_1$ is distributed as follows:
\begin{equation}
\rX_1 = \begin{cases} (1-u)x &\text{ with probability } u + (1-x)(1-u)\\
(1-u)x + u &\text{ with probability } x(1-u).\end{cases}
\end{equation}
\end{lemma}
\begin{proof}
We use the lookdown representation and work under measure $\bbP$ in the proof. Fix $i\geq 1$. We have $\bbP$-a.s. $\rX_{t_i}=1-|\hat{\Pi}_{0,t_i}(1)|$ and $\hat{\Pi}_{0,t_i}=\Coag(\pi_i,\hat{\Pi}_{0,t_{i-1}})$. Let $K$ be the index of the non-singleton block of $\pi_i$. Lemma \ref{LemmaAsymptFreq} entails that $\bbP$-a.s. if $K\in\hat{\Pi}_{0,t_{i-1}}(1)$ then $|\hat{\Pi}_{0,t_i}(1)|=|\hat{\Pi}_{0,t_{i-1}}(1)|(1-u_i) + u_i$ while if $K\notin\hat{\Pi}_{0,t_{i-1}}(1)$ then $|\hat{\Pi}_{0,t_i}(1)|=|\hat{\Pi}_{0,t_{i-1}}(1)|(1-u_i)$. Therefore
\begin{eqnarray*}
\bbP\big(\rX_i=\rX_{i-1}(1-u_i) \,|\, u_i, \rX_{i-1}\big) \!\!\!&=&\!\!\! \bbP(K\in\hat{\Pi}_{0,t_{i-1}}(1)\,|\,u_i,\rX_{i-1})\\
\bbP\big(\rX_i=\rX_{i-1}(1-u_i)+u_i \,|\, u_i,\rX_{i-1}\big) \!\!\!&=&\!\!\! \bbP(K\notin\hat{\Pi}_{0,t_{i-1}}(1)\,|\,u_i,\rX_{i-1})
\end{eqnarray*}
Observe that $\bbP$-a.s. $\{K\in\hat{\Pi}_{0,t_{i-1}}(1)\}=\{\xi_{t_{i-1}}(K)=\xi_0(1)\}$. Since $u_i$ has law $\nu(\cdot)/\nu([0,1))$ and is independent from $\rX_{i-1}$, the proof boils down to determining
\[ \bbP(\xi_{t_{i-1}}(K)=\xi_0(1)\,|\,u_i,\rX_{i-1}).\]
Since $\pi_i$ is an exchangeable random partition independent of the lookdown process up to time $t_{i-1}$ and since $u_i$ is the asymptotic frequency of its unique non-singleton block, we have for all $k\geq 1$
\[ \bbP\big(K=k \,|\, (\xi_{t_{i-1}}(j))_{j\geq 1}, u_i\big)=(1-u_i)^{k-1}u_i.\]
On the event where $K=1$ the parent of the reproduction event is of type $\xi_{t_{i-1}}(1)$. Recall that $\xi_{t_{i-1}}(1)=\xi_0(1)$ since the first particle of the lookdown process is constant. On the event where $K=k\geq 2$ the parent of the reproduction event is of type $\xi_{t_{i-1}}(k)$. Proposition 3.1 of Donnelly and Kurtz~\cite{DK99} ensures that $(\xi_{t_{i-1}}(j),j\geq 1)$ is an exchangeable sequence of r.v. Its empirical measure is $\bbP$-a.s. equal to $\rho_{t_{i-1}}$. Consequently
\[ \forall k\geq 2,\;\;\bbP\big(\xi_{t_{i-1}}(k)=\xi_{t_{i-1}}(1) \,|\, u_i, \rX_{i-1}\big) = 1-\rX_{i-1}.\]
Therefore we get
\begin{eqnarray*}
\bbP(\xi_{t_{i-1}}(K)=\xi_0(1)\,|\,u_i,\rX_{i-1}) \!\!\!&=&\!\!\! \sum_{k=1}^{\infty}\bbP(K=k | (\xi_{t_{i-1}}(j))_{j\geq 1},u_i)\\
&&\hspace{1cm}\times\bbP(\xi_{t_{i-1}}(k)=\xi_{t_{i-1}}(1) | u_i, \rX_{i-1})\\
\!\!\!&=&\!\!\! u_i + (1-\rX_{i-1})(1-u_i).
\end{eqnarray*}
This ends the proof of the lemma.
\end{proof}
\begin{proposition}\label{PropRegime1}
There exists a random time $T > 0$ after which $\bbP$-a.s.~the process $t\mapsto\rho_t(\{\xi_0(1)\})$ makes only positive jumps. In other terms, eventually all the reproduction events choose a parent of type $\xi_0(1)$.
\end{proposition}
\begin{proof}
We work on $(\Omega,\ccF,\bbP)$. We introduce the random variables $\tau_0:=\inf\{i\geq 0: \rX_i<1/2\}$ and recursively for every $n\geq 0$, $r_n:=\inf\{i > \tau_n: \rX_i-\rX_{i-1}>0\}$ and $\tau_{n+1}:=\inf\{i>r_n: \rX_i < 1/2\}$. We use the convention $\inf\emptyset = \infty$. In words, $\tau_0$ is the first time the process $\rX$ hits $(0,1/2)$ and $r_0$ is the first time after $\tau_0$ the process $\rX$ makes a positive jump. Recursively $\tau_{n+1}$ is the first time after $r_n$ the process $\rX$ hits again $(0,1/2)$, and $r_{n+1}$ is the time of the next positive jump. Set ${\cal F}_{k}:=\sigma(\rX_i, 0 \leq i \leq k)$. Recall that the objective is to prove that $\bbP$-a.s.~$\rX$ eventually makes only negative jumps. The proof of the proposition therefore boils down to showing that $\bbP$-a.s.~the sequence $(r_n)_{n\geq 0}$ eventually equals $+\infty$. The transition probabilities of the chain yields that $x\mapsto\bQ_x(\rX\mbox{ makes only negative jumps})$ is decreasing. Thus for all $n\geq 0$, $\bbP$-a.s.
\begin{eqnarray*}
\bbP(r_{n+1} = \infty\, |\, {\cal F}_{\tau_{n+1}})\mathbf{1}_{\{r_n<\infty\}}&=& \bQ_{\rX_{\tau_{n+1}}}(\rX\mbox{ makes only negative jumps})\mathbf{1}_{\{r_n<\infty\}}\\
&\geq& \bQ_{1/2}(\rX\mbox{ makes only negative jumps})\mathbf{1}_{\{r_n<\infty\}}.
\end{eqnarray*}
Consequently for all $n\geq 0$
\[ \bbP(r_n < \infty) \leq \big(1-\bQ_{1/2}(\rX\mbox{ makes only negative jumps})\big)^{n+1} .\]
It is therefore sufficient to show that $\bQ_{1/2}(\rX\mbox{ makes only negative jumps})$ is strictly positive in order to prove the proposition. From the transition probabilities of the chain, this quantity is equal to
\begin{eqnarray*}
\bbE\bigg[\prod_{i\geq 1} \Big(u_i + (1-u_i)\big(1-\frac{1}{2}\prod_{j=1}^{i-1}(1-u_j)\big)\Big)\bigg]
&\geq& \bbE\bigg[\prod_{i\geq 1} \big(1-\frac{1}{2}\prod_{j=1}^{i-1}(1-u_j)\big)\bigg]\\
&\geq&\bbE\bigg[\exp\Big(\sum_{i\geq 1} \log\big(1-\frac{1}{2}\prod_{j=1}^{i-1}(1-u_j)\big)\Big)\bigg]
\end{eqnarray*}
Let us now prove that the negative r.v.~inside the exponential is finite almost surely. Its expectation is given by
\begin{eqnarray}\label{EquationSumProd}
\bbE\bigg[\sum_{i\geq 1} \log\Big(1-\frac{1}{2}\prod_{j=1}^{i-1}(1-u_j)\Big) \bigg] &\geq& \bbE\bigg[-\frac{c}{2}\,\sum_{i\geq 1}\prod_{j=1}^{i-1}(1-u_j) \bigg]\\\nonumber
&\geq& -\frac{c}{2}\, \sum_{i\geq 1} \bbE\bigg[\prod_{j=1}^{i-1}(1-u_j)\bigg]
\end{eqnarray}
where $c$ is a positive constant such that $\log(1-y) \geq -c\,y$ for all $y\in [0,1/2]$. We have used the monotone convergence theorem to go from the first to the second line. We get
\[ \bbE\bigg[\prod_{j=1}^{i-1}(1-u_j)\bigg] = \bigg(\frac{\int_{[0,1)}(1-u)\,\nu(du)}{\int_{[0,1)}\nu(du)} \bigg)^{i-1}.\]
The measure $\nu$ is supported by $(0,1)$ since we are in regime \Disc, consequently the r.h.s.~is strictly smaller than $1$. This ensures that the series at the second line of (\ref{EquationSumProd}) converges, which in turn implies the almost sure finiteness of the negative r.v.~inside the exponential above. Therefore
\[\bQ_{1/2}(\rX\mbox{ makes only negative jumps}) > 0.\]
\end{proof}
\textit{Proof of Theorem \ref{ThEvesDust} for regime \Disc.} We know from Lemma \ref{LemmaAsymptFreq} and Proposition \ref{PropRegime1} that $\bbP$-a.s.~for every $i\geq 1$ if $t_i > T$ then $|\hat{\Pi}_{0,t_i}(1)|=(1-u_i)|\hat{\Pi}_{0,t_{i-1}}(1)|+u_i$ and $|\hat{\Pi}_{0,t_i}(2)|=(1-u_i)|\hat{\Pi}_{0,t_{i-1}}(2)|$. Consequently $\bbP$-a.s.~for every $i\geq 1$ if $t_i > T$ then $\rho_{t_i}([0,1]\backslash\{\xi_0(1)\})=(1-u_i)\rho_{t_{i-1}}([0,1]\backslash\{\xi_0(1)\})$ and $\rho_{t_i}(\{\xi_0(2)\})=(1-u_i)\rho_{t_{i-1}}(\{\xi_0(2)\})$. Moreover Proposition \ref{PropClassif} entails that $\bbP$-a.s.~for all $t\in\bbQ_+$ the measure $\rho_t$ has dust and therefore $\rho_t(\{\xi_0(2)\})<\rho_t([0,1]\backslash\{\xi_0(1)\})$. Consequently $\bbP$-a.s.~the process
\begin{equation}\label{EqProcessRatio}
\frac{\rho_t(\{\xi_0(2)\})}{\rho_t([0,1]\backslash\{\xi_0(1)\})},t\geq 0
\end{equation}
is constant after time $T$ and is strictly lower than $1$ so that the $\Lambda$-Fleming-Viot process does not admit a second Eve.\cqfd

\subsection{Regime \Dust}
Recall that for every $n\geq 2$, $L_t(n)$ is defined as  the smallest integer in $\hat{\Pi}_{0,t}(n)$ - or equivalently as the lowest level with type $\xi_0(n)$ at time $t\geq 0$. For every $t\geq 0$ we let $C_t(n)$ be equal to $1$ if $L_s(n),s\in[0,t]$ has been chosen as the parent of a reproduction event, $0$ otherwise. In other terms $C_t(n)$ equals $1$ if and only if a reproduction event has chosen a parent of type $\xi_0(n)$ on the time interval $[0,t]$. Let us describe the dynamics of the pair $(L_t(n),C_t(n), t\geq 0)$. For every $k\geq 2$ and every $l\geq 1$ we set
\[ \tq(k,k+l) := \sum_{i=2}^{(l+1)\wedge k}\binom{k}{i}\binom{l-1}{l+1-i}\lambda_{k+l,l+1}\]
where $\lambda_{k+l,l+1}$ is defined in Formula (\ref{EqLambdank}).
\begin{lemma}
The process $t\mapsto(L_t(n),C_t(n))$ is a continuous time Markov chain with values in $\bbN\times\{0,1\}$. For every $k\geq 2$ and every $l\geq 1$ the transition rates are given by
\begin{eqnarray*}
(k,0) \rightarrow \begin{cases}(k+l,0) &\text{ at rate } \tq(k,k+l)\\
(k,1) &\text{ at rate } \lambda_{k,1}\end{cases}\;\;\;\text{and}\;\;\;
(k,1) \rightarrow (k+l,1) \text{ at rate } \tq(k,k+l)
\end{eqnarray*}
\end{lemma}
\begin{proof}
This is a consequence of the transitions of the lookdown process given at Definition (\ref{DefDetLDProcess}). Indeed suppose that the process $t\mapsto L_t(n)$ is currently at level $k$. It jumps to a higher level if a reproduction event involves at least two levels among $[k]$. This higher level equals $k+l$ if:\begin{itemize}
\item $i$ levels are involved among the $[k]$ first, with $i\in\{2,\ldots,(l+1)\wedge k\}$,
\item $l+1-i$ levels are involved among $\{k+1,\ldots,k+l-1\}$,
\item level $k+l$ is not involved.
\end{itemize}
This yields the rate $\tq(k,k+l)$. Concerning the second coordinate, observe that level $k$ is chosen as the parent of a reproduction event at rate $\lambda_{k,1}$. Once the second coordinate reaches $1$ it does not evolve any more by definition.
\end{proof}
This lemma has two consequences. First the process $(L_t(n),t\geq 0)$ is a continuous time Markov chain. Second conditionally given this process, the probability that $C_\infty(n):=\lim\limits_{t\rightarrow\infty} C_t(n)$ equals $0$ is given by
\[ \bbP\big(C_\infty(n)=0\,\big|\, (L_t(n),t\geq 0)\big) = \exp\Big(-\int_0^\infty\lambda_{L_t(n),1}dt\Big).\]
Finally let us observe that the map $k\mapsto\lambda_{k,1}$ is decreasing.
\begin{proposition}\label{PropRegime2}
Consider regime \Dust\ and assume that $\int_{[0,1)}x\log\frac{1}{x}\,\nu(dx)<\infty$. Almost surely there exists at least one initial type whose frequency remains null.
\end{proposition}
\begin{remark}
A simple adaptation of the proof actually shows that there exists an infinity of initial types whose frequencies remain null.
\end{remark}
\begin{proof}
The frequency of $\xi_0(n)$ remains null if $L(n)$ is never chosen as the parent of any reproduction event or equivalently if $C_\infty(n)=0$. To prove the proposition, it suffices to show the $\bbP$-a.s. existence of an integer $n\geq 1$ such that $C_{\infty}(n)=0$. First we claim the existence of an integer $n_0\geq 2$ and of a real value $w\in(0,1]$ such that
\begin{equation}\label{EqClaimLtn}
\forall n\geq n_0,\;\; \bbP\big(C_\infty(n)=0\big)=\bbE\bigg[\exp\Big(-\int_0^\infty\!\!\lambda_{L_t(n),1}\,dt\Big)\bigg] \geq w.
\end{equation}
We postpone the proof of this claim below and complete the proof of the proposition. First observe that $\bbP$-a.s.~the sets $\{C_\infty(n)=0\}$ and $\{\forall t\geq 0, \hat{\Pi}_{0,t}(n)\text{ is a singleton}\}$ coincide. We stress that the latter set $\bbP$-a.s.~coincides with $\{\forall t\geq 0,|\hat{\Pi}_{0,t}(n)|=0\}$. One inclusion is trivial since a singleton has null frequency. Let us prove the other inclusion. By exchangeability $\bbP$-a.s.~for all $t\in\bbQ_+^*$ the block $\hat{\Pi}_{0,t}(n)$ is either a singleton or admits a strictly positive asymptotic frequency. Thus  $\bbP$-a.s.~on $\{\forall t\geq 0,|\hat{\Pi}_{0,t}(n)|=0\}$ the partition $\hat{\Pi}_{0,t}(n)$ is a singleton for all rational times $t$. We extend this property to all times $t\geq 0$ thanks to the right continuity of the blocks. The converse inclusion follows. Consequently $\bbP$-a.s.~the set $\{C_\infty(n)=0\}$ coincides with $\{\forall t\geq 0,|\hat{\Pi}_{0,t}(n)|=0\}$.\\
Consider the integer $n_0$ of (\ref{EqClaimLtn}). Let $r_0:=\inf\{t\geq 0: |\hat{\Pi}_{0,t}(n_0)| > 0\}$ be the first time at which the $n_0$-th block gets a positive frequency. On the event where $r_0 < \infty$, we let $n_1:=\inf\{n>n_0: |\hat{\Pi}_{0,r_0}(n)|=0\}$ be the lowest level $n$ such that $\xi_0(n)$ has never reproduced on $[0,r_0]$. Observe that $n_1$ is $\bbP$-a.s. finite on the event $r_0<\infty$ since on this event $\bbP$-a.s.~the partition $\hat{\Pi}_{0,r_0}$ has singleton blocks. Then recursively we define for every $k\geq 1$, $r_k:= \inf\{t\geq 0: |\hat{\Pi}_{0,t}(n_k)| > 0\}$ and on the event where $r_k < \infty$, $n_{k+1}:=\inf\{n>n_k: |\hat{\Pi}_{0,r_k}(n)|=0\}$. Here again on the event where $r_k<\infty$, the r.v.~$n_{k+1}$ is $\bbP$-a.s.~finite thanks to the same argument. To prove the proposition we need to show that $\bbP$-a.s. there exists $k\geq 1$ such that $r_k=\infty$. From (\ref{EqClaimLtn}) we have
\[ \bbP(r_0 =\infty)= \bbP\big(C_\infty(n_0)=0\big) \geq w.\]
For every $k\geq 0$, $r_k$ is a stopping time in the filtration $\ccF_t,t\geq 0$ of the flow of partitions. Lemma \ref{LemmaMarkovFoP} entails that on the event where $r_k < \infty$ the process $(\hat{\Pi}_{r_k,r_k+t},t\geq 0)$ has the same distribution as $(\hat{\Pi}_{0,t},t\geq 0)$ so that (\ref{EqClaimLtn}) does also hold for this process:
\begin{equation*}
\forall n\geq n_0,\;\; \bbP\Big(\forall t\geq 0,\,|\hat{\Pi}_{r_k,r_k+t}(n)|=0\,\big|\, r_k < \infty \Big) \geq w.
\end{equation*}
Thus we get
\[ \bbP(r_{k+1} =\infty | r_k < \infty) \geq w.\]
This easily implies that $\bbP(r_k <\infty) \leq (1-w)^{k+1}$. Since these events are nested we get $\bbP(\cap_{k\geq 0}\{r_k <\infty\}) =0$. The proof is complete.\vspace{4pt}\\
\textit{Proof of (\ref{EqClaimLtn}).} The strategy of the proof is to construct on an auxiliary probability space $(A,\cA,\bQ)$ an integer-valued process $Y=(Y_t,t\geq 0)$ which is stochastically lower than $L(n)$ for every $n\geq n_0$ with $n_0$ suitably chosen. From the description of the process $(L(n),C(n))$, the probability under $\bbP$ of $C_\infty(n)=0$ is bounded below by the expectation under $\bQ$ of $\exp\big(-\int_0^\infty\lambda_{Y_t,1}dt\big)$.\\
Let $u^*$ be a real value in $(0,1)$ such that $\Lambda\big((u^*,1)\big) > 0$ and set $a = \frac{4-u^*}{4-2u^*}$. Notice that $a > 1$. There exists $n_0 \geq 3$ such that
\begin{equation}\label{EqConstants}
0 < 1 -\frac{1-\frac{u^*}{4} - \frac{2}{n_0}}{a} < u^* \;\;\;\;,\;\;\;\;\frac{4a}{(u^*)^2n_0} < 1\;\;\;\;,\;\;\;\;an_0 \geq n_0+1.
\end{equation}
We set $u':=1 - \frac{1-\frac{u^*}{4} - \frac{2}{n_0}}{a}$. Then we claim that for every $n\geq n_0$
\begin{equation}\label{EqBoundRate}
\sum_{l=0}^{\infty}\tq(n,\lfloor an\rfloor+l) \geq \int_{(u',1)}\nu(dx)x^2\sum_{j=\lfloor an\rfloor-1-n}^{\lfloor an\rfloor-3}\binom{\lfloor an\rfloor-3}{j}x^j(1-x)^{\lfloor an\rfloor-3-j}.
\end{equation}
An analytic proof of this claim is given in Subsection \ref{AppendixProofBoundRate} but let us make the following comment. On the left, we have the total rate at which the individual at level $n$ jumps to a level above (or equal to) $\lfloor an\rfloor$ while on the right, we have the rate at which occur reproduction events satisfying:\begin{itemize}
\item The proportion of individuals involved in the reproduction event belongs to $(u',1)$.
\item The two first levels participate to the reproduction event.
\item Among levels $\{3,\ldots,\lfloor an\rfloor-1\}$ the number $j$ of levels that participate is greater than or equal to $\lfloor an\rfloor-1-n$ so that in such a reproduction event the individual at level $n$ jumps to a level above (or equal to) $\lfloor an\rfloor$.
\end{itemize}
Fix $x\in(u',1)$, $n\geq n_0$ and let $B$ be a Binomial r.v. with $\lfloor an\rfloor-3$ trials and probability of success $x$. We observe that for every given $x\in(u',1)$
\[ \sum_{j=\lfloor an\rfloor-1-n}^{\lfloor an\rfloor-3}\binom{\lfloor an\rfloor-3}{j}x^j(1-x)^{\lfloor an\rfloor-3-j}\]
is the probability that $B$ is greater than or equal to $\lfloor an\rfloor-1-n$. From the first expression in (\ref{EqConstants}) and the fact that $x\in(u',1)$ we easily verify that
\[ x(\lfloor an\rfloor-3) -\lfloor an\rfloor+1+n = n-2 -(1-x)(\lfloor an\rfloor-3) \geq n-2 -(1-x)an \geq \frac{u^*}{4}n > 0.\]
Let $\tP$ be the probability distribution of $B$. Using the Bienaym\'e-Chebyshev inequality at the third line, we get
\begin{eqnarray*}
\tP(B < \lfloor an\rfloor-1-n)&=&\tP\big(B-x(\lfloor an\rfloor-3) < \lfloor an\rfloor-1-n-x(\lfloor an\rfloor-3)\big)\\
&\leq& \tP\big(|B-x(\lfloor an\rfloor-3)| > n-2-(1-x)(\lfloor an\rfloor-3)\big)\\
&\leq& \frac{x(1-x)(\lfloor an\rfloor-3)}{\big(n-2-(1-x)(\lfloor an\rfloor-3)\big)^2}\leq \frac{4a}{(u^*)^2n}
\end{eqnarray*}
Putting these arguments together we get
\begin{eqnarray*}
\int_{(u',1)}\!\!\!\nu(dx)x^2\sum_{j=\lfloor an\rfloor-1-n}^{\lfloor an\rfloor-3}\binom{\lfloor an\rfloor-3}{j}x^j(1-x)^{\lfloor an\rfloor-3-j} &\geq& \int_{(u',1)}\nu(dx)x^2\Big(1-\frac{4a}{(u^*)^2n}\Big)\\
&\geq&\Lambda\big((u',1)\big)(1-\frac{4a}{(u^*)^2n_0})=:r.
\end{eqnarray*}
Together with (\ref{EqConstants}) and (\ref{EqBoundRate}) this ensures that the rate at which the process $L$ jumps from $n$ to a level above or equal to $\lfloor an\rfloor$ is greater than or equal to $r > 0$, uniformly for all $n\geq n_0$. We now introduce a process that starts at $n_0$ and only jumps from its current position, say $n$, to $\lfloor an\rfloor$ at rate $r$ so that it is stochastically lower than $L(n_0)$. Consider an auxiliary probability space $(A,\cA,\bQ)$ on which is defined a Poisson process with rate $r$. Denote by $0=t_0 < t_1 < t_2 < \ldots$ the jump times of this Poisson process. Still on $(A,\cA,\bQ)$ we define the process $Y=(Y_t,t\geq 0)$ as follows. Initially $Y_0=n_0$, and $Y$ stays constant on every interval $[t_k,t_{k+1})$ while its transitions are given by $Y_{t_{k+1}}=\lfloor aY_{t_{k}}\rfloor$. We then set $b:=\inf\{\lfloor an\rfloor/n, n\geq n_0\}$. Thanks to (\ref{EqConstants}) we have for every $n\geq n_0$, $\lfloor an\rfloor/n \geq (n+1)/n$. Moreover $\lfloor an\rfloor/n \rightarrow a >1$ as $n\rightarrow\infty$. Consequently $b > 1$ and thus for every $k\geq 0$, $Y_{t_k}\geq b^k$. From the bound obtained on the jump rates, we deduce that the process $Y$ is stochastically lower than the process $L(n)$ for every $n\geq n_0$. Set $w(n):=\bbP\big(C_\infty(n)=0\big)$. For every $n\geq n_0$ the expression for the Laplace transform of the exponential distribution and a simple calculation give
\begin{equation*}
w(n) \geq \bQ\Big[\exp\big(-\sum_{k\geq 0}(t_{k+1}-t_k)\lambda_{Y_{t_k},1}\big)\Big] = \exp\Big(\sum_{k=0}^{\infty}\log\Big(\frac{r}{r+\lambda_{Y_{t_k},1}}\Big) \Big).
\end{equation*}
The r.h.s.~is strictly positive if and only if $\sum_{k\geq 0}\lambda_{Y_{t_k},1} < \infty$. Using the simple inequality
\begin{equation*}
\forall m\geq 1,\;\;\lambda_{m,1}\leq \int_{(0,m^{-\frac{1}{2}}]}\!\!\!x\,\nu(dx) + (1-m^{-\frac{1}{2}})^{m-1}\int_{(0,1)}\!\!\!x\,\nu(dx)
\end{equation*}
we get
\begin{equation*}
\sum_{k= 0}^{\infty}\lambda_{Y_{t_k},1} \leq \sum_{k=0}^{\infty}\int_{\big(0,b^{-\frac{k}{2}}\big]}x\,\nu(dx)+\sum_{k=0}^{\infty}(1-Y_{t_k}^{-\frac{1}{2}})^{Y_{t_k}-1}\int_{(0,1)}\!\!\!x\,\nu(dx)
\end{equation*}
The second sum on the right converges since $Y_{t_k} \geq b^k > 1$. Thus we deduce that $\sum_{k\geq 0}\lambda_{Y_{t_k},1} < \infty$ if
\[\sum_{k=0}^{\infty}\int_{\big(0,b^{-\frac{k}{2}}\big]}x\,\nu(dx) < \infty.\]
Observe that this last quantity is equal to
\begin{eqnarray*}
\sum_{k=0}^{\infty}(k+1)\!\int_{\big(b^{-\frac{k+1}{2}},b^{-\frac{k}{2}}\big]}\!\!x\,\nu(dx)
&=&\sum_{k=0}^{\infty}\!\int_{\big(b^{-\frac{k+1}{2}},b^{-\frac{k}{2}}\big]}\!\!x\,\nu(dx)+\sum_{k=0}^{\infty}k\!\int_{\big(b^{-\frac{k+1}{2}},b^{-\frac{k}{2}}\big]}\!\!x\,\nu(dx)\\
&=&\int_{(0,1)}x\nu(dx) + \frac{2}{\log b}\sum_{k=0}^{\infty}\log\Big(b^{\frac{k}{2}}\Big)\!\!\int_{\big(b^{-\frac{k+1}{2}},b^{-\frac{k}{2}}\big]}\!\!x\,\nu(dx)
\end{eqnarray*}
which is finite since $\int_{[0,1)}x\log\frac{1}{x}\,\nu(dx)<\infty$. Therefore $w:=\inf_{n\geq n_0} w(n) > 0$.
\end{proof}
\textit{Proof of Theorem \ref{ThEvesDust} in regime \Dust.} We know from Proposition \ref{PropRegime2} that $\bbP$-a.s.~there exists $n\geq 2$ such that $\rho_t(\{\xi_0(n)\})=0$ for all $t\geq 0$ so that the $n$-th Eve is not defined.\cqfd

\subsection{Regime \Inf}\label{SubsectionBS}
When $\Lambda$ is the Lebesgue measure on $[0,1]$, one obtains the celebrated Bolthausen-Sznitman coalescent~\cite{BolthausenSznitman98}. Its $\Lambda$-Fleming-Viot counterpart belongs to regime \Inf. The proof of Proposition \ref{PropBS} relies strongly on the connection with measure-valued branching processes obtained by Bertoin and Le Gall~\cite{BertoinLeGall-0} for the Bolthausen-Sznitman coalescent, and by Birkner et al.~\cite{Article7} for all the Beta$(2-\alpha,\alpha)$ coalescents with $\alpha\in(0,2)$. We refer to Dawson~\cite{Dawson93}, Etheridge~\cite{EtheridgeSuperprocesses} and Le Gall~\cite{LeGall99} for further details on measure-valued branching processes. Let us also mention that the existence of an infinite sequence of Eves for the measure-valued branching process with the Neveu branching mechanism can be obtained thanks to the results in~\cite{DuquesneLabbe14,Labbe12}.\vspace{4pt}\\
\textit{Proof of Proposition \ref{PropBS}.} One can construct $\rho$ by rescaling a measure-valued branching process $(\rmm_t,t\geq 0)$ associated with the Neveu branching mechanism $\Psi(u)=u\log u$ as follows:
$$ \rho_t(dx) := \frac{\rmm_t(dx)}{\rmm_t([0,1])},\;\;\forall t\geq 0 $$
This result was initially stated for the Bolthausen-Sznitman coalescent by Bertoin and Le Gall in~\cite{BertoinLeGall-0}, and later on by Birkner et al. for the forward-in-time process in~\cite{Article7}. As a consequence of this connection and using Proposition \ref{PropEvesLD}, we deduce that there exists a r.v. $\rme$ such that almost surely
$$ \frac{\rmm_t(\{\rme\})}{\rmm_t([0,1])} \underset{t\rightarrow\infty}{\longrightarrow} 1 $$
The branching property ensures that the restrictions of $\rmm$ to any two disjoint subintervals of $[0,1]$ are independent. For every integer $n\geq 1$, we divide $[0,1]$ into dyadic subintervals of the form
$$ [0,2^{-n}), [2^{-n}, 2\times 2^{-n}),\ldots, [1-2^{-n},1] $$
and we consider the corresponding restrictions of $\rmm$. Obviously, for each subinterval $[(i-1)2^{-n},i\,2^{-n})$ there exists $e(i,n)$ such that
$$ \frac{\rmm_t(\{e(i,n)\})}{\rmm_t([(i-1)2^{-n},i\,2^{-n}))} \underset{t\rightarrow\infty}{\longrightarrow} 1 $$
Necessarily $\rmm$ restricted to the union of two subintervals indexed by $i\ne j \in [2^{n}]$ admits either $e(i,n)$ or $e(j,n)$ as an Eve and therefore
$$ \lim\limits_{t\rightarrow\infty}\frac{\rmm_t(\{e(i,n)\})}{\rmm_t(\{e(j,n)\})} \in \{0,+\infty\} $$
Hence one can order the $e(i,n), i\in [2^{n}]$ by asymptotic sizes. Using the consistency of the restrictions when $n$ varies, one gets the existence of a sequence $(\rme^{i})_{i\geq 1}$ fulfilling the formula of the statement.\cqfd\\
It is rather unfortunate that this simple argument does not apply to other measures in regime \Inf.

\subsection{Regime \Fin}
In this subsection, we work on the probability space $(\Omega,\ccF,\bbP)$ on which is defined the $\Lambda$ flow of partitions $\hat{\Pi}$ and the sequence of initial types $\xi_0=(\xi_0(i),i\geq 1)$ i.i.d. uniform$[0,1]$. For every $t\geq 0$ we set $\rho_t:=\ccE_{0,t}(\hat{\Pi},\xi_0)$ and we let $(\xi_t(i),t\geq 0),i\geq 1$ be the lookdown process constructed from $\hat{\Pi}$ and $\xi_0$. Recall the filtration introduced in (\ref{EqFiltration}). A classical argument, based on the right continuity and the independence of the increments of the flow of partitions, ensures that $\ccF_{0+}$ is a trivial sigma-field under $\bbP$, that is, for every $A \in \ccF_{0+}$ we have $\bbP(A) \in \{0,1\}$.\\
\textit{Proof of Proposition \ref{PropProbaE}.} Recall from Lemma \ref{LemmaMeasurabilityCDI} that $\bbP$-a.s.~for every $i\geq 1$, $d(i)$ stands for the extinction time of the initial type $\xi_0(i)$ and that it coincides with the first time at which $\hat{\Pi}_{0,t}(i)$ is empty. The r.v.~$d(i)$ is an $\ccF_t$-stopping time. We also know from Lemma \ref{LemmaMeasurabilityCDI} that $\bbP$-a.s. $d(i) \downarrow 0$ as $i\rightarrow\infty$. Then it is easy to check that $\cap_{i\geq 1}\ccF_{d(i)} = \ccF_{0+}$. The event $\tE$ of the statement can be written as follows:
\[ \tE :=\big\{ \exists i\geq 2: d(i) = d(i+1) \big\}.\]
We introduce the following collection of nested events
\[ \tE_i := \big\{\exists j\geq i:d(j) = d(j+1) \big\},\;\;i\geq 1\]
and we set $\tE_\infty := \cap_{i\geq 1}\tE_i$. Since $\tE_i \in \ccF_{d(i)}$, we deduce that $\tE_\infty\in\ccF_{0+}$ so that $\bbP(\tE_\infty)\in\{0,1\}$. To end the proof, we show that $\bbP(\tE_\infty)=\bbP(\tE)$.\\
Suppose that $\bbP(\tE_\infty)=1$. Since $\tE_\infty \subset \tE$, we deduce that $\bbP(\tE)=1$. Suppose now that $\bbP(\tE_\infty)=0$ and assume that there exists $i\geq 1$ and $p\in (0,1]$ such that $\bbP(d(i)=d(i+1))= p$. Recall Definition \ref{DefLt}. For every $k\geq i$, let $\tau_k(i) := \inf\{t\geq 0: L_t(i) \geq k\}$. This is a finite $\ccF_t$-stopping time. By consistency, $\bbP$-a.s. $L_{\tau_k(i)}(i)<L_{\tau_k(i)}(i+1)\leq +\infty$ and we have
\[\{d(i)\! =\! d(i+1)\}\!=\!\big\{(L_{\tau_k(i)+t}(i),t\geq 0) \text{ and } (L_{\tau_k(i)+t}(i+1),t\geq 0) \text{ reach }\infty \text{ simultaneously}\big\}.\]
Since $\tau_k(i)$ is a $\bbP$-a.s.~finite $\ccF_t$-stopping time, Lemma \ref{LemmaMarkovFoP} ensures that $(\hat{\Pi}_{\tau_k(i),\tau_k(i)+s},s\geq 0)$ has the same distribution as $(\hat{\Pi}_{0,s},s\geq 0)$ and is independent of $\ccF_{\tau_k(i)}$. We deduce that
\begin{eqnarray*}
\bbP(d(i) = d(i+1)) &=& \sum_{j' > j \geq k}\bbP\big(L_{\tau_k(i)}(i)=j,L_{\tau_k(i)}(i+1)=j'\big)\bbP\big(d(j) = d(j')\big)\\
&\leq& \sum_{j' > j \geq k}\bbP\big(L_{\tau_k(i)}(i)=j,L_{\tau_k(i)}(i+1)=j'\big)\bbP\big(d(j) = d(j+1)\big)\\
&\leq& \bbP(\tE_k)
\end{eqnarray*}
since $\bbP\big(d(j) = d(j+1)\big) \leq \bbP(\tE_k)$ whenever $j\geq k$. We have proved that for all $k\geq i$, $\bbP(\tE_k) \geq p > 0$. Since the events $\tE_k,k\geq i$ are nested this ensures that $\bbP(\tE_\infty) \geq p > 0$ which contradicts the hypothesis. Therefore for all $i\geq 1$, $\bbP(d(i)=d(i+1))=0$ and consequently $\bbP(\tE)=0$.\cqfd\vspace{4pt}\\
For any $n\geq 1$ and any $c > 0$ we introduce the event $\tE_{n,c} := \cup_{m\geq n}\{d\big(\lfloor(1+c)m\rfloor\big) = d(m)\}$. In order to prove Theorem \ref{ThCDI} we prove the following preliminary result. 
\begin{proposition}\label{PropPEUN}
Suppose that $\bbP(\tE)=1$. There exists $c>0$ such that $\bbP\big(\cap_{n\geq 1}\tE_{n,c})=1$.
\end{proposition}
\noindent This proposition, combined with the fact that $\bbP$-a.s.~$d(m)\downarrow 0$ as $m\rightarrow\infty$, implies that when $\bbP(\tE)=1$, there are infinitely many simultaneous extinctions of large proportions of initial types near time $0$.
\begin{proof}
The equality $\bbP(\tE)=1$ entails the existence of an integer $j\geq 2$ such that $\bbP(d(j)=d(j+1))> 0$. We claim that it implies that $\bbP(d(2)=d(3))>0$. Indeed let $T:=\inf\{t\geq 0: (L_t(2),L_t(3))=(j,j+1)\}$ and observe that $\bbP(T<\infty)>0$. By Lemma \ref{LemmaMarkovFoP} on the event $\{T<\infty\}$ the process $(\hat{\Pi}_{T,T+t},t\geq 0)$ has the same distribution as $(\hat{\Pi}_{0,t},t\geq 0)$ and is independent of $\ccF_T$. We deduce that $\bbP(d(2)=d(3)) \geq \bbP(T<\infty)\bbP(d(j)=d(j+1))$. The claim follows. Recall that $\bbP$-a.s.~$1-|\hat{\Pi}_{0,t}(1)|=\rho_t\big([0,1]\backslash\{\xi_0(1)\}\big)$. We introduce for every $n\geq 1$ the $\ccF_t$-stopping time
\[ \tau_n:=\inf\Big\{t\geq 0: \rho_t\big([0,1]\backslash\{\xi_0(1)\}\big) \leq \frac{1}{n^2}\Big\} \]
which is finite almost surely since the primitive Eve $\rme=\xi_0(1)$ fixes. We first show that $L_{\tau_n}(3)-L_{\tau_n}(2)$ becomes large as $n$ goes to infinity. To that end, we define
\[ Z_{\tau_n}:= \inf\{k>L_{\tau_n}(2): k\notin\hat{\Pi}_{0,\tau_n}(1)\} = \inf\{k>L_{\tau_n}(2): \xi_{\tau_n}(k)\ne \xi_0(1)\}. \]
Hence $L_{\tau_n}(2)$ and $Z_{\tau_n}$ are the two smallest levels at time $\tau_n$ which are not of type $\xi_0(1)$. Necessarily $\xi_{\tau_n}(Z_{\tau_n})$ is either equal to $\xi_0(2)$ or is equal to $\xi_0(3)$ in which case $Z_{\tau_n}=L_{\tau_n}(3)$. In both cases $Z_{\tau_n}$ is lower than or equal to $L_{\tau_n}(3)$. Proposition 3.1 in~\cite{DK99} ensures that the sequence $(\xi_{\tau'}(i))_{i\geq 1}$ is exchangeable with empirical measure $\rho_{\tau'}$ as soon as $\tau'$ is a stopping time in the filtration of $\rho$. This result can be extended easily to show that $(\xi_{\tau'}(i))_{i\geq 2}$ is exchangeable with empirical measure $\rho_{\tau'}$ whenever $\tau'$ is a stopping time in the filtration of the pair $\big(\,\rho,\rho(\{\xi_0(1)\})\,\big)$. We deduce that the sequence $(\xi_{\tau_n}(i))_{i\geq 2}$ is exchangeable with empirical measure $\rho_{\tau_n}$ so that conditionally given $\rho_{\tau_n}$ the $\xi_{\tau_n}(i),i\geq 2$ are i.i.d. with law $\rho_{\tau_n}$. The law of $(L_{\tau_n}(2),Z_{\tau_n})$ can then be characterised as follows.\\
Fix a real value $p\in(0,1)$ and consider an auxiliary probability space $(A,\cA,\bQ)$ on which are defined two independent geometric r.v. $G$ and $G'$ with parameter $p$, that is:
\[ \forall k\geq 1,\, \bQ(G=k)= (1-p)^{k-1}p.\]
Then the distribution under $\bbP$ of $(L_{\tau_n}(2),Z_{\tau_n})$ conditionally given $\rho_{\tau_n}\big([0,1]\backslash\{\xi_0(1)\}\big)=p$ is the distribution under $\bQ$ of $(1+G,1+G+G')$. A simple calculation yields for all $c>0$
\begin{eqnarray}\label{EqGG}
\bQ\Big(\frac{G'}{1+G} \geq c\,;\, G \geq n\Big) &=& \sum_{k=n}^{\infty}\sum_{l=c(k+1)}^{\infty}p^2(1-p)^{k+l-2}\nonumber\\
&=& \frac{p(1-p)^{c-2+n(1+c)}}{1-(1-p)^{1+c}}.
\end{eqnarray}
By the right continuity of $\rho$, $\bbP$-a.s.~$\rho_{\tau_n}\big([0,1]\backslash\{\xi_0(1)\}\big) \leq 1/n^2$. From the dominated convergence theorem and (\ref{EqGG}) we get
\[ \bbP\bigg(\frac{Z_{\tau_n}-L_{\tau_n}(2)}{L_{\tau_n}(2)} \geq c\,;\, L_{\tau_n}(2) \geq n\bigg) \underset{n\rightarrow\infty}{\longrightarrow} \frac{1}{1+c}.\]
Recall that $L_{\tau_n}(3)\geq Z_{\tau_n}$ so that
\begin{eqnarray*}
\varliminf\limits_{n\rightarrow\infty}\bbP\bigg(\frac{L_{\tau_n}(3)-L_{\tau_n}(2)}{L_{\tau_n}(2)} \geq  c\,;\, L_{\tau_n}(2) \geq n\bigg)&\geq&  \varliminf\limits_{n\rightarrow\infty}\bbP\bigg(\frac{Z_{\tau_n}-L_{\tau_n}(2)}{L_{\tau_n}(2)} \geq c\,;\, L_{\tau_n}(2) \geq n\bigg)\\
&=& \frac{1}{1+c}.
\end{eqnarray*}
We now introduce the event
\[ \tB_n := \Big\{ d(2) = d(3) \, ;\, \frac{L_{\tau_n}(3)-L_{\tau_n}(2)}{L_{\tau_n}(2)} \geq c \, ;\, L_{\tau_n}(2) \geq n \Big\}.\]
Let $\eta:=\bbP(d(2)=d(3))$ and recall we showed at the beginning of the proof that this quantity is strictly positive. Take $c > 0$ such that $\frac{1}{1+c}+\eta > 1$. There exists $\epsilon > 0$ such that $\frac{1-\epsilon}{1+c}+\eta > 1$. From the bound on the $\varliminf$ above we deduce the existence of $n_0\geq 1$ such that for all $n\geq n_0$
\begin{equation*}
\bbP\bigg(\frac{L_{\tau_n}(3)-L_{\tau_n}(2)}{L_{\tau_n}(2)} \geq  c\,;\, L_{\tau_n}(2) \geq n\bigg)\geq \frac{1-\epsilon}{1+c}.
\end{equation*}
Using the inequality $\bbP(D\cap D')\geq \bbP(D)+\bbP(D')-1$ that holds for any two events $D,D'$ we get
\begin{eqnarray*}
\bbP(\tB_n) &\geq& \bbP(d(2)=d(3)) + \bbP\bigg(\frac{L_{\tau_n}(3)-L_{\tau_n}(2)}{L_{\tau_n}(2)} \geq  c\,;\, L_{\tau_n}(2) \geq n\bigg) - 1\\
&\geq& \eta + \frac{1-\epsilon}{1+c} - 1 > 0
\end{eqnarray*}
Moreover we have
\begin{eqnarray}\label{EqBnEnc}
\tB_n =\underset{m\geq n}{\bigcup}\left\{\begin{array}{c}L_{\tau_n}(3)\geq (1+c)L_{\tau_n}(2)\,;\,L_{\tau_n}(2)=m\,;\\
(L_{\tau_n+t}(3),t\geq 0)\text{ and }(L_{\tau_n+t}(2),t\geq 0)\text{ reach }\infty\text{ simultaneously}\end{array}\right\}.
\end{eqnarray}
Since $\tau_n$ is a $\bbP$-a.s.~finite $\ccF_t$-stopping time, Lemma \ref{LemmaMarkovFoP} ensures that $(\hat{\Pi}_{\tau_n,\tau_n+s},s\geq 0)$ has the same distribution as $(\hat{\Pi}_{0,s},s\geq 0)$ and is independent of $\ccF_{\tau_n}$. This yields together with (\ref{EqBnEnc}) that for every $n\geq n_0$
\[ \bbP(\tE_{n,c}) \geq \bbP(\tB_n) \geq \eta + \frac{1-\epsilon}{1+c} - 1 > 0.\]
Since the events $\tE_{n,c},n\geq 1$ are nested, $\bbP(\cap_{n\geq 1}\tE_{n,c})=\lim_{n\rightarrow\infty}\bbP(\tE_{n,c})>0$. Finally since $\cap_{n\geq 1}\tE_{n,c}\in \ccF_{0+}$ we deduce that its probability is actually equal to $1$. The proposition is proved.
\end{proof}
\textit{Proof of Theorem \ref{ThCDI}.} Recall the function $\Psi$ from Equation (\ref{EqPsiLambda}). Since $\int^\infty\frac{du}{\Psi(u)}<\infty$, we can introduce the continuous map $t\mapsto v(t)$ as the unique solution of
$$ \int_{v(t)}^{\infty}\frac{du}{\Psi(u)} = t,\;\;\forall t > 0 $$
Proposition 15 in Berestycki et al.~\cite{CouplingBBL13} ensures that for all $\epsilon \in (0,1)$ we have
\begin{equation}\label{EqLimSupLimInf}
\bbP\bigg(\liminf\limits_{t\rightarrow 0}\frac{\#\hat{\Pi}_{0,t}}{v\big(\frac{1+\epsilon}{1-\epsilon}\,t\big)} \geq \frac{1}{1+\epsilon}\;;\;\limsup\limits_{t\rightarrow 0}\frac{\#\hat{\Pi}_{0,t}}{v\big(\frac{1-\epsilon}{1+\epsilon}\,t\big)} \leq \frac{1}{1-\epsilon}  \bigg) = 1
\end{equation}
where $\#\hat{\Pi}_{0,t}$ denotes the number of blocks of $\hat{\Pi}_{0,t}$. Assume that $\Psi$ is regularly varying at $+\infty$ with index $\alpha\in(1,2]$. Then $v$ is itself regularly varying at $0+$ with index $-1/(\alpha-1)$ (see Subsection \ref{ProofRegularVarLambda} for a proof of this fact). Consequently we have for all $\epsilon\in(0,1)$
$$ v\Big(\frac{1+\epsilon}{1-\epsilon}\,t\Big)\Big(\frac{1+\epsilon}{1-\epsilon}\Big)^{1/(\alpha-1)} \underset{t\downarrow 0}{\sim} v\Big(\frac{1-\epsilon}{1+\epsilon}\,t\Big)\Big(\frac{1-\epsilon}{1+\epsilon}\Big)^{1/(\alpha-1)} \underset{t\downarrow 0}{\sim} v(t). $$
Together with (\ref{EqLimSupLimInf}) this yields
\begin{equation*}
\bbP\bigg(\lim\limits_{t\rightarrow 0}\,\frac{\#\hat{\Pi}_{0,t}}{v(t)} = 1 \bigg) = 1.
\end{equation*}
This forces the jumps of $t\mapsto\#\hat{\Pi}_{0,t}$ to be small near $0+$. More precisely for any $c > 0$
\begin{equation}\label{EqJumpsRho}
\bbP\bigg(\limsup\limits_{t\rightarrow 0}\,\frac{\#\hat{\Pi}_{0,t-}-\#\hat{\Pi}_{0,t}}{\#\hat{\Pi}_{0,t}} < c \bigg) = 1.
\end{equation}
Observe that $d\big(\lfloor(1+c)m\rfloor\big)=d(m)=t$ implies that $\#\hat{\Pi}_{0,t-}-\#\hat{\Pi}_{0,t} \geq c\,\#\hat{\Pi}_{0,t}$. Recall that $\bbP$-a.s.~$d(m)\downarrow 0$ as $m\rightarrow\infty$. Using (\ref{EqJumpsRho}) and the fact that the collection of events $\tE_{n,c},n\geq 1$ is nested, we deduce that for any $c>0$
\begin{equation*}
\bbP\big(\underset{n\geq 1}{\cap}\tE_{n,c}) = 0.
\end{equation*}
This identity combined with Proposition \ref{PropPEUN} entails that $\bbP(\tE)$ is not equal to $1$. Proposition \ref{PropProbaE} in turn ensures that $\bbP(\tE)=0$.\cqfd

\section{Unification}\label{SectionUnification}
Let $(\Omega,\ccF,\bbP)$ be a probability space on which is defined a $\Lambda$ flow of bridges $\rF$ according to Definition \ref{DefFoB}. We consider a modification of this flow of bridges that we still denote $\rF$ and such that for every $s\in\bbR$ the process $(\rho_{s,s+t},t\geq 0)$ is a c\`adl\`ag $\Lambda$-Fleming-Viot process. This modification does exist since the $\Lambda$-Fleming-Viot process has a Feller semigroup, see p.278 in~\cite{BertoinLeGall-1} . From now on, we assume that the measure $\Lambda$ is such that the $\Lambda$-Fleming-Viot process admits almost surely an infinite sequence of Eves. Let us emphasise the fact that we only rely on Definition \ref{DefSequenceEves} and do not restrict ourselves to the particular measures $\Lambda$ presented in Proposition \ref{PropBS} and Theorem \ref{ThCDI}.

\subsection{The evolving sequence of Eves}
For each $s\in\bbR$, we define the sequence of Eves $(\rme_s^i,i\geq 1)$ of the $\Lambda$-Fleming-Viot process $(\rho_{s,s+t},t\geq 0)$. Notice that this sequence is defined on an event of $\bbP$-probability $1$ that depends on $s$. For each $s\in\bbR$, outside this event we set an arbitrary value to the sequence $(\rme_s^i,i\geq 1)$. Let us provide a simple description of the connection between the Eves taken at two distinct times. We rely on a key property due to Bertoin and Le Gall~\cite{BertoinLeGall-1} that we now recall. Consider an exchangeable bridge $B$ and an independent sequence $V=(V_i,i\geq 1)$ of i.i.d. uniform$[0,1]$ r.v. We define an exchangeable random partition $\pi=\pi(B,V)$ by setting
\[ i \sim j \Leftrightarrow B^{-1}(V_i) = B^{-1}(V_j)\]
where $B^{-1}$ is the c\`adl\`ag inverse of $B$. For each $j\geq 1$, if the $j$-th block $\pi(j)$ is not empty then we define $V'_j:= B^{-1}(V_i)$ for an arbitrary integer $i\in\pi(j)$. If the number of blocks of $\pi$ is finite, we complete the sequence $V'$ with i.i.d. uniform$[0,1]$ r.v. Then Lemma 2 in~\cite{BertoinLeGall-1} entails that the $(V'_j,j\geq 1)$ are i.i.d. uniform$[0,1]$ and are independent of $\pi$.
\begin{proposition}\label{PropEvesFoB}
Fix $s > 0$. Define the partition $\pi=\pi(\rF_{0,s},(\rme_s^i)_{i\geq 1})$. Then the sequence $(\rme_0^j,j\geq 1)$ is independent of $\pi$ and $\bbP$-a.s. for every $j\geq 1$, if $\pi(j)$ is not empty then $\rme_0^j=\rF_{0,s}^{-1}(\rme_s^i)$ for any $i\in\pi(j)$.
\end{proposition}
\noindent This allows to introduce $(\hat{\Pi}_{s,t},-\infty < s \leq t < \infty)$ by setting for every $s \leq t$
\[ \hat{\Pi}_{s,t} := \pi(\rF_{s,t},(\rme_t^i)_{i\geq 1})\text{ if }s < t\;\;,\;\; \hat{\Pi}_{s,t} := \tO_{[\infty]}\text{ if }s=t.\]
\begin{proof}
Let $\Omega_{0,s}$ be an event of $\bbP$-probability one on which the definitions of the Eves at times $0$ and $s$ hold and on which for all $t\in\bbQ_+$ we have $\rF_{0,s+t}=\rF_{s,s+t}\circ\rF_{0,s}$. Recall that $(\rme_s^i,i\geq 1)$ is a sequence of i.i.d.~uniform$[0,1]$ r.v.~which only depends on $\sigma(\rF_{s,s+t},t\geq 0)$ so that $(\rme_s^i,i\geq 1)$ is independent of $\rF_{0,s}$ from the independence of the increments of the flow of bridges. The sequence $(\rme_s^i,i\geq 1)$ plays the r\^ole of the sequence $(V_i,i\geq 1)$ above. We introduce $K$ as the random number of blocks of $\pi$ which is $\bbP$-a.s. finite in the \CDIPE\ (the bridge $\rF_{0,s}$ has a finite number of jumps and no drift) while it is $\bbP$-a.s. infinite in the \BS\ (the bridge $\rF_{0,s}$ has an infinite number of jumps and possibly a drift). For every $j\in[K]$ we set $V'_j := \rF^{-1}_{0,s}(\rme^{i_j}_s)$ where $i_j:=\min \pi(j)$. If $K$ is finite, we set $V'_j := \rme^{j}_0$ for all $j > K$. To prove the proposition, it remains to show that:
\begin{enumerate}[(i)]
\item $\bbP$-a.s. $\rme^{j}_0 = V'_j$ for every $j \in [K]$,
\item $(\rme^{j}_0)_{j \geq 1}$ is independent of $\pi$.
\end{enumerate}
We start with the first assertion. We argue deterministically on the event $\Omega_{0,s}$ of $\bbP$-probability one. We claim that for every $j \in [K]$ and all $t \in\bbQ_+$
\begin{equation}\label{EquationBounds}
\rho_{s,s+t}\big(\{\rme^{i_j}_s\}\big) \leq \rho_{0,s+t}\big(\{V'_j\}\big) \leq \rho_{s,s+t}\big([0,1]\backslash\{\rme^{1}_s,\ldots,\rme^{i_j-1}_s\}\big).
\end{equation}
Let us prove (\ref{EquationBounds}). Recall that $V'_j$ is the pre-image of $\rme_s^{i_j}$ through $\rF_{0,s}$ so that $V'_j$ is either the location of a jump of the bridge or it is a point of continuous increase. First assume that $V'_j$ is the location of a jump of $\rF_{0,s}$. This yields that $\rF_{0,s}(V'_j) \geq \rme_s^{i_j} > \rF_{0,s}(V'_j-)$. Consequently for all $t\in\bbQ_+$ $\rF_{0,s+t}(V'_j) \geq \rF_{s,s+t}(\rme_s^{i_j}) \geq \rF_{s,s+t}(\rme_s^{i_j}-) \geq \rF_{0,s+t}(V'_j-)$ and the first inequality of (\ref{EquationBounds}) follows. To prove the second inequality, observe that for every $l\in[i_j-1]$, $\rme_s^l$ does not belong to $(\rF_{0,s}(V'_j-),\rF_{0,s}(V'_j)]$. Consequently for all $t\in\bbQ_+$
\[ \rho_{0,s+t}\big(\{V'_j\}\big)=\rho_{s,s+t}\big((\rF_{0,s}(V'_j-),\rF_{0,s}(V'_j)]\big) \leq 1-\sum_{l=1}^{i_j-1}\rho_{s,s+t}(\{\rme_s^{l}\})\]
and the second inequality follows. Assume now that $\rF_{0,s}$ is continuous but increasing at $V'_j$. Then $\rF_{0,s}(V'_j) = \rme_s^{i_j} = \rF_{0,s}(V'_j-)$. Since $\rF_{0,s}$ is increasing at $V'_j$ we have for all $t\in\bbQ_+$ $\rF_{0,s+t}(V'_j-)=\rF_{s,s+t}(\rme_s^{i_j}-)$ and we get $\rho_{0,s+t}\big(\{V'_j\}\big)=\rho_{s,s+t}\big(\{\rme_s^{i_j}\}\big)$. The second equality of (\ref{EquationBounds}) derives from the fact that for every $l\in[i_j-1]$, $\rme_s^{l}\ne\rme_s^{i_j}$ so that
\[ \rho_{0,s+t}\big(\{V'_j\}\big)=\rho_{s,s+t}\big(\{\rme_s^{i_j}\}\big) \leq 1-\sum_{l=1}^{i_j-1}\rho_{s,s+t}(\{\rme_s^{l}\}).\]
Therefore (\ref{EquationBounds}) is proved. We now carry out the proof of Assertion \rm{(i)}, we treat separately two cases. We start with the \CDIPE. The bridge $\rF_{0,s}$ has no drift and $K$ jumps. These $K$ jumps are located on the points $(V'_j)_{j \in [K]}$. Necessarily the $K$ first Eves are the jump locations of $\rF_{0,s}$ so that the sets $\{V'_j,j \in [K]\}$ and $\{\rme^{j}_0,j\in[K]\}$ coincide. By definition of the Eves, the processes $t\mapsto\rho_{s,s+t}([0,1]\backslash\{\rme^{1}_s,\ldots,\rme^{i_j-1}_s\})$ and $t\mapsto\rho_{s,s+t}(\{\rme^{i_j}_s\})$ reach $0$ at the same time. From (\ref{EquationBounds}) we conclude that $t\mapsto\rho_{0,s+t}(\{V'_j\})$ reaches $0$ also at that time. Hence, the collection $(V'_j)_{j\in[K]}$ is ordered by decreasing extinction times and we conclude that $V'_j = \rme^{j}_0$ for every $j\in [K]$.
We turn to the \BS. The definition of the Eves implies that for every $j \geq 1$
\begin{equation*}
\lim\limits_{t\rightarrow\infty}\displaystyle\frac{\rho_{s,s+t}\big(\{\rme^{i_j}_s\}\big)}{\rho_{s,s+t}\big([0,1]\backslash\{\rme^{1}_s,\ldots,\rme^{i_j-1}_s\}\big)}=1
\end{equation*}
The arguments used in the proof of (\ref{EquationBounds}) can be adapted to show that for all $t\in\bbQ_+$
\[ \rho_{s,s+t}\big(\{\rme^{1}_s,\ldots,\rme^{i_j-1}_s\}\big) \leq \rho_{0,s+t}\big(\{V'_1,\ldots,V'_{j-1}\}\big) \]
Together with (\ref{EquationBounds}) this yields
\begin{equation*}
\lim\limits_{\substack{t\rightarrow \infty\\t\in\bbQ}}\frac{\rho_{0,s+t}(\{V'_j\})}{\rho_{0,s+t}\big([0,1]\backslash\{V'_1,\ldots,V'_{j-1}\}\big)} \geq \lim\limits_{\substack{t\rightarrow \infty\\t\in\bbQ}}\frac{\rho_{s,s+t}(\{\rme^{i_j}_s\})}{\rho_{s,s+t}\big([0,1]\backslash\{\rme^{1}_s,\ldots,\rme^{i_j-1}_s\}\big)} = 1
\end{equation*}
By the uniqueness of the Eves we get $V'_j = \rme^{j}_0$ for all $j \geq 1$. The first assertion is proved.\\
We turn to Assertion \rm{(ii)}. In the \BS, this assertion is a consequence of the key property of Bertoin and Le Gall and of the first assertion since $K=\infty$ $\bbP$-a.s. We consider the \CDIPE. Let $\ccG_s:=\sigma(\rF_{r,t},-\infty < r \leq t \leq s)$, implicitly we augment this sigma-field with the $\bbP$-null sets. Observe that on the event $\{K=k\}$, the Eves $(\rme_0^j)_{j>k}$ have extinct progenies at time $s$ so that they are measurable in the sigma-field generated by $(\rF_{0,r},0 \leq r \leq s)$. Consequently $\tun_{\{K=k\}}(\rme_0^j)_{j>k}$ is $\ccG_s$-measurable. For every $k\geq 1$, we define a sigma-field on $\{K=k\}$ by setting
\[ \cB_k := \big\{A\cap\{K=k\}: A \in \sigma(\rF_{0,s},(\rme^j_s)_{j\geq 1})\big\}.\]
We claim that $(\rme^j_0)_{j> k}$ is independent of $\cB_k$. Indeed let $h$ be a bounded measurable map on $\bbD([0,1],[0,1])$ and let $g_0,g_s$ be two bounded measurable maps on $[0,1]^\bbN$ (endowed with the product sigma-field). 
\begin{eqnarray*}
\bbE\Big[g_0\big((\rme^j_0)_{j> k}\big) h(\rF_{0,s})g_s\big((\rme^j_s)_{j\geq 1}\big)\tun_{\{K=k\}} \Big] \!\!\!\!\!\!&=&\!\!\!\!\!\! \bbE\bigg[\bbE\Big[g_0\big((\rme^j_0)_{j> k}\big) h(\rF_{0,s})g_s\big((\rme^j_s)_{j\geq 1}\big)\tun_{\{K=k\}}|\ccG_s \Big]\bigg]\\
\!\!\!\!\!\!&=&\!\!\!\!\!\!\bbE\Big[g_0\big((\rme^j_0)_{j> k}\big) h(\rF_{0,s})\tun_{\{K=k\}}\Big]\bbE\big[g_s\big((\rme^j_s)_{j\geq 1}\big)\big]\\
\!\!\!\!\!\!&=&\!\!\!\!\!\!\bbE\Big[g_0\big((\rme^j_0)_{j> k}\big)\Big]\bbE\Big[h(\rF_{0,s})\tun_{\{K=k\}}\Big]\bbE\big[g_s\big((\rme^j_s)_{j\geq 1}\big)\big]
\end{eqnarray*}
where the second equality comes from the independence of the increments of a flow of bridges and the $\ccG_s$-measurability of $\tun_{\{K=k\}}(\rme^j_0)_{j> k}$, while the third equality comes from Corollary \ref{CorDistribEves} and the fact that on $\{K=k\}$ the bridge $\rF_{0,s}$ only depends on $(\rme_0^i)_{i\leq k}$ and $(\rho_{0,s}(\{\rme_0^i\}))_{i\leq k}$. The claimed independence follows. We now prove the independence of $(\rme_0^j)_{j\geq 1}$ with $\pi$. Let $m\geq 1$ be an integer, $f_1,\ldots,f_m$ be $m$ bounded measurable maps on $[0,1]$ and $1 \leq i_1 < i_2 < \ldots < i_m$ be $m$ integers. Let also $\phi$ be a bounded measurable map on $\ccP_\infty$. Let $\bbP_k:=\bbP(\cdot\,|\,\{K=k\})$. From the claim proved above, for every $k\geq 1$ we have
\[ \bbE_k\Big[\prod_{j: i_j > k} f_j(\rme_0^{i_j})\, | \, \cB_k\Big] = \bbE\Big[\prod_{j: i_j > k} f_j(\rme_0^{i_j})\Big].\]
Since on $\{K=k\}$ the r.v.~$\pi$ and $(\rme_0^{j})_{j\leq k}$ are $\cB_k$-measurable, we obtain
\begin{eqnarray*}
\bbE\Big[\prod_{j=1}^m f_j(\rme_0^{i_j})\phi(\pi)\Big] &=& \sum_{k=1}^\infty \bbE_k\Big[\prod_{j=1}^m f_j(\rme_0^{i_j})\phi(\pi)\Big]\bbP(K=k)\\
&=& \sum_{k=1}^\infty \bbE_k\Big[\bbE_k\Big[\prod_{j=1}^m f_j(\rme_0^{i_j})\phi(\pi)\, | \, \cB_k\Big]\Big]\bbP(K=k)\\
&=& \sum_{k=1}^\infty\bbE\Big[\prod_{j: i_j > k} f_j(\rme_0^{i_j})\Big]\bbE\Big[\prod_{j: i_j \leq k} f_j(\rme_0^{i_j})\phi(\pi)\tun_{\{K=k\}}\Big]
\end{eqnarray*}
Then we apply the key property of Bertoin and Le Gall together with Assertion \rm{(i)} to obtain
\begin{eqnarray*}
\bbE\Big[\prod_{j=1}^m f_j(\rme_0^{i_j})\phi(\pi)\Big] &=& \sum_{k=1}^\infty\bbE\Big[\prod_{j:i_j> k} f_j(\rme_0^{i_j})\Big]\bbE\Big[\prod_{j:i_j \leq k} f_j(\rme_0^{i_j})\Big]\bbE\Big[\phi(\pi)\tun_{\{K=k\}}\Big]\\
&=& \bbE\Big[\prod_{j=1}^m f_j(\rme_0^{i_j})\Big]\bbE\Big[\phi(\pi)\Big]
\end{eqnarray*}
Assertion \rm{(ii)} follows. This ends the proof of the proposition.
\end{proof}
\textit{Proof of Theorem \ref{ThFoP}.}  The cocycle property is trivially fulfilled if $r=s$ or $s=t$ thus we fix $r < s < t$ and prove the cocyle property. We argue deterministically on the event of $\bbP$-probability one on which $\rF_{r,t}=\rF_{s,t}\circ\rF_{r,s}$ and on which the definition of the Eves at times $r$, $s$ and $t$ hold. Fix two integers $i,j$ and let $k_i,k_j$ be the integers such that $\rF_{s,t}^{-1}(\rme_t^i)=\rme_s^{k_i}$ and $\rF_{s,t}^{-1}(\rme_t^j)=\rme_s^{k_j}$. Observe that
\[ \rF_{r,t}^{-1}(\rme_t^i) = \rF_{r,t}^{-1}(\rme_t^j) \Leftrightarrow \rF_{r,s}^{-1}\circ\rF_{s,t}^{-1}(\rme_t^i)=\rF_{r,s}^{-1}\circ\rF_{s,t}^{-1}(\rme_t^j).\]
The r.h.s.~is equivalent with
\[ k_i=k_j \text{ or } \rF_{r,s}^{-1}(\rme_s^{k_i}) = \rF_{r,s}^{-1}(\rme_s^{k_j}).\]
Consequently we have proved that $i$ and $j$ are in a same block of $\hat{\Pi}_{r,t}$ if and only if they are in a same block of $\hat{\Pi}_{s,t}$ or the indices of their respective blocks in $\hat{\Pi}_{s,t}$, say $k_i$ and $k_j$, are in a same block of $\hat{\Pi}_{r,s}$. The cocycle property follows.\\
Theorem 1~\cite{BertoinLeGall-1} ensures that for every $s\in\bbR$, $(\hat{\Pi}_{s-t,s},t\geq 0)$ is an exchangeable coalescent. Since the flow of bridges is associated with the measure $\Lambda$, we deduce that $(\hat{\Pi}_{-t,0},t\geq 0)$ is a $\Lambda$-coalescent. Another consequence of this fact is that $\hat{\Pi}_{s,t}$ is an exchangeable random partition whose distribution only depends of $t-s$.\\
Fix $s_1 < s_2 < \ldots < s_n$. If we prove that $\hat{\Pi}_{s_{n-1},s_n}$ is independent of $\hat{\Pi}_{s_1,s_2},\ldots,\hat{\Pi}_{s_{n-2},s_{n-1}}$, then an easy induction allows to prove the independence of $\hat{\Pi}_{s_1,s_2},\ldots,\hat{\Pi}_{s_{n-1},s_n}$. For every $i\in[n-1]$, $\hat{\Pi}_{s_i,s_{i+1}}$ is a $\sigma(\rF_{s_i,s_{i+1}},(\rme_{s_{i+1}}^{j})_{j\geq 1}))$ measurable r.v. and if $i\in[n-2]$ then $(\rme_{s_{i+1}}^{j})_{j\geq 1}$ is measurable in the sigma-field
\[ \sigma\Big( (\rme_{s_{i+2}}^j)_{j\geq 1}\; , \; (\rF_{s_{i+1},t}, s_{i+1} \leq t \leq s_{i+2}) \Big).\]
Consequently it is sufficient to prove that $\hat{\Pi}_{s_{n-1},s_{n}}$ is independent of
\[\sigma\big((\rme_{s_{n-1}}^{j})_{j\geq 1},(\rF_{s,t},-\infty < s \leq t \leq s_{n-1})\big).\]
Fix $k\geq 1$. Let $\ccG_{s_{n-1}}$ be the sigma-field generated by $(\rF_{s,t},-\infty < s \leq t \leq s_{n-1})$. Let $f_1,\ldots,f_k$ be $k$ bounded measurable maps on $[0,1]$ and let $h$ be a bounded measurable map on $\ccP_\infty$. For all $j_1,\ldots,j_k \geq 1$ and all $A \in \ccG_{s_{n-1}}$ we have
\begin{eqnarray*}
\bbE\Big[\prod_{l=1}^k f_l(\rme_{s_{n-1}}^{j_l})\,\tun_A\, h(\hat{\Pi}_{s_{n-1},s_n})\Big] &=& \bbE\bigg[\bbE\Big[\prod_{l=1}^k f_l(\rme_{s_{n-1}}^{j_l})\,\tun_A\, h(\hat{\Pi}_{s_{n-1},s_n})|\ccG_{s_{n-1}}\Big]\bigg]\\
&=& \bbE\bigg[\tun_{A}\,\bbE\Big[\prod_{l=1}^k f_l(\rme_{s_{n-1}}^{j_l})h(\hat{\Pi}_{s_{n-1},s_n})\Big] \bigg]\\
&=&\bbE\Big[\prod_{l=1}^k f_l(\rme_{s_{n-1}}^{j_l})\Big]\bbE\big[\tun_A\big]\bbE\Big[h(\hat{\Pi}_{s_{n-1},s_n})\Big]
\end{eqnarray*}
where we use the independence of the increments of a flow of bridges at the second line, and Proposition \ref{PropEvesFoB} at the third line. Hence $\hat{\Pi}_{s_{n-1},s_n}$ is independent of
\[\sigma\big((\rme_{s_{n-1}}^{j})_{j\geq 1},(\rF_{s,t},-\infty < s \leq t \leq s_{n-1})\big).\]
The asserted independence follows.\\
Finally the convergence in distribution of $\hat{\Pi}_{0,t}$ towards $\tO_{[\infty]}$ as $t\downarrow 0$ derives from Lemma 1 in~\cite{BertoinLeGall-1}. Since the limit is deterministic and since $\ccP_\infty$ is a metric space, the convergence in probability is immediate.\cqfd\vspace{4pt}\\
The arguments provided in the previous proof actually yield the following result.
\begin{corollary}\label{CorIndep}
For every $s\in\bbR$, the Eves $(\rme_s^i,i\geq 1)$ are independent of $(\hat{\Pi}_{s+r,s+t},0 \leq r \leq t)$.
\end{corollary}

\subsection{Proof of Theorem \ref{ThLDFoB}}
From the $\Lambda$ flow of bridges, we have defined a $\Lambda$ flow of partitions. The trajectories of the latter are not necessarily deterministic flows of partitions. However using the regularisation procedure of Subsection \ref{SubsubsectionStoFoPRegul}, we can consider a modification of this flow of partitions that we still denote $(\hat{\Pi}_{s,t},-\infty < s \leq t < \infty)$ and whose trajectories are deterministic flows of partitions $\bbP$-a.s. At any given time $s\in\bbR$, we use the flow of partitions and the Eves at time $s$ to define a measure-valued process thanks to the lookdown construction, Remark \ref{RemarkLookdown} and Corollary \ref{CorIndep}
\[ t\mapsto \ccE_{s,s+t}(\hat{\Pi},(\rme_s^{i})_{i\geq 1}) := \sum_{i\geq 1}|\hat{\Pi}_{s,s+t}(i)|\delta_{\rme_s^i} + \big(1-\sum_{i\geq 1}|\hat{\Pi}_{s,s+t}(i)|\Big)\,\ell.\]
Fix $s\in\bbR$. Both $(\ccE_{s,s+t}(\hat{\Pi},(\rme_s^{i})_{i\geq 1}),t\geq 0)$ and $(\rho_{s,s+t},t\geq 0)$ are $\bbP$-a.s.~c\`adl\`ag processes. To prove that they coincide $\bbP$-a.s., it suffices to show that for every $t\geq 0$, $\bbP$-a.s. $\ccE_{s,s+t}(\hat{\Pi},(\rme_s^{i})_{i\geq 1})=\rho_{s,s+t}$. Fix $t\geq 0$. Since $\rF_{s,s+t}$ is an exchangeable bridge, we know that $\bbP$-a.s. $\rF_{s,s+t}$ has a collection of jumps and possibly a drift component. Proposition \ref{PropEvesFoB} ensures that $\bbP$-a.s. the jump locations of $\rF_{s,s+t}$ are included in the set $\{\rme_s^i,i\geq 1\}$. Recall the definition of $\hat{\Pi}_{s,s+t}$ from $(\rme_{t+s}^i)_{i\geq 1}$ and $\rF_{s,s+t}$. Since $(\rme_{t+s}^i,i\geq 1)$ is independent of $\rF_{s,s+t}$ $\bbP$-a.s. each block of the partition $\hat{\Pi}_{s,s+t}$ either admits an asymptotic frequency equal to the size of one jump of $\rF_{s,s+t}$ or is a singleton. Since the $i$-th block is the collection of $j\geq 1$ such that $\rF_{s,s+t}^{-1}(\rme_{s+t}^j)=\rme_s^i$, we deduce that $\bbP$-a.s.~for every $i\geq 1$, $\rho_{s,s+t}(\{\rme_s^i\})=|\hat{\Pi}_{s,s+t}(i)|$. Consequently $\bbP$-a.s.
\[ \rho_{s,s+t} = \sum_{i\geq 1}|\hat{\Pi}_{s,s+t}(i)|\delta_{\rme_s^i} + \big(1-\sum_{i\geq 1}|\hat{\Pi}_{s,s+t}(i)|\big)\,\ell = \ccE_{s,s+t}(\hat{\Pi},(\rme_s^i)_{i\geq 1}).\]
The first assertion of Theorem \ref{ThLDFoB} is proved. We turn to the second assertion. Let $\hat{\Pi}'$ be another $\Lambda$ flow of partitions and for every $s\in\bbR$, let $\chi_s=(\chi_s(i),i\geq 1)$ be a sequence of r.v. taking distinct values in $[0,1]$. We assume that for all $s\in\bbR$, $\bbP$-a.s.
\[(\ccE_{s,s+t}(\hat{\Pi},(\rme_s^i)_{i\geq 1}),t\geq 0)=(\ccE_{s,s+t}(\hat{\Pi}',\chi_s),t\geq 0)=(\rho_{s,s+t},t\geq 0).\]
From Proposition \ref{PropEvesLD} and Remark \ref{RemarkDistinct} we deduce that $\bbP$-a.s. for every $i\geq 1$, $\chi_s(i)=\rme_s^i$.\\
Thanks to these almost sure equalities we get $\bbP$-a.s.~for all $i\geq 1$, all $s\in \bbQ$ and all $t\in\bbQ_+$
\begin{equation}\label{EqPiPi'}
\rho_{s,s+t}(\{\rme_s^i\}) = |\hat{\Pi}_{s,s+t}(i)| = |\hat{\Pi}_{s,s+t}'(i)|.
\end{equation}
We now work on an event $\Omega^*$ of $\bbP$-probability one on which (\ref{EqPiPi'}) holds true and on which both $\hat{\Pi}, \hat{\Pi}'$ satisfy the regularity properties of Proposition \ref{PropRegularityFoP}. We argue deterministically on $\Omega^*$. The regularity properties of Proposition \ref{PropRegularityFoP} ensure that the equalities (\ref{EqPiPi'}) do not only hold true at rational times but also at all times. To end the proof, it suffices to obtain that for all $s\in\bbR$, $\hat{\Pi}_{s-,s}=\hat{\Pi}_{s-,s}'$. This is a consequence of the next lemma which concludes the proof of Theorem \ref{ThLDFoB}.
\begin{lemma}\label{LemmaEquivalencePiJumps}
Let $\hat{\Pi}^{\times}$ denote indifferently either $\hat{\Pi}$ or $\hat{\Pi}'$. Let $I$ be a subset of $\mathbb{N}$. The following assertions are equivalent
\begin{enumerate}[(a)]
\item\label{Item1} $\hat{\Pi}^{\times}_{s-,s}$ has a unique non-singleton block $I$.
\item\label{Item2} For every $i\geq 1$, let $b(i)$ be the smallest integer such that $i \!=\! b(i)\! -\!(\#\{I\cap[b(i)]\}\!-\!1)\vee 0$. Then for every $i\ne \min I$, $(|\hat{\Pi}^{\times}_{s-,s+t}(i)|,t\geq 0) \!=\! (|\hat{\Pi}^{\times}_{s,s+t}(b(i))|,t\geq 0)$.
\end{enumerate}
\end{lemma}
\noindent The integer $b(i)$ can equivalently be defined as follows: $b(i)=i$ for every $i\leq \min I$, $b(i)=\inf\{n>b(i-1): n\notin I\}$ for every $i > \min I$.
\begin{proof}
Suppose \textit{(\ref{Item1})}. Recall that $\hat{\Pi}^{\times}_{s-,s+t} = \Coag(\hat{\Pi}^{\times}_{s,s+t},\hat{\Pi}^{\times}_{s-,s})$. Let $i$ be an integer distinct from $\min I$. The definition of the coagulation operator implies that for all $t\geq 0$, $\hat{\Pi}^{\times}_{s-,s+t}(i)=\hat{\Pi}^{\times}_{s,s+t}(b(i))$. Consequently the identity on the asymptotic frequencies follows.\\
Suppose \textit{(\ref{Item2})}. Since the trajectories of $\hat{\Pi}^{\times}$ are deterministic flows of partitions, we know that $\hat{\Pi}^{\times}_{s-,s}$ is a partition with at most one non-singleton block. For $i\ne\min I$ we stress that the equality $(|\hat{\Pi}^{\times}_{s-,s+t}(i)|,t\geq 0) \!=\! (|\hat{\Pi}^{\times}_{s,s+t}(b(i))|,t\geq 0)$ implies that $\hat{\Pi}^{\times}_{s-,s}(i)=\{b(i)\}$. We first prove that $b(i) \in \hat{\Pi}^{\times}_{s-,s}(i)$, let us denote by $J$ the latter block. In the \CDIPE\ for $j< j'$ the process $(|\hat{\Pi}^{\times}_{s,s+t}(j)|,t\geq 0)$ reaches $0$ strictly after $(|\hat{\Pi}^{\times}_{s,s+t}(j')|,t\geq 0)$ so that the process $(|\hat{\Pi}^{\times}_{s-,s+t}(i)|,t\geq 0)$ reaches $0$ at the same time as $(|\hat{\Pi}^{\times}_{s,s+t}(\min J)|,t\geq 0)$ reaches $0$. Consequently $\min J = b(i)$. In the \BS\ for every $j$ the ratio $|\hat{\Pi}^{\times}_{s,s+t}(j)|/\sum_{j' \geq j}|\hat{\Pi}^{\times}_{s,s+t}(j')|$ goes to $1$ as $t\rightarrow\infty$. Consequently we have $|\hat{\Pi}^{\times}_{s-,s+t}(i)| \sim |\hat{\Pi}^{\times}_{s,s+t}(\min J)|$ as $t\rightarrow\infty$. If $\min J < b(i)$ then $|\hat{\Pi}^{\times}_{s,s+t}(b(i))|$ is negligible compared with $|\hat{\Pi}^{\times}_{s-,s+t}(i)|$ as $t\rightarrow\infty$ while if $\min J > b(i)$ the converse occurs. Therefore $\min J = b(i)$. We now prove that $J = \{b(i)\}$. Suppose that there is another element $k\in J$. Then we have $|\hat{\Pi}^{\times}_{s-,s+t}(i)| \geq |\hat{\Pi}^{\times}_{s,s+t}(b(i))|+|\hat{\Pi}^{\times}_{s,s+t}(k)|$ for all $t\geq 0$. The process $(|\hat{\Pi}^{\times}_{s,s+t}(k)|,t\geq 0)$ cannot be null at all times since otherwise the Eve $\rme_s^k$ would not be well-defined. We then deduce that the equality $|\hat{\Pi}^{\times}_{s-,s+t}(i)| = |\hat{\Pi}^{\times}_{s,s+t}(b(i))|$ is not satisfied at all times, which is contradictory. Hence $J = \{b(i)\}$.\\
We have proved that for every $i\ne\min I$, $\hat{\Pi}^{\times}_{s-,s}(i)=\{b(i)\}$. This suffices to recover completely the partition $\hat{\Pi}^{\times}_{s-,s}$ and to assert that $\hat{\Pi}^{\times}_{s-,s}(\min I)=I$.
\end{proof}

\section{Appendix}\label{SectionAppendix}
\subsection{Proof of Proposition \ref{PropRegularityFoP}}\label{AppendixRegularityFoP}
Recall that $\cP$ is a $\Lambda$ lookdown graph. Without loss of generality, we can assume that for every $\omega\in\Omega$, $\cP(\omega)$ is a deterministic lookdown graph. Since $\hat{\Pi}=J^{-1}(\cP)$, we deduce that for every $\omega\in\Omega$, $\hat{\Pi}(\omega)$ is a deterministic flow of partitions. The difficulty of the proof lies in the regularity of the asymptotic frequencies.
\begin{lemma}
$\bbP$-a.s. for all $s\in\bbR$, $\hat{\Pi}_{s-,s}$ admit asymptotic frequencies.
\end{lemma}
\begin{proof}
Fix $n\geq 1$. Consider the set $\{(s,\hat{\Pi}_{s-,s}):\hat{\Pi}_{s-,s}^{[n]}\ne\tO_{[n]}\}$. A simple argument shows that this is a Poisson point process on $\bbR\times\ccP_\infty$ with intensity $dt\times\phi(m_n)$ where $\phi(m_n)$ is the pushforward of the finite measure
\[ m_n:= (\mu_\Lambda+\mu_K)\big(\,.\,\cap\cS_n\big)\]
through the map $\phi:\cS_\infty\rightarrow\cP_\infty$ that associates to an element $v\in\cS_\infty$ the partition with a unique non-singleton block equal to $\{i\geq 1: v(i)=1\}$. If we enumerate the points in $\{(s,\hat{\Pi}_{s-,s}):\hat{\Pi}_{s-,s}^{[n]}\ne\tO_{[n]}\}$ by increasing absolute time coordinate, then we get a collection $(t_i,\pi_i)_{i\geq 1}$ where $(\pi_i)_{i\geq 1}$ is a sequence of i.i.d.~exchangeable $\ccP_\infty$-valued r.v.~with distribution $\phi(m_n)/m_n(\cS_n)$. Consequently almost surely for every $i\geq 1$ the partition $\pi_i$ admits asymptotic frequencies. We deduce that with probability one all the $\hat{\Pi}_{s-,s}$ such that $\hat{\Pi}_{s-,s}^{[n]}\ne\tO_{[n]}$ admit asymptotic frequencies. Taking a countable intersection on $n\geq 1$, we get the asserted result.
\end{proof}
\begin{lemma}\label{LemmaFreqCadlag}
With probability one, for every rational value $s\in\bbQ$ and every integer $i\geq 1$, the process $(|\hat{\Pi}_{s,s+t}(i)|,t\geq 0)$ is well-defined and c\`adl\`ag.
\end{lemma}
\begin{proof}
It is sufficient to show the lemma for $s=0$, and then to take a countable intersection of events of probability one. Let us first consider regime \Fin. Let $\xi_0(i),i\geq 1$ be a sequence of i.i.d.~uniform$[0,1]$ r.v.~and consider the lookdown process $(\xi_t(i),t\geq 0), i\geq 1$ constructed from $\hat{\Pi}$ and $\xi_0(i),i\geq 1$ according to Definition \ref{DefLDProcess2}. Recall that this lookdown process is almost surely equal to that obtained with Definition \ref{DefLDProcess} using the $\Lambda$ lookdown graph $\cP$. Let us use the results recalled in the proof of Theorem \ref{ThDK}. In particular, we use the notation
\[ t\mapsto\Xi_t^{[n]}:= \frac{1}{n}\sum_{i=1}^n\delta_{\xi_t(i)},\]
introduced there. In the proof of Theorem \ref{ThDK}, we saw that almost surely $\Xi^{[n]}$ converges to a process $\Xi$ which is a c\`adl\`ag $\Lambda$-Fleming-Viot process and that for every continuous function $f$ on $[0,1]$, almost surely the process $(\int_{[0,1]}f(x)\Xi_t^{[n]}(dx),t\geq 0),n\geq 1$ is a Cauchy sequence in $\bbD([0,\infty),\bbR)$ endowed with the distance
\[ d_u(X,Y) = \int_{[0,\infty)}\!\!\!e^{-t}\sup_{s\leq t}1\wedge|X_s-Y_s|\,dt .\]
Let $f_k,k\geq 1$ be a sequence of continuous functions from $[0,1]$ into $[0,\infty)$ such that for every integer $m\geq 1$ and every integer $2\leq i < 2^m$, there exists $k\geq 1$ such that the function $f_k$ equals $1$ on $((i-1)2^{-m},i2^{-m}]$ and vanishes on $[0,1]\backslash((i-2)2^{-m},(i+1)2^{-m})$. Taking a countable union of events, we know that $\bbP$-a.s.~for all $k\geq 1$ the sequence of processes $(\int_{[0,1]}f_k(x)\Xi_t^{[n]}(dx),t\geq 0),n\geq 1$ converges for the distance $d_u$. By construction, we have $\bbP$-almost surely for every $i\leq n$ and all $t\geq 0$, $\Xi_t^{[n]}(\{\xi_0(i)\})=|\hat{\Pi}_{0,t}^{[n]}(i)|$. We now argue deterministically on an event of probability one on which the last two assertions hold true. Let $i\geq 1$ and $t_0>0$ be given. Let $i_0$ be the number of blocks in $\hat{\Pi}_{0,t_0}$. There exists $k\geq 1$ such that $f_k(\xi_0(j))=0$ for all $j\in \{1,\ldots,i_0\}\backslash\{i\}$ and $f_k(\xi_0(i))=1$. Consequently for every $t\geq t_0$ $|\hat{\Pi}_{0,t}^{[n]}(i)|=\int_{[0,1]}f_k(x)\Xi_t^{[n]}(dx)$. We deduce the convergence of $(|\hat{\Pi}_{0,t}^{[n]}(i)|,t\geq t_0)$ as $n$ goes to infinity for the metric $d_u$ restricted to $[t_0,\infty)$. Taking $t_0$ arbitrarily small, this yields the convergence of $(|\hat{\Pi}_{0,t}^{[n]}(i)|,t>0)$ for $d_u$ restricted to $(0,\infty)$. Since $\Xi$ is c\`adl\`ag and $\Xi_0$ is the Lebesgue measure, $0=\Xi_0(\{\xi_0(i)\})\geq \varlimsup_{t\downarrow 0}\Xi_t(\{\xi_0(i)\})=\varlimsup_{t\downarrow 0}|\hat{\Pi}_{0,t}(i)|$. Since $\hat{\Pi}_{0,0}(i)=\{i\}$, we deduce that $(|\hat{\Pi}_{0,t}(i)|,t\geq 0)$ is c\`adl\`ag. This ends the proof for regime \Fin.\\

We now turn to regimes \Disc, \Dust\ and \Inf. We need to adapt the proofs of Lemma 3.4 and 3.5 in~\cite{DK99} to our context but the arguments are essentially the same. We fix $i\geq 1$ and we set $\pi([i]):=\pi(1)\cup\pi(2)\ldots\cup\pi(i)$ to denote the union of the $i$ first blocks of a partition $\pi$. Recall that $\pi^{[m]}$ stands for the restriction of a partition $\pi$ to the $m$ first integers. We define $U_t:= \sum_{s\leq t}\sum_{j\geq 1}|\hat{\Pi}_{s-,s}(j)|^2$, that is, the sum of the square of the asymptotic frequencies of the non-singleton blocks of the partitions $\{\hat{\Pi}_{s-,s}:\hat{\Pi}_{s-,s}\ne\tO_{[\infty]}\}$ up to time $t$. Recall the process $(L_t(i),t\geq 0)$ from Definition \ref{DefLt}. For every $m\geq 1$, we set
\[ \cH_t^m := \sigma \{ |\hat{\Pi}_{0,s}^{[m]}([i])|, U_s , L_s(i) ;s \leq t  \}.\]
Let $\epsilon>0$ and let $q\geq 1$ be a positive integer. We stress that there exists $\eta=\eta(\epsilon)>0$ such that for every $m\geq \ell \geq 2q\epsilon^{-1}$ and every $\cH_t^m$-stopping time $\alpha$, we have:
\begin{equation}\label{Eq:BoundStopTime}
\bbP\Big( \big| |\hat{\Pi}_{0,\alpha}^{[m]}([i])| - |\hat{\Pi}_{0,\alpha}^{[\ell]}([i])| \big| \geq \epsilon \,\Big| \, L_{\alpha}(i) \leq q \Big) \leq 2e^{-\eta(\ell-q)}.
\end{equation}
We postpone the proof of this bound below and carry on the proof of the lemma. Let $T>0$. We define a collection of $\cH_t^m$-stopping times to which this estimate can be applied:
\[\alpha_1^m := \inf\{t\geq 0: U_t > \ell^{-4}\} \wedge \ell^{-4} \wedge T,\]
and, recursively for every $k\geq 1$:
\begin{align*}
\alpha_{k+1}^m &:= \inf\{t\geq 0: U_t > U_{\alpha_k^m}+ \ell^{-4}\} \wedge (\alpha_k^m+\ell^{-4}) \wedge T,\\
\tilde{\alpha}_{k}^m &:= \inf\{t\geq \alpha_k^m: \big| |\hat{\Pi}_{0,t}^{[m]}([i])| - |\hat{\Pi}_{0,\alpha_k^m}^{[m]}([i])| \big| \geq 5\epsilon \} \wedge T.
\end{align*}
Let $k_\ell:= 2(c+T)\ell^4$ and observe that $\bbP(\alpha_{k_\ell}^m< T \,,\, U(\alpha_{k_\ell}^m) < c)=0$. Then, we set
\[ H_k^m := \Big| |\hat{\Pi}_{0,\alpha_k^m}^{[m]}([i])| - |\hat{\Pi}_{0,\alpha_k^m}^{[\ell]}([i])| \Big| \vee \Big| |\hat{\Pi}_{0,\tilde{\alpha}_k^m}^{[m]}([i])| - |\hat{\Pi}_{0,\tilde{\alpha}_k^m}^{[\ell]}([i])| \Big|, \]
and we have
\begin{align*}
\bbP(\sup_{k\leq k_\ell} H_k^m \geq \epsilon \,;\, L_{\alpha_{k_l}^m}(i)\leq q) &\leq \bbP\Big(\underset{k\leq k_\ell}{\cup}\{H_k^m \geq \epsilon \,;\, L_{\alpha_{k}^m}(i)\leq q\}\Big)\\
&\leq \sum_{k\leq k_\ell}\bbP\big(H_k^m \geq \epsilon \,|\, L_{\alpha_{k}^m}(i)\leq q\big)\\
&\leq 8(c+T)\ell^4e^{-\eta (\ell-q)},
\end{align*}
where we have bounded $\bbP(L_{\alpha_{k}^m}(i)\leq q)$ by $1$ and we have used (\ref{Eq:BoundStopTime}). Now, we define
\[ K := \max_{k\leq k_\ell} \sup_{\alpha_k^m \leq t < \alpha_{k+1}^m} \Big| |\hat{\Pi}_{0,t}^{[\ell]}([i])| - |\hat{\Pi}_{0,\alpha_k^m}^{[\ell]}([i])| \Big|.\]
This corresponds to the r.v.~$K_1$ in~\cite{DK99}, in our context the r.v.~$K_2$ is not needed since there is no mutation. Observe that $K$ is bounded by the maximum over $k\leq k_\ell$ of $N(\alpha_k^m,\alpha_{k+1}^m)/\ell$, where $N(\alpha_k^m,\alpha_{k+1}^m)$ is the number of new individuals appeared among the $\ell$ first levels on the time interval $(\alpha_k^m,\alpha_{k+1}^m)$. Let us denote by $\zeta_i$ the number of individuals involved among the $\ell$ first levels at the reproduction time $t_i$ and let $p_i$ denote the frequency of the non-singleton block corresponding to this reproduction event. By definition of the stopping times, $\sum_{\alpha_k^m < t_i < \alpha_{k+1}^m} p_i^2 \leq \ell^{-4}$. We have $N(\alpha_k^m,\alpha_{k+1}^m)=\sum_{\alpha_k^m < t_i < \alpha_{k+1}^m}(\zeta_i-1)_+$, and $\bbP(\zeta_i=r\,|\,p_i)=\binom{l}{r} p_i^r(1-p_i)^{\ell-r}$ for every $r\in\{0,\ldots,\ell\}$. Using the inequalities
\begin{align*}
(z-1)_+^2 &\leq z(z-1) \;\;,\;\; (z-1)_+^3 \leq (z-1)_+^4 \leq 3z(z-1)(z-2)(z-3)+3z(z-1),
\end{align*}
and conditioning under the collection of points $(p_i)_i$, a simple calculation yields
\begin{align*}
\bbE[ N(\alpha_k^m,\alpha_{k+1}^m)^4 ] &\leq  \bbE\big[\sum_i(\zeta_i-1)_+^4\big]+ 4 \bbE\big[\sum_i(\zeta_i-1)_+^3\big]\,\bbE\big[\sum_i(\zeta_i-1)_+\big]  \\
&\;\;+ 6\, \bbE\big[\sum_i(\zeta_i-1)_+^2\big]^2 + 6\, \bbE\big[\sum_i(\zeta_i-1)_+^2\big]\,\bbE\big[\sum_i(\zeta_i-1)_+\big]^2\\
&\;\;+ \bbE\big[\sum_i(\zeta_i-1)_+\big]^4 \\
&\leq 84\,\ell^{-2}.
\end{align*}
Putting these arguments together, we get
\[ \bbP(K \geq 2\epsilon ) \leq \sum_{k\leq k_\ell} \frac{\bbE[ N(\alpha_k^m,\alpha_{k+1}^m)^4 ]}{2^4\epsilon^4\ell^4} \leq \frac{84 k_\ell}{2^4 \epsilon^4 \ell^4} \ell^{-2} \leq 11(c+T)\epsilon^{-4}\ell^{-2}.\]
On the event where $\sup_{k\leq k_\ell} H_k^m \leq \epsilon$ and $K \leq 2\epsilon$, we have necessarily for every $k\leq k_\ell$, $\tilde{\alpha}_{k}^m \geq \alpha_{k+1}^m$ and therefore $\sup_{t\in[\alpha_k^m,\alpha_{k+1}^m)}\big| |\hat{\Pi}_{0,t}^{[m]}([i])| - |\hat{\Pi}_{0,\alpha_k^m}^{[m]}([i])| \big| < 5\epsilon$. Consequently for all $m\geq \ell \geq 2q\epsilon^{-1}$:
\begin{align}\label{Eq:Boundml}
&\bbP\Big( \sup_{t\in[0,T]} \big| |\hat{\Pi}_{0,t}^{[m]}([i])| -|\hat{\Pi}_{0,t}^{[\ell]}([i])| \big| \geq 8\epsilon \,;\, U_T < c \,;\, L_T(i) \leq q\Big) \nonumber\\
\leq\; &\bbP\Big( \sup_{t\leq \alpha_{k_\ell}^m} \big| |\hat{\Pi}_{0,t}^{[m]}([i])| -|\hat{\Pi}_{0,t}^{[\ell]}([i])| \big| \geq 8\epsilon \,;\, U_T < c \,;\, L_T(i) \leq q \Big)\nonumber\\
\leq\; &\bbP(\sup_{k\leq k_\ell} H_k^m \geq \epsilon\,;\, L_{\alpha_{k_\ell}^m}(i) \leq q) + \bbP(K\geq 2\epsilon)\nonumber\\
\leq\; & 8(c+T)\ell^4e^{-\eta (\ell-q)} + 11(c+T)\epsilon^{-4}\ell^{-2}=:\delta_\ell.
\end{align}
Let us now show that
\[ \bbP\Big(\varlimsup_{\ell \rightarrow \infty} \underset{m\geq \ell}{\cup} \big\{ \sup_{t\in[0,T]} \big| |\hat{\Pi}_{0,t}^{[m]}([i])| -|\hat{\Pi}_{0,t}^{[\ell]}([i])| \big| > 16\epsilon \,;\, U_T < c \,;\, L_T(i) \leq q\big\}\Big) = 0.\]
Since our processes are c\`adl\`ag, this is equivalent to show that
\[ \bbP\Big(\varlimsup_{\ell \rightarrow \infty} \underset{m\geq \ell}{\cup} \underset{t\in[0,T]\cap\bbQ}{\cup} \big\{ \big| |\hat{\Pi}_{0,t}^{[m]}([i])| -|\hat{\Pi}_{0,t}^{[\ell]}([i])| \big| > 16\epsilon \,;\, U_T < c \,;\, L_T(i) \leq q \big\}\Big) = 0.\]
The exchangeability of the partitions $\hat{\Pi}_{0,t}$ ensures that $\bbP$-a.s.~for all $t\in[0,T]\cap\bbQ$, $|\hat{\Pi}_{0,t}^{[m]}([i])|\rightarrow|\hat{\Pi}_{0,t}([i])|$ as $m\rightarrow\infty$. Therefore the left hand side of the latter equation is bounded above by
\begin{align}\label{Eq:BoundUnifml}
&\bbP\big(\varlimsup_{\ell \rightarrow \infty} \underset{t\in[0,T]\cap\bbQ}{\cup} \big\{ \big| |\hat{\Pi}_{0,t}^{[\ell]}([i])| -|\hat{\Pi}_{0,t}([i])| \big| > 8\epsilon \,;\, U_T < c \,;\, L_T(i) \leq q \big\}\big) \nonumber\\
=\; & \bbP\big(\varlimsup_{\ell \rightarrow \infty} \underset{t\in[0,T]\cap\bbQ}{\cup}\varliminf_{m\rightarrow\infty} \big\{ \big| |\hat{\Pi}_{0,t}^{[\ell]}([i])| -|\hat{\Pi}_{0,t}^{[m]}([i])| \big| > 8\epsilon \,;\, U_T < c \,;\, L_T(i) \leq q \big\}\big)\nonumber\\
\leq\; & \bbP\big(\varlimsup_{\ell \rightarrow \infty} \varliminf_{m\rightarrow\infty} \underset{t\in[0,T]\cap\bbQ}{\cup} \big\{ \big| |\hat{\Pi}_{0,t}^{[\ell]}([i])| -|\hat{\Pi}_{0,t}^{[m]}([i])| \big| > 8\epsilon \,;\, U_T < c \,;\, L_T(i) \leq q \big\}\big)\nonumber\\
\leq\; & \bbP\big(\varlimsup_{\ell \rightarrow \infty} \varliminf_{m\rightarrow\infty}  \big\{ \sup_{t\in [0,T]}\big| |\hat{\Pi}_{0,t}^{[\ell]}([i])| -|\hat{\Pi}_{0,t}^{[m]}([i])| \big| > 8\epsilon \,;\, U_T < c \,;\, L_T(i) \leq q \big\}\big).
\end{align}
By (\ref{Eq:Boundml}), we have $\bbP\big(\varliminf_{m\rightarrow\infty}  \big\{ \sup_{t\in [0,T]}\big| |\hat{\Pi}_{0,t}^{[\ell]}([i])| -|\hat{\Pi}_{0,t}^{[m]}([i])| \big| > 8\epsilon \,;\, U_T < c \,;\, L_T(i) \leq q \big\}\big) \leq \delta_\ell$. Since $\sum_\ell\delta_\ell < \infty$, the Borel-Cantelli lemma shows that the right hand side of (\ref{Eq:BoundUnifml}) vanishes. Consequently, for $\bbP$-almost all $\omega\in\{U_T<c \,;\, L_T(i) \leq q\}$, there exists $\ell_0$ such that for all $m\geq \ell\geq \ell_0$
\[ \sup_{t\leq T} \big| |\hat{\Pi}_{0,t}^{[m]}([i])| -|\hat{\Pi}_{0,t}^{[\ell]}([i])| \big| \leq 16\epsilon.\]
Taking $\epsilon$ arbitrarily small, we deduce that for $\bbP$-almost all $\omega\in\{U_T<c \,;\, L_T(i) \leq q\}$, $(|\hat{\Pi}_{0,t}^{[m]}([i])|,t\in[0,T])_m$ is a Cauchy sequence in the space $\bbD([0,T],\bbR)$ endowed with the uniform norm. This space is complete, therefore this sequence converges. Taking $c$ and $q$ arbitrarily large, we deduce that $\bbP$-almost surely for all $i\geq 1$, the sequences $(|\hat{\Pi}_{0,t}^{[m]}([i])|,t\in[0,T])_m$ converge in $\bbD([0,T],\bbR)$ endowed with the uniform norm. This ensures that $\bbP$-almost surely for every $i\geq 1$, the process $(|\hat{\Pi}_{0,t}(i)|,t\in[0,T])= (|\hat{\Pi}_{0,t}([i])|,t\in[0,T]) - (|\hat{\Pi}_{0,t}([i-1])|,t\in[0,T])$ is a well-defined c\`adl\`ag process. Since $T$ can be taken arbitrarily large, the lemma follows.

\smallskip
It remains to prove (\ref{Eq:BoundStopTime}). Fix $t\geq 0$, $q \in [m]$ and $A\in\cH_t^m$. There exists a measurable map $f$ such that $A=\{f(|\hat{\Pi}_{0,s}^{[m]}([i])|, U_s , L_s(i) ;s \leq t)=1\}$. Let $\sigma$ be a permutation of $\bbN$ that maps $\{q+1,\ldots,m\}$ into itself and is the identity on $\{1,\ldots,q\}\cup\{m+1,m+2,\ldots\}$. For any given $\pi\in\ccP_\infty$, we define $\sigma(\pi)$ as the partition whose blocks are given by the subsets $\{\sigma^{-1}(j):j\in\pi(i)\},i\geq 1$. We first show that, conditionally given $\{L_t(i)=q\}\cap A$, the distribution of $\hat{\Pi}_{0,t}$ is invariant under $\sigma$. To that end, we define on the same probability space a $\Lambda$ flow of partitions $\hat{\Pi}'$ on $[0,t]$ - together with its associated r.v. $U', L'$ and the associated event $A':=\{f(|(\hat{\Pi}')_{0,s}^{[m]}([i])|, U_s' , L_s'(i) ;s \leq t)=1\}$ - such that almost surely $\{L_t(i)=q\}\cap A=\{L'_t(i)=q\}\cap A'$ and such that, on this event, $\sigma(\hat{\Pi}_{0,t})=\hat{\Pi}_{0,t}'$. This will ensure that on $\{L_t(i)=q\}\cap A$, the distribution of $\hat{\Pi}_{0,t}$ is invariant under $\sigma$. For any given $n\geq m$, we let $p$ be the random number of reproduction events affecting at least two levels among the $n$ first on the time interval $[0,t]$ and let $(t_1,\pi_1),\ldots,(t_p,\pi_p) \in [0,t]\times\ccP_\infty$ be these events ordered by increasing times. We also let $t_0=0$ and $t_{p+1}:=\infty$. Set $\sigma_p:=\sigma$. Following the proof of Lemma 4.3 of Bertoin~\cite{BertoinRandomFragmentation}, we define recursively $\sigma_j$ as the permutation of $\bbN$ s.t.~the $i$-th block of $\sigma_{j+1}(\pi_{j+1})$ is $\sigma_{j+1}^{-1}(\pi_{j+1}(\sigma_j(i)))$, for every $j\in\{0,1,\ldots,p-1\}$. For every $s\in(0,t]$, we let $j$ be the unique integer such that $s\in[t_j,t_{j+1})$ and we set
\[ \hat{\Pi}'_{s-,s}:=\sigma_j(\hat{\Pi}_{s-,s}). \]
This definition ensures that the asymptotic frequencies of the reproduction events are not modified, so $U'_s=U_s$ for all $s\in[0,t]$. Since $\sigma_p=\sigma$ is the identity on $\{m+1,m+2,\ldots\}$, we deduce that for every $j\leq p$, $\sigma_j$ is also the identity on $\{m+1,m+2,\ldots\}$, so that the definition of $\hat{\Pi}'$ does not depend on $n\geq m$. Now observe that, for every $j\in \{1,2,\ldots,p\}$, $\sigma_j$ is independent of $(\pi_1,\ldots,\pi_j)$. Since $(\pi_1^{[n]},\ldots,\pi_p^{[n]})$ are $p$ independent exchangeable random partitions of $[n]$, we deduce that the same holds for $(\sigma_1(\pi_1^{[n]}),\ldots,\sigma_p(\pi_p^{[n]}))$ and therefore $\hat{\Pi}'$ has the same distribution as $\hat{\Pi}$ on $[0,t]$. Since $\sigma$ is the identity on $\{1,\ldots,q\}$ and is a permutation of $\{q+1,\ldots,m\}$ into itself, we deduce that on the event $\{L_t(i)=q\}$, we have $|(\hat{\Pi}')_{0,s}^{[m]}([i])|=|\hat{\Pi}_{0,s}^{[m]}([i])|$ and $L'_s(i)=L_s(i)$ for all $s\in[0,t]$. This ensures that almost surely $\{L_t(i)=q\}\cap A=\{L'_t(i)=q\}\cap A'$. By construction, $\sigma(\hat{\Pi}_{0,t})=\hat{\Pi}'_{0,t}$. We have proved that, conditionally given $\{L_t(i)=q\}\cap A$, the distribution of $\hat{\Pi}_{0,t}$ is invariant under $\sigma$.\\
Let $\alpha$ be a discrete $\cH_t^m$-stopping time, that is, an $\cH_t^m$-stopping time that takes a countable collection of values. The above result ensures that conditionally given $L_{\alpha}(i)=q$, the distribution of $\hat{\Pi}_{0,\alpha}$ is invariant under $\sigma$. Using a classical approximation, this remains true if we remove the hypothesis that $\alpha$ is discrete. Therefore, conditionally given $L_{\alpha}(i)=q$, the random variables
\[\tun_{\big\{ q + 1 \,\in\,  \hat{\Pi}_{0,\alpha}^{[m]}([i])\big\}}, \tun_{\big\{ q + 2 \,\in\,  \hat{\Pi}_{0,\alpha}^{[m]}([i])\big\}}, \ldots, \tun_{\big\{ m \,\in\,  \hat{\Pi}_{0,\alpha}^{[m]}([i])\big\}}\]
are exchangeable. Fix $\epsilon>0$. We define $M_k:=k^{-1}\sum_{j=1}^k \tun_{\{ q + j \in  \hat{\Pi}_{0,\alpha}^{[m]}([i])\}}$. On the event $\{L_{\alpha}(i)=q\}$, we apply Lemma A.2 in~\cite{DK99}
which ensures the existence of a constant $\eta=\eta(\epsilon)>0$ such that
\[ \bbP\big( |M_{\ell-q} - M_{m-q}| \geq \epsilon/2 \,|\, L_{\alpha}(i)=q \big) \leq 2 e^{-\eta (\ell-q)} .\]
Since $m|\hat{\Pi}_{0,\alpha}^{[m]}([i])|=q+(m-q)M_{m-q}$ and $\ell|\hat{\Pi}_{0,\alpha}^{[\ell]}([i])|=q+(\ell-q)M_{\ell-q}$, we obtain
\[ \big| |\hat{\Pi}_{0,\alpha}^{[m]}([i])| - |\hat{\Pi}_{0,\alpha}^{[\ell]}([i])| \big| \leq q(\frac{1}{\ell}-\frac{1}{m}) + |M_{\ell-q} - M_{m-q}|.\]
Since $q \leq \frac{\ell\epsilon}{2}$, the first term on the right hand side is smaller than $\epsilon/2$ and therefore (\ref{Eq:BoundStopTime}) follows.
\end{proof}
\noindent To end the proof, we need to distinguish three cases.\\
1- \Fin. The process $(\hat{\Pi}_{t-r,t},r\geq 0)$ starts with infinitely many blocks, and immediately after time $0$, comes down from infinity. Note that this property holds a priori on an event of probability one that depends on $t$, but the cocycle property allows to assert that the coming down from infinity holds for all $t$ simultaneously on a same event of probability one. The jumps of $(\hat{\Pi}_{t-r,t},r\geq 0)$ are finitely many on any compact interval of $(0,\infty)$ since the jump rate of a coalescent with a finite number of blocks is finite. Thus $\bbP$-a.s. for all $s\in\bbR, t> 0$, there exists a rational value $q(\omega)=q \in (s,s+t)$ such that $\hat{\Pi}_{s,s+t}=\hat{\Pi}_{q,s+t}$ and the existence of asymptotic frequencies follows from the rational case. The limit
\begin{equation*}
\lim\limits_{\epsilon\downarrow 0}|\hat{\Pi}_{s+\epsilon,s+t}(i)|=|\hat{\Pi}_{s,s+t}(i)|,\ \forall i\in\mathbb{N}
\end{equation*}
is then obvious. Similarly, there exists a rational value $p(\omega)=p < s$ such that $\hat{\Pi}_{p,s+t}=\hat{\Pi}_{s-,s+t}$. It implies the existence of asymptotic frequencies for the latter along with the limit
\begin{equation*}
\lim\limits_{\epsilon\downarrow 0}|\hat{\Pi}_{s-\epsilon,s+t}(i)|=|\hat{\Pi}_{s-,s+t}(i)|,\ \forall i\in\mathbb{N}
\end{equation*}
The same kind of arguments apply to show the regularity when $t$ varies.\\
Finally the process $(\sum_{i\geq 1}|\hat{\Pi}_{s,s+t}(i)|,t\geq 0)$ is c\`adl\`ag on $(0,\infty)$ since at any time $t > 0$ only finitely many $|\hat{\Pi}_{s,s+t}(i)|$'s are non-null and since each of these are c\`adl\`ag processes in $t$.\\
2- \Disc\ and \Dust. On any compact interval of time, only a finite number of reproduction events choose any given level as the parent (recall that $\int_{(0,1)}\!\!u\,\nu(du)<\infty$), therefore $r\mapsto\hat{\Pi}_{t-r,t}(i)$ evolves at discrete times for any given $i\geq 1$. The arguments of the previous regime can therefore be applied by considering $\hat{\Pi}_{s,t}(i)$ instead of $\hat{\Pi}_{s,t}$ in order to show the regularities in frequency.\\
Showing that $(\sum_{i\geq 1}|\hat{\Pi}_{s,s+t}(i)|,t> 0)$ is c\`adl\`ag on $(0,\infty)$ is more involved. This will be a consequence of a uniform bound on the block frequencies. More precisely we introduce for every $t\geq 0$
\[ N_{0,t}(n) := \sum_{s\in (0,t]}\#\{i\in [n]: v_s(i)=1\}.\]
One can show that the sequence of processes $(N_{0,t}(n)/n,t\geq 0)$ converges $\bbP$-a.s.~to a c\`adl\`ag process $(Y_{0,t},t\geq 0)$ for the uniform norm on $\bbD([0,\infty),\bbR)$. Then a simple argument based on the transitions of the lookdown process ensures that for every $i,n\geq 1$ and every $s,t,\epsilon \geq 0$
\[ \frac{1}{n}\,\big|\sum_{j=1}^{i}\#\{\hat{\Pi}_{s,s+t+\epsilon}(j)\cap[n]\}-\sum_{j=1}^i\#\{\hat{\Pi}_{s,s+t}(j)\cap[n]\}\big| \leq \frac{N_{0,s+t+\epsilon}(n)-N_{0,s+t}(n)}{n}\]
so that for every $i\geq 1$
\[ \Big|\sum_{j=1}^i|\hat{\Pi}_{s,s+t+\epsilon}(j)|-\sum_{j=1}^i|\hat{\Pi}_{s,s+t}(j)|\Big| \leq Y_{0,s+t+\epsilon}-Y_{0,s+t}.\]
Taking the limit as $i\rightarrow\infty$, one gets the asserted right continuity of $(\sum_{i\geq 1}|\hat{\Pi}_{s,s+t}(i)|,t>0)$ whenever $s\geq 0$. This can be extended to all $s\in\bbR$ by taking a countable union of events of probability one. A similar argument yields the existence of left limits.\\
3- \Inf. In this regime, all the partitions have infinitely many blocks and no singleton. Let us show that with probability one for every $s\in\bbQ$ the process $(\sum_{i\geq 1}|\hat{\Pi}_{s,s+t}(i)|,t>0)$ is constant equal to $1$. Let $\xi_0(i),i\geq 1$ be an independent sequence of i.i.d.~uniform$[0,1]$ r.v.~and let $(\Xi_t,t\geq 0)$ be the $\Lambda$-Fleming-Viot process constructed from the $\xi_0(i)$'s and the $\Lambda$ lookdown graph $\cP$, as in Theorem \ref{ThDK}. Recall that $\hat{\Pi}=J^{-1}(\cP)$. By construction, $\bbP$-almost surely for all $t\in\bbQ_+$ and all $i\geq 1$ we have $\Xi_t(\{\xi_0(i)\})=|\hat{\Pi}_{0,t}(i)|$. Moreover, by Lemma \ref{LemmaFreqCadlag}, $\bbP$-almost surely for all $i\geq 1$, the processes $(|\hat{\Pi}_{0,t}(i)|,t\geq 0)$ are c\`adl\`ag. Fix $\eta \in \bbQ_+^*$ and let $a_1,a_2,\ldots$ be the atoms of $\Xi_\eta$ in the decreasing order of their masses. The set $\{a_i,i\geq 1\}$ coincides with the set $\{\xi_0(i),i\geq 1\}$. For every $n\geq 1$ and $i\geq 1$, let $\alpha_{n,i}:=\inf\{t\geq \eta: \sum_{1\leq j \leq i} \Xi_{t}(\{a_j\}) \leq 1-2^{-n}\}$. This is a stopping time in the filtration of $\Xi$, and therefore $\alpha_n:=\sup_{i\geq 1}\alpha_{n,i}$ is also a stopping time in the filtration of $\Xi$. Using Proposition 3.2 in~\cite{DK99}, we deduce that $\xi_{\alpha_n}(i),i\geq 1$ is exchangeable on the event $\{\alpha_n<\infty\}$. By construction, we have $\bbP$-a.s.~$\alpha_n=\inf\{t\geq \eta: \sum_{i\geq 1}|\hat{\Pi}_{0,t}(i)| \leq 1-2^{-n}\}$. Furthermore for every integers $i,j$ we have the equivalence $\xi_{\alpha_n}(i)=\xi_{0}(j) \Leftrightarrow i\in\hat{\Pi}_{0,\alpha_n}(j)$ so that $\hat{\Pi}_{0,\alpha_n}$ is an exchangeable partition on the event $\{\alpha_n<\infty\}$. Since this partition does not have singleton blocks, it admits proper frequencies: $\bbP$-a.s.~$\sum_{i\geq 1}|\hat{\Pi}_{0,t}(i)|=1$ on the event $\{\alpha_n<\infty\}$. Consequently $\bbP(\alpha_n<\infty)=0$. Taking a sequence $\eta_p\downarrow 0$, we deduce that $\bbP$-almost surely $(\sum_{i\geq 1}|\hat{\Pi}_{0,t}(i)|,t>0)$ is constant equal to $1$, and therefore $\bbP$-almost surely for all rational $s$ the process $(\sum_{i\geq 1}|\hat{\Pi}_{s,s+t}(i)|,t>0)$ is constant equal to $1$.\\
We now prove the existence of asymptotic frequencies for $\hat{\Pi}_{s,s+t}$ when $s$ is not rational. Fix $i\geq 1$ and a rational value $p\in (s,s+t)\cap\bbQ$. For every $n\geq i$, we set
\begin{equation*}
\eta(n) := 1-\sum_{l=1}^{n}|\hat{\Pi}_{p,s+t}(l)|
\end{equation*}
From the property proved above $\eta(n) \rightarrow 0$ as $n\rightarrow\infty$. Let us denote by $j_1,j_2,\ldots$ the elements of the block $\hat{\Pi}_{s,p}(i)$. From the cocycle property, we have:
\begin{eqnarray}\label{EqLimInfSup}
\sum_{j_l\leq n}|\hat{\Pi}_{p,s+t}(j_l)| &\leq& \varliminf\limits_{m\rightarrow\infty} \frac{\#\big\{\hat{\Pi}_{s,s+t}(i)\cap[m]\big\}}{m}\\\nonumber
\varlimsup\limits_{m\rightarrow\infty} \frac{\#\big\{\hat{\Pi}_{s,s+t}(i)\cap[m]\big\}}{m} &\leq& \sum_{j_l\leq n}|\hat{\Pi}_{p,s+t}(j_l)| + \eta(n).
\end{eqnarray}
Letting $n$ go to infinity ensures the existence of $|\hat{\Pi}_{s,s+t}(i)|$. The same reasoning applies to $\hat{\Pi}_{s-,s+t}$ and $\hat{\Pi}_{s,s+t-}$. We now prove the regularity properties. Since $p$ is rational, we know that there exists $\epsilon_0(\omega)=\epsilon_0 > 0$ small enough so that
\begin{equation*}
\big||\hat{\Pi}_{p,s+t+\epsilon}(j)|-|\hat{\Pi}_{p,s+t}(j)|\big| \leq \frac{\eta(n)}{n},\; \forall \epsilon < \epsilon_0, \forall j \leq n
\end{equation*}
Therefore
\begin{equation*}
1-\sum_{l=1}^{n}|\hat{\Pi}_{p,s+t+\epsilon}(l)| \leq 2\eta(n),\; \forall \epsilon < \epsilon_0
\end{equation*}
Combined with Equation (\ref{EqLimInfSup}), this ensures the convergence of $|\hat{\Pi}_{s,s+t+\epsilon}(i)|$ towards $|\hat{\Pi}_{s,s+t}(i)|$ when $\epsilon$ goes to $0$. We get similarly the convergence when $t-\epsilon$ goes to $t-$. To prove the convergences when $s-\epsilon$ goes to $s-$ and $s+\epsilon$ goes to $s$, one remarks that there exists $\epsilon_0(\omega)=\epsilon_0 >0$ such that no reproduction events affecting at least two levels among $[n]$ fall in $(s-\epsilon_0,s)$ nor in $(s,s+\epsilon_0)$. Hence $\hat{\Pi}_{s-\epsilon,p}(i)\cap[n]$ and $\hat{\Pi}_{s+\epsilon,p}(i)\cap[n]$ do not vary whenever $\epsilon$ is in $(0,\epsilon_0)$. Similar arguments as above apply.\\
To end the proof, we need to check that $(\sum_{i\geq 1}|\hat{\Pi}_{s,s+t}(i)|,t>0)$ is constant equal to $1$ for all $s\in\bbR$. We have checked it for rational values $s$. Fix $s\in\bbR$ and $t>0$. Take an arbitrary $p\in(s,s+t)\cap\bbQ$. From the cocycle property and Fatou lemma, we know that for all $i\geq 1$ we have
\[ |\hat{\Pi}_{s,s+t}(i)| \geq \sum_{j\in\hat{\Pi}_{s,p}(i)}|\hat{\Pi}_{p,s+t}(j)|\]
so that
\[ \sum_{i\geq 1}|\hat{\Pi}_{s,s+t}(i)| \geq \sum_{i\geq 1}\sum_{j\in\hat{\Pi}_{s,p}(i)}|\hat{\Pi}_{p,s+t}(j)| = \sum_{j\geq 1}|\hat{\Pi}_{p,s+t}(j)|=1.\]
This concludes the proof.\cqfd

\subsection{Proof of Proposition \ref{PropRegularisation}}\label{SubsectionProofRegularisation}
The proof of the proposition relies on three lemmas.
\begin{lemma}\label{LemmaStoFlowAccumulate}
Consider the restriction $(\hat{\Pi}_{s,t},s\leq t\in\mathbb{Q})$ of the flow of partitions to its rational marginals. Then there exists an event $\Omega_{\hat{\Pi}}$ of probability $1$ on which:\begin{itemize}
\item $\forall r < s < t \in \bbQ,\; \hat{\Pi}_{r,t} = \Coag(\hat{\Pi}_{s,t},\hat{\Pi}_{r,s})$
\item For every $n\geq 1$ and every $l\geq 1$, there exists $m_0\geq 1$ and $k\geq 0$ s.t.~for every $m\geq m_0$ the partitions $\hat{\Pi}^{[n]}_{-l+\frac{i-1}{2^m}2l,-l+\frac{i}{2^m}2l}, i\in[2^m]$ have at most one non-singleton block and the number of these partitions that differ from $\tO_{[n]}$ is equal to $k$.
\end{itemize}
\end{lemma}
\noindent This lemma implies that almost surely the trajectories of $(\hat{\Pi}_{s,t},s\leq t\in\mathbb{Q})$ are deterministic flows of partitions.
\begin{proof}
First, one has
\begin{equation}\label{EquationCocycleSimult}
\mathbb{P}\Big(\hat{\Pi}_{r,t} = \Coag(\hat{\Pi}_{s,t},\hat{\Pi}_{r,s}), \forall r < s < t \in \mathbb{Q}^3\Big) = 1
\end{equation}
The second assertion is trivially satisfied by a stochastic flow of partitions $\hat{\Pi}^{\cal P}$ defined from a $\Lambda$ lookdown graph ${\cal P}$. Using Equation (\ref{EquationCocycleSimult}), the fact that the finite-dimensional distributions of $\hat{\Pi}$ and $\hat{\Pi}^{\cal P}$ are equal and the fact that the second assertion only relies on a countable set of marginals, this assertion for $\hat{\Pi}$ follows.
\end{proof}
We now define for every $s\leq t\in \mathbb{R}$ the partition $\tilde{\hat{\Pi}}_{s,t}$ on the event $\Omega_{\hat{\Pi}}$ as follows.
\begin{lemma}\label{LemmaDefModif}
On the event $\Omega_{\hat{\Pi}}$, the following random partition is well-defined.
\begin{equation}
\tilde{\hat{\Pi}}_{s,t} := \begin{cases}\hat{\Pi}_{s,t}&\mbox{ if }s,t \in \bbQ, s < t\\
\tO_{[\infty]}&\mbox{ if }s=t\\
\lim\limits_{v \downarrow t, v \in \mathbb{Q}}\hat{\Pi}_{s,v}&\mbox{ if }s \in \mathbb{Q}, t \notin \bbQ\\
\lim\limits_{r \downarrow s, r \in \mathbb{Q}}\hat{\Pi}_{r,t}&\mbox{ if }s \notin\bbQ, t \in \mathbb{Q}\\
\Coag(\tilde{\hat{\Pi}}_{q,t},\tilde{\hat{\Pi}}_{s,q})&\mbox{ for any arbitrary rational }q \in (s,t)\mbox{ if }s,t \notin\mathbb{Q}
\end{cases}
\end{equation}
Furthermore, for every $r < s < t$, $\tilde{\hat{\Pi}}_{r,t} = \Coag(\tilde{\hat{\Pi}}_{s,t},\tilde{\hat{\Pi}}_{r,s})$.
\end{lemma}
\begin{proof}
We work on the event $\Omega_{\hat{\Pi}}$ throughout this proof. Recall the cocycle property together with the left and right regularity properties satisfied by the flow restricted to its rational marginals. Fix $s \in \mathbb{Q}$ and let us prove the existence of a limit for $\hat{\Pi}_{s,v}$ when $v$ is rational and goes to a given irrational value $t \in (s,\infty)$. Fix $n\geq 1$. From Lemma \ref{LemmaStoFlowAccumulate} there exists $\epsilon=\epsilon(\omega) > 0$ such that for all rational values $p,q$ in $(t,t+\epsilon)$, $\hat{\Pi}_{p,q}^{[n]}=\tO_{[n]}$. Combined with the cocycle property on rational marginals, this ensures that $v\mapsto \hat{\Pi}_{s,v}^{[n]}$ is constant whenever $v\in(t,t+\epsilon)\cap\bbQ$. The existence of the limit follows. A similar argument shows the existence of a limit for $\hat{\Pi}_{v,t}$ when $v$ is rational and goes to an irrational value $s$ and $t$ is a given rational value.\\
Fix $r < s < t$. If all three are rational, the corresponding cocycle property holds since we are on $\Omega_{\hat{\Pi}}$. Now suppose that either $s$ is rational or both $r$ and $t$ are rational, then we stress that the corresponding cocycle property still holds. Indeed, take a limiting sequence of rational values for which the cocycle property holds and then use the continuity of the coagulation operator (see Lemma 4.2 in Bertoin~\cite{BertoinRandomFragmentation}).\\
Finally, suppose that $s,t \notin \mathbb{Q}$. To verify that our definition of $\tilde{\hat{\Pi}}_{s,t}$ makes sense, we need to show that $\Coag(\tilde{\hat{\Pi}}_{q,t},\tilde{\hat{\Pi}}_{s,q})$ does not depend on the value $q \in (s,t)\cap\mathbb{Q}$. Consider two such values $q,q' \in (s,t)\cap\mathbb{Q}$, suppose that $q < q'$ and use the associativity of the coagulation operator (see Lemma 4.2 in~\cite{BertoinRandomFragmentation}) to obtain
\begin{eqnarray*}
\Coag\big(\tilde{\hat{\Pi}}_{q',t},\tilde{\hat{\Pi}}_{s,q'}\big) \!\!\!&=&\!\!\! \Coag\big(\tilde{\hat{\Pi}}_{q',t},\Coag(\tilde{\hat{\Pi}}_{q,q'},\tilde{\hat{\Pi}}_{s,q})\big)\\\!\!\!&=&\!\!\! \Coag\big(\Coag(\tilde{\hat{\Pi}}_{q',t},\tilde{\hat{\Pi}}_{q,q'}),\tilde{\hat{\Pi}}_{s,q}\big)= \Coag\big(\tilde{\hat{\Pi}}_{q,t},\tilde{\hat{\Pi}}_{s,q}\big)
\end{eqnarray*}
Thus, the definition of $\tilde{\hat{\Pi}}_{s,t}$ does not depend on $q \in (s,t)$.\\
Finally, consider three irrational $r < s < t$, and two rational values $q,q'$ such that $q \in (r,s)$ and $q' \in (s,t)$.
\begin{eqnarray*}
\Coag\big(\tilde{\hat{\Pi}}_{s,t},\tilde{\hat{\Pi}}_{r,s}\big)&=&\Coag\big(\Coag(\tilde{\hat{\Pi}}_{q',t},\tilde{\hat{\Pi}}_{s,q'}),\Coag(\tilde{\hat{\Pi}}_{q,s},\tilde{\hat{\Pi}}_{r,q})\big)\\
&=&\Coag\big(\tilde{\hat{\Pi}}_{q',t},\Coag(\tilde{\hat{\Pi}}_{s,q'},\Coag(\tilde{\hat{\Pi}}_{q,s},\tilde{\hat{\Pi}}_{r,q}))\big) = \tilde{\hat{\Pi}}_{r,t}
\end{eqnarray*}
This concludes the proof.
\end{proof}
On the complement of $\Omega_{\hat{\Pi}}$, set any arbitrary value to $\tilde{\hat{\Pi}}_{s,t}$.
\begin{lemma}\label{JumpsLDGraph}
The collection of partitions $\tilde{\hat{\Pi}}$ is a modification of $\hat{\Pi}$, that is, for every $s \leq t$, a.s. $\tilde{\hat{\Pi}}_{s,t}=\hat{\Pi}_{s,t}$. Furthermore, for each $\omega \in \Omega_{\hat{\Pi}}$, $\tilde{\hat{\Pi}}(\omega)$ is a deterministic flow of partitions.
\end{lemma}
\begin{proof}
Fix $s\in\bbQ$. For every $t\in\bbQ_+$, $\tilde{\hat{\Pi}}_{s,s+t}=\hat{\Pi}_{s,s+t}$ on the event $\Omega_{\hat{\Pi}}$, so it holds a.s. We know that the process $(\hat{\Pi}_{s,s+t},t \geq 0)$ admits a c\`adl\`ag modification from Lemma \ref{LemmaMarkovFoP}. Almost surely for all $t\in\bbQ$ the $t$-marginals of this modification coincide with $\hat{\Pi}_{s,s+t}$ and also with $\tilde{\hat{\Pi}}_{s,s+t}$. Since the process $(\tilde{\hat{\Pi}}_{s,t},t \in [s,\infty))$ is also c\`adl\`ag, we deduce that almost surely it coincides with the modification. Thus it is itself a modification of $(\hat{\Pi}_{s,s+t},t \geq 0)$. Consequently for all $s\in\bbQ$ and all $t\geq 0$ we have almost surely $\tilde{\hat{\Pi}}_{s,s+t}=\hat{\Pi}_{s,s+t}$.\\
Now suppose $s$ irrational, take $t > 0$ and fix $n\in\mathbb{N}$. We have for all $q \in (s,s+t)\cap\mathbb{Q}$
\begin{equation*}
\mathbb{P}\Big(\tilde{\hat{\Pi}}^{[n]}_{s,s+t} = \hat{\Pi}^{[n]}_{s,s+t}\Big) \geq \mathbb{P}\Big(\tilde{\hat{\Pi}}^{[n]}_{q,s+t}=\hat{\Pi}^{[n]}_{q,s+t}\,;\,\hat{\Pi}^{[n]}_{s,s+t}=\Coag(\hat{\Pi}^{[n]}_{q,s+t},\hat{\Pi}^{[n]}_{s,q})\,;\,\hat{\Pi}^{[n]}_{s,q}=\tO_{[n]}\Big)
\end{equation*}
As $q\downarrow s$, $\mathbb{P}(\hat{\Pi}^{[n]}_{s,q}=\tO_{[n]}) \rightarrow 1$ by definition of a stochastic flow of partitions. The cocycle property of a stochastic flow of partitions together with the almost sure identity $\tilde{\hat{\Pi}}_{q,s+t}=\hat{\Pi}_{q,s+t}$ that we have already proved, ensures that the probability of the event on the r.h.s. tends to $1$ as $q \downarrow s$. Thus $\tilde{\hat{\Pi}}_{s,s+t} = \hat{\Pi}_{s,s+t}$ almost surely. Finally, when $t=0$ we know that $\hat{\Pi}_{s,s} = \tO_{[\infty]}$ almost surely by definition. Therefore, $\tilde{\hat{\Pi}}$ is a modification of $\hat{\Pi}$.\\
We need to verify that for all $\omega\in\Omega_{\hat{\Pi}}$, $\tilde{\hat{\Pi}}(\omega)$ is a deterministic flow of partitions. The cocycle property was proved in the preceding lemma. Let us show the right regularity. Fix $s\in\mathbb{R}$ and $n\in\mathbb{N}$. Recall that there exists $\epsilon=\epsilon(\omega) > 0$ such that for all rational $p<q \in (s,s+\epsilon)$, $\tilde{\hat{\Pi}}^{[n]}_{p,q} = \tO_{[n]}$. Letting $p \downarrow s$, we get $\tilde{\hat{\Pi}}^{[n]}_{s,q} = \tO_{[n]}$ for all $q \in (s,s+\epsilon)\cap\mathbb{Q}$. Similarly for all $r\in(s,s+\epsilon)$, we have $\tilde{\hat{\Pi}}^{[n]}_{q,r} = \tO_{[n]}$ as soon as $q \in (s,r)$. Using the fact that $\tilde{\hat{\Pi}}^{[n]}_{s,r} = \Coag(\tilde{\hat{\Pi}}^{[n]}_{q,r},\tilde{\hat{\Pi}}^{[n]}_{s,q})$ we get that $\tilde{\hat{\Pi}}^{[n]}_{s,r} = \tO_{[n]}$ for all $r\in(s,s+\epsilon)$. This in turn implies that $\tilde{\hat{\Pi}}^{[n]}_{s,r} \rightarrow \tilde{\hat{\Pi}}^{[n]}_{s,s}$ as $r\downarrow s$ and the right regularity is proved. The left regularity is obtained similarly.
\end{proof}

\subsection{Proof of Equation (\ref{EqBoundRate})}\label{AppendixProofBoundRate}
Consider the set $S:=\{1,\ldots,\lfloor an\rfloor+l-1\}$. Observe that
\[\sum_{i=2}^{(\lfloor an\rfloor+l+1-n)\wedge n}\binom{n}{i}\binom{\lfloor an\rfloor+l-n-1}{\lfloor an\rfloor+l+1-n-i}\]
is the number of combinations with $\lfloor an\rfloor+l+1-n$ elements among $S$ with at least $2$ elements among $\{1,\ldots,n\}$. This number is greater than the number of combinations with $\lfloor an\rfloor+l+1-n$ elements among $S$ with the constraint that $1$ and $2$ are chosen. The latter is equal to $\binom{\lfloor an\rfloor+l-3}{\lfloor an\rfloor+l-n-1}$. Consequently we can bound the l.h.s.~of (\ref{EqBoundRate}) as follows.
\begin{eqnarray*}
&&\int_{[0,1)}\nu(dx)\sum_{l=0}^{\infty}\sum_{i=2}^{(\lfloor an\rfloor+l+1-n)\wedge n}\binom{n}{i}\binom{\lfloor an\rfloor+l-n-1}{\lfloor an\rfloor+l+1-n-i}x^{\lfloor an\rfloor+l-n+1}(1-x)^{n-1}\\
&\geq& \int_{(u',1)}\nu(dx)\sum_{l=0}^{\infty}x^{\lfloor an\rfloor+l-n+1}(1-x)^{n-1}\binom{\lfloor an\rfloor+l-3}{\lfloor an\rfloor+l-n-1}
\end{eqnarray*}
The binomial factor at the second line corresponds to the number of combinations with $\lfloor an\rfloor+l-n-1$ elements among the set $\{3,\ldots,\lfloor an\rfloor+l-1\}$. Splitting this last set into $\{3,\ldots,\lfloor an\rfloor-1\}$ and $\{\lfloor an\rfloor,\ldots,\lfloor an\rfloor+l-1\}$ we get that the last quantity is equal to
\begin{eqnarray*}
&&\!\!\!\!\!\!\int_{(u',1)}\!\!\!\nu(dx)\sum_{l=0}^{\infty}x^{\lfloor an\rfloor+l-n+1}(1-x)^{n-1}\\
&&\!\!\!\times\sum_{j=\lfloor an\rfloor-n-1}^{(\lfloor an\rfloor+l-n-1)\wedge(\lfloor an\rfloor-3)}\!\!\!\binom{\lfloor an\rfloor-3}{j}\binom{l}{\lfloor an\rfloor+l-n-1-j}\\
&=&\!\!\! \int_{(u',1)}\nu(dx)x^2\sum_{j=\lfloor an\rfloor-n-1}^{\lfloor an\rfloor-3}\binom{\lfloor an\rfloor-3}{j}x^j(1-x)^{\lfloor an\rfloor-3-j}\\
&&\!\!\!\times\sum_{l=j+n+1-\lfloor an\rfloor}^\infty\binom{l}{\lfloor an\rfloor+l-n-1-j}x^{\lfloor an\rfloor-n-1-j+l}(1-x)^{n-\lfloor an\rfloor+2+j}.
\end{eqnarray*}
Finally using the change of variable $p=\lfloor an\rfloor+l-n-1-j$ and setting $k:=n+1+j-\lfloor an\rfloor$, one gets
\begin{eqnarray*}
\sum_{l=j+n+1-\lfloor an\rfloor}^\infty\!\!\!\binom{l}{\lfloor an\rfloor+l-n-1-j}x^{\lfloor an\rfloor-n-1-j+l}(1-x)^{n-\lfloor an\rfloor+2+j}\\
=\frac{(1-x)^{1+k}}{k!}\sum_{p=0}^{\infty}\frac{(p+k)!}{p!}x^p.\end{eqnarray*}
A simple induction on $k$ shows that this last quantity is equal to 1. This ends the proof.

\subsection{Calculations on regular variation}
\label{ProofRegularVarLambda}
Suppose that $\Psi$ is regularly varying at $+\infty$ with index $\alpha\in(1,2]$. Let us prove that the map $v$ defined in the proof of Theorem \ref{ThCDI} is itself regularly varying at $0+$ with index $-1/(\alpha-1)$. We have
\begin{equation*}
\frac{v(t)}{\Psi(v(t))} = \int_0^t \frac{\Psi(v(s))-v(s)\,\Psi'(v(s))}{\big(\Psi(v(s))\big)^2}\,v'(s)\,  ds\underset{t\rightarrow 0}{\sim} (\alpha-1)t
\end{equation*}
In the above calculations, we use the identity $v'(s)=-\Psi(v(s))$ and the regular variation of $\Psi$ at $\infty$ that ensures the convergence $u\Psi'(u)/\Psi(u) \rightarrow \alpha$ as $u\rightarrow\infty$ thanks to Theorem 2 in Lamperti~\cite{Lamperti58}. Therefore we have proved that
$$ \frac{t\,v'(t)}{v(t)} \underset{t\rightarrow 0}{\longrightarrow} \frac{-1}{\alpha-1} $$
This identity implies, thanks to Theorem 2 in~\cite{Lamperti58} again, that $v$ is regularly varying at $0+$ with index $-1/(\alpha-1)$.\cqfd





\ACKNO{This is part of my PhD thesis. I would like to thank my supervisors, Julien Berestycki and Amaury Lambert, for some fruitful discussions throughout this work. I am grateful to Gr\'egory Miermont for his useful remarks that helped me to improve the presentation of this paper. I would also like to express my gratitude to the anonymous referees for their very careful reading and their numerous suggestions and remarks.}


\end{document}